\documentclass{amsart}
\DeclareMathAlphabet\mathbfcal{OMS}{cmsy}{b}{n}
\usepackage{preamble}
\usepackage[T1]{fontenc}
\usepackage[OT2,T1]{fontenc}
\usepackage{comment}
\usepackage{mathtools}
\usepackage{cleveref}
\usepackage{extarrows}
\usepackage[margin=1.4in]{geometry}
\newcommand{\cdotspace}{\hspace{-1.5pt}\cdot\hspace{-1.5pt}}

\newcommand{\new}{\mathrm{new}}

\newcommand{\ord}{\mathrm{ord}}

\newcommand{\sst}{\mathrm{st}}

\newcommand{\tupH}{\textup{H}}

\newcommand{\ccF}{\mathbfcal{F}}

\DeclareSymbolFont{cyrletters}{OT2}{wncyr}{m}{n}
\DeclareMathSymbol{\Sha}{\mathalpha}{cyrletters}{"58}

\definecolor{Green}{rgb}{0.0, 0.5, 0.0}

\newcommand{\p}{\mathfrak{p}}
\numberwithin{equation}{section}

\begin{document}

\title[Rationality of Bianchi Stark--Heegner cycles]{The rationality of Stark--Heegner cycles attached to Bianchi modular forms -- The base--change scenario.}

\begin{abstract}
We study Stark--Heegner cycles attached to Bianchi modular forms, that is automorphic forms for $\mathrm{GL}(2)$ over an imaginary quadratic field $F$. The Stark--Heegner cycles are local cohomology classes in the $p$-adic Galois representation associated to the Bianchi eigenform. They are conjectured to be the restriction (at a prime $p$) of global cohomology classes in the (semistable) Bloch--Kato Selmer group defined over ring class fields of a relative quadratic extension $K/F$. In this article, we show that these conjectures hold when the Bianchi eigenform is the base-change of a classical elliptic cuspform.
\end{abstract}

\author{Guhan Venkat}
\address{Guhan Venkat\newline
Department of Mathematics\\
Ashoka University\\
Rajiv Gandhi Education City\\
Sonipat, 131029\\
India.}
\email{guhan.venkat@ashoka.edu.in}

\thanks{The author's work was partially supported by the Department of Science \& Technology's INSPIRE Faculty Grant IFA20-MA-147 and a SERB MATRICS grant -- MTR/2022/000031.}
\maketitle
\tableofcontents


\setcounter{tocdepth}{1}
\renewcommand{\baselinestretch}{0.4}
\small \tableofcontents
\renewcommand{\baselinestretch}{1.0}\normalsize


\section{Introduction} \label{sec:intro}
The arithmetic theory of automorphic $L$-functions has been the subject of mathematical research for long. The Birch \& Swinnerton--Dyer Conjecture (BSD) and its generalization, the Bloch--Kato Conjecture predict a mysterious relationship between the arithmetic of an automorphic form and the special values of its $L$-function.  For instance, if $E/\QQ$ is an elliptic curve over the field of rational numbers, then BSD predicts that the order of vanishing of the Hasse--Weil $L$-function at $s = 1$ (analytic rank) equals the Mordell--Weil rank of the group of rational points (algebraic rank), i.e. \[\mathrm{ord}_{s=1}L(E,s) = \mathrm{rank}_{\ZZ}E(\QQ).\] When the analytic rank is precisely one, the theory of Complex Multiplication via \emph{Heegner points} plays a key role in the celebrated proof of BSD by Gross--Zagier (\cite{GZ86}) and Kolyvagin (\cite{Kol88}).  In scenarios that go beyond the realms of the theory of Complex Multiplication, Darmon (\cite{Dar01}) used $p$-adic methods to construct local points in the Mordell--Weil group of the curve, known as \emph{Stark--Heegner points}. These points are conjectured to be global points and satisfy a reciprocity law under Galois automorphisms similar to the ones satisfied by Heegner points.
\paragraph*{} When $E$ is replaced by the Galois representation $\rho_{/K}$ attached to the quadratic base--change of a Bianchi modular form (i.e. a modular form over an imaginary quadratic field $F$), then the Bloch--Kato Conjecture predicts that the order of vanishing of its $L$-function at critical values equals the rank of a Selmer group attached to $\rho_{/K}$. Inspired by the ideas of Darmon and generalizing earlier works of Trifkovi\'c (\cite{Tri06}) \& Rotger--Seveso (\cite{RS12}), we construct local Selmer classes via \emph{Stark--Heegner cycles} in \cite{VW19}. We then conjecture these local classes to be global Selmer classes that satisfy a reciprocity isomorphism as in \cite{Dar01}.  The main aim of this article (Theorem~\ref{thm:maintheoremintroduction}) is to show that the global rationality conjecture of \cite{VW19} (Conjecture 6.8) holds when the Bianchi eigenform is the base--change to the imaginary quadratic field $F$ of a classical elliptic eigenform. We also record the consequences of Theorem~\ref{thm:maintheoremintroduction} towards Trifkovi\'{c}'s global rationality conjecture for Stark--Heegner points attached to (modular) elliptic curves defined over imaginary quadratic fields (\cite[Conjecture 6]{Tri06}) in \S\ref{subsec:literature} below. In particular, our results fit within Darmon's conjectural program initiated in \cite{Dar01}.
\subsection{Set--up}
Let $p$ be a rational prime and fix embeddings $\iota_{\infty} : \overline{\QQ} \rightarrow \mathbb{C}$ and $\iota_{p}: \overline{\QQ} \rightarrow \overline{\QQ}_{p}$ once and for all.  Let $F$ be an imaginary quadratic field with ring of integers $\mathcal{O}_{F}$ and discriminant $D_{F}$ such that $p$ is unramified in $F$, and we denote by $\mathfrak{p}$ to be the prime above $p$ in $\mathcal{O}_{F}$ corresponding to $\iota_{p}$. For ease of notation, we assume that the class number of $F$ is one. Let $\cF \in S_{\underline{k_{0}}+2}(U_{0}(\mathcal{N}))^{\mathrm{new}}$ be a \emph{Bianchi cuspidal newform} of (parallel) even weight $k_{0} + 2$\footnote{In \cite{VW19}, we denoted the form to have parallel weight $(k_{0},k_{0})$.} and square--free level $\cN = \p\cM \subseteq \cO_{F}$ (such that $(\p, \cM) = 1$). Let $K/F$ be a quadratic extension of $F$ of relative discriminant (resp. absolute discriminant) $\mathcal{D}_{K/F}$ (resp. $D_{K}$) relatively prime to $\cN$. Further we assume that $K$ satisfies the following Stark--Heegner hypothesis (\textbf{SH-Hyp})
\begin{itemize}
\item[•] $\p$ is inert in $K$
\item[•] All primes $\mathfrak{l} \mid \cM$ split in $K$
\end{itemize}
Under (\textbf{SH-Hyp}), the sign of the functional equation of the base--change $L$-function $L(\cF/K,s)$ is $-1$ and in particular forces the vanishing of the central critical value, i.e.
\[ L(\cF/K,k_{0}/2 + 1) = 0 \]
This allows us to force higher orders of vanishing over ring class extensions of the field $K$. Let $\mathcal{C} \subseteq \cO_{F}$ be any ideal relatively prime to $\cN\mathcal{D}_{K/F}$ and let 
\[ \cO_{\mathcal{C}} := \cO_{F} + \mathcal{C}\cO_{K} \]
be the $\cO_{F}$-order of conductor $\mathcal{C}$ in $K$. Let $H_{\mathcal{C}}/K$ be the ring class field of conductor $\mathcal{C}$ and let $G_{\mathcal{C}} := \mathrm{Gal}(H_{\mathcal{C}}/K)$ which we know by global class field theory is isomorphic to $\mathrm{Pic}(\cO_{\mathcal{C}})$. For any character $\chi:G_{\mathcal{C}} \rightarrow \mathbb{C}^{\times}$, the sign of the twisted $L$-series $L(\cF/K,\chi,s)$ is again $-1$. Further, the $L$-series admits a factorisation
\begin{equation} \label{eqn:factorisation}
L(\cF/H_{\mathcal{C}},s) = \prod\limits_{\chi \in G_{\mathcal{C}}^{\vee}}L(\cF/K,\chi,s)
\end{equation}
and it follows that 
\[ \mathrm{ord}_{s = k_{0}/2 + 1}L(\cF/H_{\mathcal{C}},s) \geq h(\mathcal{O}_{\mathcal{C}}) := |G_{\mathcal{C}}| \]
We denote by $V_{p}(\cF)$ to be the two dimensional $G_{F} \defeq \mathrm{Gal}(\overline{\QQ}/F)$--representation attached to $\cF$, taking values in a finite extension $L/\QQ_{p}$. The Bloch--Kato conjecture then predicts the existence of a family of non-trivial cohomology classes
\[ \{s_{\mathcal{C}} \in \mathrm{Sel}_{\sst}(H_{\mathcal{C}}, V_{p}(\cF)(k_{0}/2 + 1))\}\]
over towers of class fields $H_{\mathcal{C}}$ for $\mathcal{C}$ relatively prime to $\cN\mathcal{D}_{K/F}$. Following the ideas of \cite{Dar01} and \cite{RS12}, in \cite{VW19} we proposed conjectural candidates for such a family of cohomology classes, viz. Stark--Heegner cycles which can be regarded as \emph{local cohomology classes} 
\[s_{\Psi} \in \mathrm{H}^{1}_{\mathrm{st}}(L, V_{p}(\cF)(k_{0}/2+1))\] 
assoicated to \emph{optimal embeddings} of $\cO_{F}[1/{\p}]$-orders. See \S\ref{subsec:starkheegnercycles} below where we briefly recall the construction of Stark--Heegner cycles. The aim of this article (See Theorem~\ref{thm:maintheoremintroduction} below) is to give some evidence for Conjecture 6.8 of \cite{VW19} by showing that the Stark--Heegner classes are in fact (the restriction at $\p$ of) global Selmer classes in the base--change scenario. 
\paragraph*{} Before stating Theorem~\ref{thm:maintheoremintroduction} precisely, we introduce some notation. Let $\mathcal{R}$ be the Eichler $\mathcal{O}_{F}[1/\mathfrak{p}]$-order in $\textup{M}_{2}(\mathcal{O}_{F}[1/\mathfrak{p}])$ of $2\times 2$ matrices that are upper triangular modulo $\mathcal{M}$ and let $\Gamma := \mathcal{R}_{1}^{\times}$ be the set of invertible matrices of $\mathcal{R}$ of determinant $1$. Let $\Psi : \mathcal{O}_{\mathcal{C}} \hookrightarrow \mathcal{R}$ be an optimal embedding of $\mathcal{O}_{\mathcal{C}}$ -- an $\mathcal{O}_{F}[1/\mathfrak{p}]$-order in $K$ of conductor $\mathcal{C}$ prime to $\mathcal{N}\mathcal{D}_{K}$. The \emph{Stark--Heegner} (or \emph{Darmon}) cycle is then a homology class
\[ 
\mathrm{D}_{\Psi} \in (\Delta_{0} \otimes \textup{Div}(\mathcal{H}_{\mathfrak{p}}^{\textup{ur}}) \otimes V_{k_{0},k_{0}})_{\Gamma} 
\]
where $\Delta_{0} := \textup{Div}^{0}(\mathbb{P}^{1}(F))$, $\textup{Div}(\mathcal{H}_{\mathfrak{p}}^{\textup{ur}})$ denotes the subgroup of divisors supported on the unramified $\p$-adic upper half plane $\cH_{\p}^{\mathrm{ur}} \defeq \left(\mathbb{P}^{1}(\mathbb{Q}_{p}^{\textup{ur}})\setminus\mathbb{P}^{1}(F_{\mathfrak{p}})\right)^{{\textup{Gal}(\mathbb{Q}_{p}^{\textup{ur}}/L^{0})}}$ where $L^{0} := L \cap \mathbb{Q}_{p}^{\textup{ur}}$, and $V_{k_{0},k_{0}} := V_{k_{0}} \otimes V_{k_{0}}$ where $V_{k_{0}}$ is the ring of homogenous polynomials of degree $k_{0}$ in two variables with coefficients in $L$. This space should be regarded as an explicit substitute for the local Chow group. 
\paragraph*{} In \cite{VW19}, we developed a `modular symbol theoretic' $\p$-adic integration theory following \cite{Sev12} (See \S\ref{subsec:padicintegrationrecap} below). Our $\p$-adic integration theory can be regarded as a morphism
\[ \Phi_{\cF}^{\sigma} : (\Delta_{0} \otimes \textup{Div}^{0}(\mathcal{H}_{\mathfrak{p}}^{\textup{ur}}) \otimes V_{k_{0},k_{0}})_{\Gamma} \rightarrow \mathbf{D}_{\cF, L}^{\sigma}/\mathrm{Fil}^{\frac{k_{0}+2}{2}}(\mathbf{D}_{\cF, L}^{\sigma}) \]
for each embedding $\sigma : F_{\p} \hookrightarrow L$. Here $\mathbf{D}_{\cF, L}^{\sigma}$ is a two-dimensional filtered $L$-vector space built from the space of overconvergent Bianchi modular symbols over the Bruhat--Tits tree $\mathcal{T}_{\p}$ for $\mathrm{GL}_{2}/F_{\p}$ associated to $\cF$, which we denote by $\mathbf{MS}_{\Gamma}(L)_{(\cF)}$. In \cite{VW19}, we showed that $\mathbf{D}_{\cF} \defeq \bigoplus\limits_{\sigma} \mathbf{D}_{\cF, L}^{\sigma} \in \mathrm{MF}(\varphi, N, F_{\p}, L)$ -- the category of filtered Frobenius modules over $F_{\p}$ with coefficients in $L$. The $\p$-adic Abel--Jacobi map that we construct is a lift of $\Phi_{\cF}^{\sigma}$ (See Theorem~\ref{thm:padicabeljacobiimageofdarmoncycles} below)
\[ \Phi_{\sigma}^{\mathrm{AJ}} : (\Delta_{0} \otimes \textup{Div}(\mathcal{H}_{\mathfrak{p}}^{\textup{ur}}) \otimes V_{k_{0},k_{0}})_{\Gamma} \rightarrow \mathbf{D}_{\cF, L}^{\sigma}/\mathrm{Fil}^{\frac{k_{0}+2}{2}}(\mathbf{D}_{\cF, L}^{\sigma}) \]
removing the condition on the degree of divisors on $\cH_{\p}^{\mathrm{ur}}$. In \S\ref{sec:reviewofpadicAJ} below, we show that the $\left(\mathbf{D}_{\cF, L}^{\sigma}/\mathrm{Fil}^{\frac{k_{0}+2}{2}}\right)$-valued $\p$-adic integration theory can be realized as an $\mathbf{MS}_{\Gamma}(L)_{(\cF)}^{\vee}$-valued integration theory via
\[ \mathrm{log}\:\Phi^{\mathrm{AJ}}_{\sigma}: (\Delta_{0} \otimes \textup{Div}(\mathcal{H}_{\mathfrak{p}}^{\textup{ur}}) \otimes V_{k_{0},k_{0}})_{\Gamma} \rightarrow \mathbf{D}_{\cF, L}^{\sigma}/\mathrm{Fil}^{\frac{k_{0}+2}{2}}(\mathbf{D}_{\cF, L}^{\sigma}) \xrightarrow{\cong} \mathbf{MS}_{\Gamma}(L)_{(\cF)}^{\vee} \]
We also define 
\[\mathrm{log}\:\Phi^{\mathrm{AJ}} \defeq \sum\limits_{\sigma} \mathrm{log}\:\Phi^{\mathrm{AJ}}_{\sigma}\] 
and show that there exists a $\left(\mathbf{D}_{\cF, L}^{\sigma}/\mathrm{Fil}^{\frac{k_{0}+2}{2}}\right)$-valued integration theory, viz. $\Phi^{\mathrm{AJ}}$, equivalent to $\mathrm{log}\;\Phi^{\mathrm{AJ}}$ for any choice of $\sigma : F_{\p} \hookrightarrow L$ (See Remark~\ref{rem:choiceofembedding} in particular). 
\paragraph*{} Let $\mathbb{D}_{\cF} \defeq \mathbb{D}_{\mathrm{st}}(V_{p}(\cF)) \in MF(\varphi, N, F_{\p}, L)$ be Fontaine's semistable Dieudonn\'{e} module attached to the local Galois representation $V_{p}(\cF)|_{G_{F_{\p}}}$. The \emph{trivial zero conjecture} (See \cite[Conjecture 4.2]{VW19}) would then afford a $(\varphi, N)$-module (over $F_{\p}$ with coefficients in $L$) isomorphism
\[ \mathbf{D}_{\cF} \overset{\varphi}{\cong} \mathbb{D}_{\cF} \]
which induces an identification of the tangent spaces
\[ \frac{\mathbf{D}_{\cF, L}^{\sigma}}{\mathrm{Fil}^{\frac{k_{0}+2}{2}}(\mathbf{D}_{\cF, L}^{\sigma})} \overset{\varphi}{\cong} \frac{\mathbb{D}_{\cF, L}^{\sigma}}{\mathrm{Fil}^{\frac{k_{0}+2}{2}}(\mathbb{D}_{\cF, L}^{\sigma})} \overset{\mathrm{exp}_{\mathrm{BK}}}{\cong} \mathrm{H}^{1}_{\mathrm{st}}(L, V_{p}(\cF)(k_{0}/2+1)) \]
where $\mathbb{D}_{\cF, L}^{\sigma} \defeq \mathbb{D}_{\cF} \otimes_{F_{\p}\otimes L, \sigma} L$ and the last isomorphism is given by the Bloch--Kato exponential. 

To a character $\chi : \textup{Gal}(H_{\mathcal{C}}/K) \rightarrow \mathbb{C}^{\times}$, we define a \emph{$\chi$-twisted Stark--Heegner cycle} (See Definition~\ref{defn:stark-heegnercycle2} below) 
\[ 
\mathrm{D}_{\chi} \in (\Delta_{0} \otimes \textup{Div}(\mathcal{H}_{\mathfrak{p}}^{\textup{ur}}) \otimes V_{k_{0},k_{0}})_{\Gamma} \otimes \chi,
 \]
 where $(-) \otimes \chi$ denotes suitable scalar extension by $\chi$. Let $H_{\chi}$ denote the abelian subextension of $H_{\cC}$ cut out by the character $\chi$. Then similar to \cite[Conjecture 6.8]{VW19}, we may formulate
 \begin{conjecture} \label{conj:starkheegnerconjectureintro} 
 There exists a global Selmer class $\mathcal{S}_{\chi} \in \mathrm{Sel}_{\mathrm{st}}(H_{\chi}, V_{p}(\cF)(k_{0}/2+1))^{\chi}$ such that
 \[ \mathrm{exp}_{\mathrm{BK}} \circ \varphi \left( \Phi^{\mathrm{AJ}}(\mathrm{D}_{\chi}) \right) = \mathrm{res}_{\p}\left(\mathcal{S}_{\chi} \right) \]
 where $\left( - \right)^{\chi}$ denotes the $\chi$-isotypic component.
 \end{conjecture}
 \begin{remark} \label{rem:conjectureremarkintro}
The formulation of Conjecture~\ref{conj:starkheegnerconjectureintro} above is slightly different from that in \cite[Conjecture 6.8]{VW19} which asserts the global rationality of $\Phi^{\mathrm{AJ}}_{\sigma}(\mathrm{D}_{\chi})$ over each embedding $\sigma : F_{\p} \hookrightarrow L$. As explained in \S\ref{sec:reviewofpadicAJ} below, $\Phi^{\mathrm{AJ}}(\mathrm{D}_{\chi})$ should be considered as the sum of $\Phi^{\mathrm{AJ}}_{\sigma}(\mathrm{D}_{\chi})$ over all possible embeddings $\sigma : F_{\p} \hookrightarrow L$.
 \end{remark}
\begin{remark} \label{rem:maintheoremintroremark}
When $\chi = \chi_{\mathrm{triv}} : \mathrm{Gal}(H_{K}/K) \rightarrow \mathbb{C}^{\times}$ -- the trivial character thought of as an unramified character associated to the maximal order $\cO_{\cC} = \cO_{K}$, then we denote the $\chi$--twisted Stark--Heegner cycle $\mathrm{D}_{\chi}$ simply by $\mathrm{D}_{\mathbbm{1}}$. Here $H_{K}$ is the Hilbert class field of $K$ and $H_{\chi} = K$.
\end{remark}
\subsection{Main results} \label{subsec:mainresults}
Now suppose that $\cF\in S_{\underline{k_{0}}+2}(U_{0}(\cN))^{\mathrm{new}}$ is the base-change of an elliptic cuspidal newform $f \in S_{k_{0}+2}(\Gamma_{0}(N))^{\mathrm{new}}$, where $\cN = N\cO_{F}$. Note that this is always the case when $(N, D_{F}) = 1$ which we assume to hold. Since the base--change $\cF$ is cuspidal, we know that $f$ doesn't have CM by the imaginary quadratic field $F$. By Atkin--Lehner--Li theory, we know that $a_{p}(f) = \omega_{p}p^{k_{0}/2}$ where $-\omega_{p}$ is the eigenvalue of the Atkin-Lehner involution $W_{p}$. We shall assume that $\omega_{p} = 1$ throughout (i.e. $f$ has split multiplicative reduction at $p$). Further assume that the level $N$ admits a factorization of relatively prime integers
\begin{equation} \label{eqn:factorizationofN}
N = pM = pN^{+}N^{-}
\end{equation}
such that the following Heegner hypothesis holds (\textbf{Heeg--Hyp}) :-
\begin{itemize}
\item[•] $p$ is inert in $F$
\item[•] All primes dividing $N^{+}$ (resp. $N^{-}$) split (resp. are inert) in F
\item[•] $N^{-}$ is the square--free product of an odd number of primes.
\end{itemize} 
Let $V_{p}(f)$ denote Deligne's two dimensional $p$-adic $G_{\QQ}$--representation attached to the newform $f$. Note that $V_{p}(\cF) = V_{p}(f)|_{G_{F}}$ as $p$-adic $G_{F}$-representations. Let $\omega_{\cM}$ be the eigenvalue of the Atkin--Lehner involution $W_{\cM}$ acting on $S_{\underline{k_{0}}+2}(U_{0}(\cN))^{\mathrm{new}}$. The main result of this article is to shed some evidence towards Conjecture~\ref{conj:starkheegnerconjectureintro} formulated above.
\begin{theorem} \label{thm:maintheoremintroduction}
With notation as above, suppose that $\omega_{\cM} = (-1)^{\frac{k_{0}+2}{2}}$. Then there exists a global Selmer class
\[ \mathcal{S}_{K} \in \mathrm{Sel}_{\mathrm{st}}(K, V_{p}(\cF)(k_{0}/2+1)) \]
such that
\[ \mathrm{exp}_{\mathrm{BK}}\circ\varphi\left(\Phi^{\mathrm{AJ}}(\mathrm{D}_{\mathbbm{1}})\right) = \mathrm{res}_{\p}(\mathcal{S}_{K}) \]
\end{theorem}
\paragraph*{} The strategy of our proof is inspired from the ideas developed in \cite{BD09}, \cite{LV14}, \cite{LMH20}, \cite{Sev12} and \cite{GSS16} (where similar global rationality results of Stark--Heegner points/cycles have been established)  viz. via comparing $p$-adic Gross--Zagier formulas, which we briefly describe. Let $\cW_{F}(L)$ be the Bianchi weight space introduced in \S\ref{sec:p-adicfamilies} and $\cW_{F,\mathrm{par}}(L)$ be the parallel weight line defined as the image of $\cW_{\QQ}$ in $\cW_{F}$. We may regard the set of integers $\ZZ$ as a subset of $\cW_{F, \mathrm{par}}$ via the characters $\lambda_{k}$, for $k \in \ZZ$, given by $\lambda_{k}(z) \defeq [\mathrm{N}_{F/\QQ}(z)]^{k}$. Let $U \subseteq \cW_{F, \mathrm{par}} \subseteq \cW_{F}(L)$ be a slope-adapted affinoid centred around the point $\lambda_{k_{0}}$. In \S\ref{subsec:p-adicL-functionsoverK}, we construct a base--change $p$-adic $L$-function 
\[ L_{p}(\mathbfcal{F}/K, \chi, - ) : U \rightarrow \mathbb{C}_{p} \]
which interpolates the central critical $L$-values $L^{\mathrm{alg}}(\cF_{k}/K, \chi, k/2 + 1)$ of the classical specialisations of the Coleman family $\ccF$ (See Theorem~\ref{thm:popageneralization} and Theorem~\ref{thm:padicinterpolationoverK} below). When $K/F$ is a relative quadratic extension that satisfies (\textbf{SH-Hyp}) mentioned above, we show a $p$-adic Gross--Zagier type formula relating the second derivative of this $p$-adic $L$-function to the $p$-adic Abel--Jacobi image of the Stark--Heegner cycles described above. More precisely, we show that
\begin{theorem} \label{thm:padicGZformula1intro}
\[ \frac{d^{2}}{d\lambda_{\kappa}^{2}}[L_{p}(\mathbfcal{F}/K, \lambda_{\kappa})]_{\lambda_\kappa = \lambda_{k_{0}}} =
\begin{cases}
2\left(\mathrm{N}_{F/\QQ}(\cD_{K/F})\right)^{\frac{k_{0}}{2}}\left(\mathrm{log}\:\Phi^{\mathrm{AJ}}(\mathrm{D}_{\mathbbm{1}})(\Phi_{\cF}^{\mathrm{har}})\right)^{2} & \text{if }\omega_{\cM} = (-1)^{\frac{k_{0}+2}{2}}  \\
0 & \text{if }\omega_{\cM} = (-1)^{\frac{k_{0}}{2}}
\end{cases}
\]
where $\Phi_{\cF}^{\mathrm{har}} \in \mathbf{MS}_{\Gamma}(L)_{(\cF)}$ is the harmonic modular symbol attached to $\cF$ in \S\ref{subsec:harmonicmodularsymbol}.
\end{theorem}
Let $\epsilon_{K/F}$ be the quadratic id\`{e}le class character of $F$ that cuts out the relative quadratic extension $K/F$. In \S\ref{subsec:factorization}, we show a $p$-adic Artin formalism for the base--change $p$-adic $L$-function described above. 
\begin{theorem} \label{thm:factorizationofpadicLfnsintro}
For all $\lambda_{\kappa} \in U$, we have a factorization of $p$-adic $L$-functions, 
\[ (D_{K})^{\lambda_{\kappa}/2}L_{p}(\mathbfcal{F}/K, \lambda_{\kappa}) = \eta L_{p}(\mathbfcal{F}, \lambda_{\kappa})L_{p}(\mathbfcal{F}, \epsilon_{K/F}, \lambda_{\kappa})\]
for some constant $\eta \in \overline{\QQ}^{\times}$.
\end{theorem}
Here $L_{p}(\mathbfcal{F}, \lambda_{\kappa})$ and $L_{p}(\mathbfcal{F}, \epsilon_{K/F}, \lambda_{\kappa})$ are the two variable base--change Bianchi $p$-adic $L$-functions constructed by Seveso in \cite{Sev12} and recalled in Section~\ref{subsec:bianchip-adicLfns} below.
Recall the factorization $N = pM = pN^{+}N^{-}$. Let $\cB$ be the indefinite quaternion algebra ramified at the primes dividing $pN^{-}$.  Let $X \defeq X_{N^{+},pN^{-}}$ be the Shimura curve associated to $\cB$ and an Eichler order of level $N^{+}$ in $\cB$. Let $\cM_{k_{0}}/\QQ$ be the Chow motive attached to the space of weight $k_{0}+2$ modular forms on the Shimura curve $X$ and let $\mathrm{CH}^{k_{0}/2+1}(\cM_{k_{0}} \otimes F)$ denote the Chow group of co-dimension $k_{0}/2+1$ cycles of $\cM_{k_{0}}$ base-changed to $F$. Let $M_{k_{0}+2}(\Gamma', L)$ denote the space of rigid analytic modular forms, over $L$, on the Mumford curve $X_{\Gamma'}$ associated to the arithmetic subgroup $\Gamma'$  defined in \S\ref{subsec:Heegnercycles}.  The Mumford curve $X_{\Gamma'}$ can be identified with the rigid analytification -- $X^{\mathrm{an}}$ of the Shimura curve $X$ via the Cerednik-Drinfeld Theorem of $p$-adic uniformization.  Then the $p$-adic \'{e}tale Abel--Jacobi map described in \S\ref{subsec:Heegnercycles} can be regarded as 
\begin{equation} \label{abeljacobiintroduction}
\mathrm{log\;cl}_{f,L}: \mathrm{CH}^{k_{0}/2+1}\left( \cM_{k_{0}} \otimes F \right)  
  \rightarrow M_{k_{0}+2}(X, L)_{(f^{\mathrm{JL}})}^{\vee} \rightarrow M_{k_{0}+2}(\Gamma', L)_{(f^{\mathrm{rig}})}^{\vee}.
\end{equation}
where $f^{\mathrm{JL}} \in M_{k_{0}+2}(X,  L)$ (resp.  $f^{\mathrm{rig}} \in M_{k_{0}}(\Gamma',  L)$) is the modular form on the Shimura curve $X$ (resp.  on the Mumford curve $X_{\Gamma'}$) associated to $f \in S_{k_{0}+2}(\Gamma_{0}(N))^{\mathrm{new}}$ via the Jacquet--Langlands correspondence (resp. via the Cerednik--Drinfeld $p$-adic uniformization Theorem).  Let $\mathcal{Y} \in \mathrm{CH}^{k_{0}/2+1}\left( \cM_{k_{0}} \otimes F \right)$ be the Heegner cycle constructed in \cite[Section 8]{IS03} using the theory of Complex Multiplication.  We then have the following result of M. Seveso (\cite{Sev14}) on the $p$-adic Abel--Jacobi image of the Heegner cycle :-

\begin{theorem} \label{thm:HeegnercyclespadicAJintro}
\[ \frac{d^{2}}{d\lambda_{\kappa}^{2}}\left[ L_{p}(\mathbfcal{F}, \lambda_\kappa) \right]_{\lambda_\kappa = \lambda_{k_{0}}} = \frac{d^{2}}{d\lambda_{\kappa}^{2}}\left[ L_{p}(\mathbf{f}/F, \lambda_{\kappa}) \right]_{\lambda_{\kappa} = \lambda_{k_{0}}} = 2 \mathrm{log\;cl}_{f,L}(\mathcal{Y})(f^{\mathrm{rig}})^{2} \]
\end{theorem}

Under the assumption that $L(F, \epsilon_{K/F}, k_{0}/2+1) \neq 0$, it can be shown that the $p$-adic $L$-value $L_{p}(\mathbfcal{F}, \epsilon_{K/F}, \lambda_{k_{0}}) \neq 0$.  On comparing the $p$-adic Abel--Jacobi image of Stark--Heegner cycles with that of classical Heegner cycles via the $p$-adic Artin formalism described above, we show in \S\ref{sec:mainresult}  that (See Theorem~\ref{thm:comparingAJimages}) :-

\begin{theorem} \label{thm:comparingAJimagesintro}
Suppose that $\omega_{\cM} = (-1)^{\frac{k_{0}+2}{2}}$. Then there exists $\mathcal{Y} \in \mathrm{CH}^{k_{0}/2+1}\left( \cM_{k_{0}} \otimes F \right) \subset \mathrm{CH}^{k_{0}/2+1}\left( \cM_{k_{0}} \otimes K \right)$ and $s_{\cF} \in \QQ(\cF)^{\times}$ such that
\[ \mathrm{log}\:\Phi^{\mathrm{AJ}}(\mathrm{D}_{\mathbbm{1}})(\Phi_{\cF}^{\mathrm{har}}) = s_{\cF}\cdot \mathrm{log\;cl}_{f,L}(\mathcal{Y})(f^{\mathrm{rig}})\]
\end{theorem}
from which Theorem~\ref{thm:maintheoremintroduction} above follows.

\begin{remark} \label{rem:assumption}
It can be shown unconditionally that 
\[ \left(\Phi^{\mathrm{AJ}}(\mathrm{D}_{\mathbbm{1}})(\Phi_{\cF}^{\mathrm{har}})\right)^{2} = S_{\cF}\cdot \left(\mathrm{log\;cl}_{f,L}(\mathcal{Y})(f^{\mathrm{rig}})\right)^{2}\]
for some constant $S_{\cF} \in \QQ(\cF)^{\times}$. We wish to remark that Theorem~\ref{thm:comparingAJimagesintro} above is conditional on the fact that $S_{\cF} \in \left(\QQ(\cF)^{\times}\right)^{2}$ is indeed a square, which is consistent with the Birch and Swinnerton--Dyer conjecture. See Assumption~\ref{ass:sFisasquare} below.
\end{remark}


\subsection{Comparison to relevant literature} \label{subsec:literature} We have already mentioned that the construction of Stark--Heegner cycles for Bianchi modular forms in \cite{VW19} should be viewed as a higher weight generalization of Trifkovi\'{c}'s construction of Stark--Heegner points for (modular) elliptic curves defined over imaginary quadratic fields (\cite{Tri06}). In particular, Conjecture~\ref{conj:starkheegnerconjectureintro} formulated above is a direct generalization of \cite[Conjecture 6]{Tri06}. 
\paragraph*{} Assume in this section that $f \defeq f_{E} \in S_{2}(\Gamma_{0}(N))^{\mathrm{new}}$ is the weight two newform associated to an elliptic curve $E/\QQ$ of conductor $N$ by modularity and let $\cF_{E} \in S_{2}(U_{0}(\cN))^{\mathrm{new}}$ (i.e. $k_{0} = 0$) be the Bianchi modular form that corresponds to the quadratic base--change $E_{/F}$. Let $V_{p}(E) \defeq T_{p}(E) \otimes_{\ZZ_{p}} \QQ_{p}$ denote the $p$-adic Galois representation attached to $E$. Then, the Kummer map gives us the following exact sequence
\begin{equation} \label{eqn:kummermap}
0 \rightarrow E(K) \otimes_{\ZZ} \QQ_{p} \xrightarrow{\kappa} \mathrm{Sel}_{\mathrm{st}}(K, V_{p}(E)) \rightarrow \Sha(E/K)[p^{\infty}] \rightarrow 0 
\end{equation}
where $\Sha(E/K)[p^{\infty}]$ denotes the $p$-primary part of the Tate--Shafarevich group. If we assume that $\Sha(E/K)[p^{\infty}]$ is finite, which is conjectured to be always true by the Birch and Swinnerton--Dyer conjecture, then $\kappa$ is an isomorphism and $\mathrm{dim}\left(\mathrm{Sel}_{\mathrm{st}}(K, V_{p}(E))\right) = \mathrm{rank}(E(K))$. 
\paragraph*{} Let $P \in E(K_{\p})$ be the Stark--Heegner point constructed by Trifkovic in \cite{Tri06}. Then Theorem~\ref{thm:maintheoremintroduction} above implies 
\begin{theorem} \label{thm:trifkovic}
Under the hypothesis that $\Sha(E/K)[p^{\infty}]$, the Stark--Heegner point $P \in E(K_{\p})$ is a global $K$--rational point in $E(K)$. 
\end{theorem}
\begin{proof}
Firstly, note that when $\cF_{E}$ corresponds to the quadratic base--change of $f = f_{E}$, then we may take our coefficient field $L/\QQ_{p}$ to be just $K_{\p}$. In particular, the local cohomology classes constructed in \cite{VW19} may be regarded as classes $\mathfrak{s}_{\chi} \in \mathrm{H}^{1}_{\mathrm{st}}(K_{\p}(\chi), V_{p}(E))$. In particular, we have a commutative diagram for the trivial character $\chi = \chi_{\mathrm{triv}}$
\begin{equation} \label{eqn:trifkovic}
\begin{tikzcd} 
E(K) \otimes \QQ_{p}  \arrow[r, "\kappa"] \arrow[d, hook]
& \mathrm{Sel}_{\mathrm{st}}(K, V_{p}(E)) \arrow[d, "\mathrm{res}_{\p}"] \\
E(K_{\p}) \otimes \QQ_{p}  \arrow[r, "\kappa_{\p}"] & \mathrm{H}^{1}_{\mathrm{st}}(K_{\p}, V_{p}(E))
\end{tikzcd}
\end{equation}
It is well known that the local Kummer map $\kappa_{\p}$ is an isomorphism (See \cite[Example 3.10.1]{BK07}) and maps the Stark--Heegner point $J \in E(K_{\p})$ to the local Selmer class $\mathfrak{s}_{\mathbbm{1}} \in \mathrm{H}^{1}_{\mathrm{st}}(K_{\p}, V_{p}(E))$ of \cite{VW19}. Theorem~\ref{thm:maintheoremintroduction} above shows the existence of a global Selmer class $\mathcal{S}_{K} \in \mathrm{Sel}_{\mathrm{st}}(K, V_{p}(E))$ such that 
\[ \mathrm{res}_{\p}(\mathcal{S}_{K}) = \mathfrak{s}_{\mathbbm{1}} \]
Further, if we assume that $\Sha(E/K)[p^{\infty}]$ is a finite group then $\kappa$ is an isomorphism and the commutativity of (\ref{eqn:trifkovic}) above shows that there exists a global $K$--rational point $\mathbf{P} \in E(K)$ that is mapped to the Stark--Heegner point $P \in E(K_{\p})$ under the natural inclusion $E(K) \hookrightarrow E(K_{\p})$.
\end{proof}
\begin{remark} \label{rem:trifkovic}
Theorem~\ref{thm:trifkovic} in particular confirms Trifkovi\'{c}'s Conjecture (\cite[Conjecture 6]{Tri06}) on the rationality of Stark--Heegner points for modular elliptic curves $E/F$ that are base--change from $\QQ$.
\end{remark}
\subsection*{Acknowledgements} We thank Henri Darmon, Lennart Gehrmann, Matteo Longo, Kimball Martin, Chung--Pang Mok, Alexandru Popa, Luis Santiago Palacios and Chris Williams for their invaluable comments.  We are extremely grateful to Santiago Molina for patiently explaining his work on higher Waldspurger formula to us. We also thank the anonymous referees for their invaluable comments on an earlier draft of this article.


\section{Review of Stark--Heegner cycles attached to Bianchi modular forms} \label{sec:stark-heegnercycles}
In this section, we briefly review the construction of Stark--Heegner cycles attached to Bianchi modular forms from \cite{VW19}.
\subsection{Bianchi modular forms} \label{subsec:bianchimodularforms} 
Bianchi modular forms are adelic automorphic forms for $\GLt$ over the imaginary quadratic field $F$. 
We recall here some basic properties of Bianchi modular forms. Let $U$ be any open compact subgroup of $\GLt(\A_F^f)$, and for any $k \ge 0$, there exists a finite-dimensional $\C$-vector space $S_{\underline{k}+2}(U)$ of \emph{Bianchi cusp forms} of (parallel) weight $k + 2$ and level $U$, which are functions
\[ \widehat{\cF} : \GLt(F) \backslash \GLt(\A_F) / U \longrightarrow V_{2k+2}(\C) \]
that transform appropriately under the subgroup $\C^\times \cdotspace \ \SUt(\C)$, and also satisfy suitable harmonicity and growth conditions. We will be chiefly interested in the case where
\[
	U = U_0(\cN) = \big\{\smallmatrd{a}{b}{c}{d} \in \GLt(\widehat{\roi}_F) : c \equiv 0 \newmod{\cN}\big\},
\]
where $\cN = \p\cM \subset \roi_F$ is square--free and $\p \nmid \cM$ as before.

Bianchi modular forms admit an analogue of $q$-expansions (cf.~\cite[\S1.2]{Wil17}), giving rise to a system of Fourier--Whittaker coefficients $c(I,\widehat{\cF})$ indexed by the ideals $I \subset \mathscr{D}^{-1}$ (where $\mathscr{D}$ is the different of $F/\Q$). These can be described as the eigenvalues of Hecke operators. In fact, one can define a family of (commuting) Hecke operators indexed by ideals $\m \subset \roi_F$, defined via double coset operators. When $\widehat{\cF}$ is a normalised Hecke eigenform (i.e. $c(1,\widehat{\cF}) = 1$), then the eigenvalue $\lambda_{\m}$ of the $\m$-th Hecke operator on $\widehat{\cF}$ is equal to $c(\m\mathcal{D}^{-1},\widehat{\cF})$ (see \cite[Cor.\ 6.2]{Hid94}). 

For $M$ any module equipped with an action of the Hecke operators, and $\widehat{\cF}$ a cuspidal Bianchi eigenform, we denote by $M_{(\widehat{\cF})}$ for the $\widehat{\cF}$-isotypic part of $M$. This is the generalised eigenspace where the Hecke operators act with the same eigenvalues as on $\widehat{\cF}$.

\subsection{Bianchi modular symbols} \label{subsec:bianchimodularsymbols} For an integer $k \geq 0$, and a ring $R$, we define $V_k(R)$ to be the space of homogeneous polynomials in two variables of degree $k$ over $R$. We define a left $\GLt(R)$ action via
\begin{equation} \label{eqn:leftGL2action}
	\left(\smallmatrd{a}{b}{c}{d}\cdotspace P\right)(x,y) \defeq P\left(by+dx, ay+cx\right).
\end{equation}
We let $V_{k,k}(R) \defeq V_{k}(R)\otimes_{R} V_{k}(R)$ which carries a left action of $\GLt(R)^{2}$, acting on each component via (\ref{eqn:leftGL2action}). In particular, if $L$ is a large enough field containing both the embeddings $\sigma : F \hookrightarrow \overline{\QQ}$, then we get a left action of $\GLt(F)$ on $V_{k,k}(L)$ acting on the first component via one embedding and on the other via its conjugate. In particular, we may think of $V_{k,k}(L)$ as either the space of homogeneous polynomials of degree $k$ in two variables $x$ and $y$ and homogeneous of degree $k$ in two further variables $\overline{x}$ and $\overline{y}$ or as the space of polynomials of degree less than or equal to $k$ in both the variables $x$ and $\overline{x}$. 

Let $\Delta_{0}\defeq \mathrm{Div}^{0}(\mathbb{P}^{1}(F))$ denote the space of degree zero divisors supported on the cusps $\mathbb{P}^{1}(F)$ of the hyperbolic $3$-space $\mathcal{H}_{3}$. Note that $\GLt(F)$ acts on $\Delta_{0}$ via M\"{o}bius transformations $\delta \mapsto (a\delta+b)/(c\delta+d)$. For $\Gamma \subset \GLt(F)$ any subgroup, and $V$ a right $\Gamma$-module, we set $\Delta(V)\defeq \mathrm{Hom}(\Delta_{0}, V)$. Further, we equip this space with a $\Gamma$ action by
\[ (\gamma.\phi)(D) \defeq \phi(\gamma.D)|\gamma \]
The space of \emph{$V$-valued modular symbols} for $\Gamma$ is then defined as the $\Gamma$-invariants
\[ \mathrm{Symb}_{\Gamma}(V) \defeq \mathrm{H}^{0}(\Gamma, \Delta(V)) \]
\begin{definition} \label{defn:bianchimdoularsymbols}
The space of \emph{Bianchi modular symbols} of level $\Gamma_{0}(\mathcal{N}) \defeq U_{0}(\mathcal{N}) \cap \mathrm{SL}_{2}(\cO_{F})$ and \textit{parallel weight} $k_{0}+2$ is defined to be the space $\mathrm{Symb}_{\Gamma_{0}(\mathcal{N})}(V_{k_{0},k_{0}}(\mathbb{C}_{p})^{\vee})$ where $V_{{k_{0},k_{0}}}(\mathbb{C}_{p})^{\vee}$ is the $\mathbb{C}_{p}$-dual of $V_{k_{0},k_{0}}(\mathbb{C}_{p})$, equipped with the right dual action of $\Gamma_{0}(\mathcal{N})$.  
\end{definition}  
\begin{remark} \label{rem:heckeinjection}
The space $\mathrm{Symb}_{\Gamma_{0}(\mathcal{N})}(V_{{k_{0},k_{0}}}(\mathbb{C}_{p})^{\vee})$ admits an action of the Hecke operators. In particular, there is a Hecke-equivariant injection (under our assumption that $F$ has class number one)
\begin{equation} \label{eqn:heckeinjection}
S_{\underline{k_{0}}+2}(U_{0}(\mathcal{N})) \hookrightarrow \mathrm{Symb}_{\Gamma_{0}(\mathcal{N})}(V_{k_{0},k_{0}}(\mathbb{C}_{p})^{\vee})
\end{equation}
with the co-kernel consisting of Eisenstein packets. In particular, to each cuspidal eigenform $\widehat{\cF} \in S_{\underline{k_{0}}+2}(U_{0}(\mathcal{N}))$, we can attach an eigensymbol $\phi_{\widehat{\cF}} \in \mathrm{Symb}_{\Gamma_{0}(\mathcal{N})}(V_{k_{0},k_{0}}(\mathbb{C}_{p})^{\vee})$.
\end{remark}
Let $\phi_{k_{0}} = \phi_{\cF}$ be the Bianchi modular symbol associated to $\cF = \cF_{k_{0}} \in S_{\underline{k_{0}}+2}(U_{0}(\mathcal{N}))^{\mathrm{new}}$.

\subsection{Overconvergent modular symbols} \label{subsec:overconvergentmodularsymbols}
We recall the theory of \emph{overconvergent modular symbols} of \cite{Wil17}. Since we are working under (\textbf{Heeg-Hyp}), we only consider the case that $p$ is inert in $F$ here. We refer the reader to \cite[\S3,\S7]{Wil17} for more details.

\begin{definition} \label{defn:locallyanalyticfunctions}
Let $F_{\p}$ denote the completion of $F$ at $\p = p$ and $\cO_{F_{\p}}$ its ring of integers. For an extension $L/\Qp$, let $\cA^{\p}_{k_{0},k_{0}}(\cO_{F_\p},L )$ denote the ring of locally analytic functions $\cO_{F_{\p}} \rightarrow L$ equipped with a natural `weight $k_{0}$' left action of the semigroup
\[
	\Sigma_0(\p) := \left\{\smallmatrd{a}{b}{c}{d} \in M_2(\cO_{F_{p}}) : v_{\p}(c) > 0, v_{\p}(a) = 0, ad-bc \neq 0\right\}
\]
given by
\[
	\smallmatrd{a}{b}{c}{d}\cdotspace g(z) = (a+cz)^{k_{0}}f\left(\tfrac{b+dz}{a+cz}\right).
\]
\end{definition}
We denote by $\cD_{k_{0},k_{0}}^{\p}(\cO_{F_\p},L)$ -- the space of locally analytic distributions on $\cO_{F_\p}$, to be the continuous $L$-dual of $\cA_{k_{0},k_{0}}^{\p}(\cO_{F_\p},L)$ equipped with the corresponding dual right (weight $k_{0}$) action. Note that the inclusion $V_{k_{0},k_{0}} \subset \cA_{k_{0},k_{0}}^{\p}$ induces, on taking the dual, a surjection $\cD_{k_{0},k_{0}}^{\p} \rightarrow V_{k_{0},k_{0}}^\vee$. In particular, we have a Hecke-equivariant map
\[
\rho : \Symb_{\Gamma_0(\cN)}(\cD_{k_{0},k_{0}}^{\p}(\cO_{F_\p},L)) \longrightarrow \Symb_{\Gamma_{0}(\cN)}(V_{k_{0},k_{0}}^\vee(L)).
\]
We recall the following \emph{Control Theorem} of Williams.
\begin{theorem}[Williams] \label{thm:control theorem}
We have
\[
	\rho|_{(\cF)} : \Symb_{\Gamma_0(\cN)}(\cD_{k_{0},k_{0}}^{\p}(\cO_{F_\p},L))_{(\cF)} \cong \Symb_{\Gamma_{0}(\cN)}(V_{k_{0},k_{0}}^\vee)_{(\cF)},
\]
that is, the restriction of $\rho$ to the $\mathcal{F}$-isotypic subspaces of the Hecke operators is an isomorphism. In particular, there is a unique overconvergent lift $\Psi_{\mathcal{F}}$ of $\phi_{\cF} = \phi_{k_{0}}$ under the map $\rho$.
\end{theorem}
\begin{proof}
This is \cite[Corollary 5.9]{Wil17} as $\cF$ is new at $p$ (See also \cite[Corollary 4.8]{BW19}).
\end{proof}
\subsection{Harmonic modular symbols} \label{subsec:harmonicmodularsymbol} \label{subsec:harmonicmodularsymbols}
In order to define a suitable $\p$-adic Integration theory which links the Stark--Heegner cycles to arithmetic data, we need to spread out the overconvergent modular symbol $\Psi_{\mathcal{F}}$ which is invariant under $\Gamma_{0}(\cN)$ to the larger Ihara group $\Gamma$, thereby enabling us to define a family of distributions over the projective line $\mathbb{P}^{1}(F_{\p})$. For more details see \cite[Section 3]{VW19} and \cite{BW19}.

\begin{definition} \label{def:bruhattitstree}
Let $\mathcal{T}_{\p}$ be the Bruhat--Tits tree for $\GLt(F_{\p})$ which is a connected tree with vertices given by homothety classes of $\cO_{F_{\p}}$-lattices $\mathscr{L} \subset (F_{\p})^{2}$. Two vertices are joined by an edge $e$ if one can find representatives of lattices $\mathscr{L}$ and $\mathscr{L}'$ such that
\[ p\mathscr{L}' \subset \mathscr{L} \subset \mathscr{L}' \]
where by abuse of notation, we denote a uniformiser in $\cO_{F_{\p}}$ simply by $p$. 
\end{definition}

Each edge comes with an orientation (given by the source and target vertices) and we denote the set of oriented edges of the Bruhat--Tits tree by $\mathcal{E}(\mathcal{T}_{\p})$ and the set of vertices by $\mathcal{V}(\mathcal{T}_{\p})$. For an ordered edge $e \in \cE(\cT_{\p})$, we denote its source vertex by $s(e)$ and target verex by $t(e)$. The edge $\overline{e} \in \cE(\cT_{\p})$ obtained from $e$ by interchanging its source and target vertices is called the edge \emph{opposite} to $e$.

Further we denote by $v_{*} \defeq [\mathscr{L}_{*}]$ to be the \emph{standard vertex} corresponding to the homothety class of lattices represented by $\mathscr{L}_{*} \defeq \cO_{F_{\p}}\oplus\cO_{F_{\p}}$. Similarly denote by $v_{\infty} \defeq [\mathscr{L}_{\infty}]$ for $\mathscr{L}_{\infty} \defeq p\cO_{F_{\p}}\oplus \cO_{F_{\p}}$. We also set $e_{\infty}$ to be the standard edge connecting $v_{*}$ and its neighbor $v_{\infty}$. We say a vertex is even (resp. odd) if it is connected to $v_{*}$ by an even (resp. odd) number of edges. An edge $e \in \mathcal{E}(\mathcal{T}_{\p})$ is called even (resp. odd) if its source vertex $v_{s(e)}$ is.  We denote the set of even (resp. odd) vertices and edges by $\mathcal{V}^{+}(\mathcal{T}_{\p})$ and $\mathcal{E}^{+}(\mathcal{T}_{\p})$ (resp. $\mathcal{V}^{-}(\mathcal{T}_{\p})$ and $\mathcal{E}^{-}(\mathcal{T}_{\p})$). There is a natural transitive action of $\PGLt(F_{\p})$ on the tree $\mathcal{T}_{\p}$ via M\"{o}bius transformation, which we can extended to a larger group. We recall 
\begin{definition}\label{groupomega}Recall that $\cN = p\cM$ with $p\nmid \cM$.
\begin{itemize}
\item[(i)] For $v$ a finite place of $F$, define 
\[R_0(\cM)_v \defeq \left\{\matr\in\mathrm{M}_2(\roi_v): c \equiv 0 \newmod{\cM}\right\}.\]
\item[(ii)] Let $R = R_0(\cM) \defeq \left\{\gamma \in \mathrm{M}_2\left(\A_F^f\right): \gamma_{v} \in R_0(\cM)_{v}\text{ for }v\neq p, \gamma_{\p}\in\mathrm{M}_2(F_{\p}) \right\}.$
\item[(iii)]Let $\widetilde{\Omega}$ denote the image of $R^\times$ in $\mathrm{PGL}_2\left(\A_F^f\right)$.
\item[(iv)] Let $\Omega \defeq \mathrm{PGL}_2^+\left(\A_F^f\right)\cap \widetilde{\Omega},$ where 
\[\PGLt^+\left(\A_F^f\right) \defeq \left\{\gamma \in \PGLt\left(\A_F^f\right):v_{\p}(\det(\gamma_{\p})) \equiv 0\newmod{2}\right\}.\]
\item[(v)] Finally, let 
\[
	\widetilde{\Gamma} =\widetilde{\Omega}\cap\PGLt(F), \hspace{15pt} \Gamma = \Omega\cap\PGLt(F).
\]
\end{itemize}
\end{definition}
The groups $\widetilde{\Omega}$ and $\Omega$ act on $\mathcal{T}_{\p}$ via projection onto $\PGLt(F_{\p})$. By \cite[Theorem 2, Chapter II.1.4]{Ser80}, we know that $\widetilde{\Omega}$ acts transitively on the sets $\mathcal{E}(\mathcal{T}_{\p})$ and $\mathcal{V}(\mathcal{T}_{\p})$ whilst $\Omega$ acts transitively on $\mathcal{E^{\pm}}(\mathcal{T}_{\p})$ and $\mathcal{V}^{\pm}(\mathcal{T}_{\p})$. Let $\PGLt(F_{\p})$ act on the projective line $\mathbb{P}^{1}(F_{\p})$ via 
\begin{equation} \label{eqn:projectivelineaction} \begin{pmatrix}a & b\\c & d\end{pmatrix}\cdot x \defeq \frac{b + dx}{a + cx}
\end{equation}
For $e \in \mathcal{E}(\mathcal{T}_{\p})$, let $\gamma_{e} \in \widetilde{\Omega}$ be such that $e = \gamma_{e}e_{*}$. We associate to the edge $e$, the open set 
\[ U_{e} \defeq \gamma_{e}^{-1}(\cO_{F_{\p}}) \defeq \lbrace x \in \mathbb{P}^{1}(F_{\p}) : \gamma_{e}.x \in \cO_{F_{\p}} \rbrace \subset \mathbb{P}^{1}(F_{\p}) \]
\begin{remark} \label{rem:compactopen}
The sets $U_{e}$, as $e$ ranges over $\mathcal{E}(\mathcal{T}_{\p})$, form a basis of compact open subsets of $\mathbb{P}^{1}(F_{\p})$.
\end{remark}
We can define Bianchi modular forms on the Bruhat--Tits tree $\mathcal{T}_{\p}$ which will allow us to extend distributions from $\cO_{F_{\p}}$ to those that are projective in $\p$. See \cite[Section 2]{BW19} and \cite[Section 3.2]{VW19} for details.
\begin{definition}
Let $\cA_{k_{0}}(\PP^1_{\p},L)$ denote the space of $L$-valued functions on $\PP^1(F_{\pri})$ that are locally analytic except for a pole of order at most $k_{0}$ at $\infty$ and let $\cD_{k_{0}}^{\pri}(\PP^1_{\pri},L)$ denote its continuous dual, i.e. $\cD_{k_{0}}^{\pri}(\PP^1_{\pri},L) := \mathrm{Hom}_{\mathrm{cts}}(\cA_{k_{0}}(\PP^1_{\pri},L), L)$.
\end{definition}

This space of distributions is a right $\Gamma$-module, and there is a natural restriction map 
\[
	\cD_{k_{0}}^{\pri}(\PP^1_{\pri},L) \longrightarrow \cD_{k_{0},k_{0}}^{\pri}(\cO_{F_\p},L),
\]

inducing
\[
	\rho_{\cT} : \Symb_{\Gamma}(\cD_{k_{0}}^{\pri}(\PP^1_{\pri},L)) \longrightarrow \Symb_{\Gamma_0(\cN)}(\cD_{k_{0},k_{0}}^{\pri}(\cO_{F_\p},L)).
\]
We call the domain $\Symb_{\Gamma}(\cD_{k_{0}}^{\pri}(\PP^1_{\pri},L))$ as `\emph{harmonic modular symbols} on $\cT_{\pri}$'. There is a natural action of the Hecke operators making the map $\rho_{\cT}$ Hecke equivariant. Recall that $\mathcal{F} \in S_{\underline{k_{0}}+2}(U_0(\cN))$ is a cuspidal Bianchi eigenform that is new at $p$.  Then Theorem 3.8 of \cite{VW19} shows that
\begin{theorem} \label{thm:overconvergentprojectivemodularsymbol}
We have an isomorphism on the $\mathcal{F}$-isotypic Hecke-eigenspace
\[
\rho_{\cT}|_{(\mathcal{F})} : \Symb_{\Gamma}(\cD_{k_{0}}^{\pri}(\PP^1_{\pri},L))_{(\mathcal{F})} \cong \Symb_{\Gamma_0(\cN)}(\cD_{k_{0},k_{0}}^{\pri}(\cO_{F_\p},L))_{(\mathcal{F})}.
\] 

\end{theorem}

In particular, after combining with Theorem \ref{thm:control theorem} above, we obtain a canonical element $\Phi_{\cF}^{\mathrm{har}} \in \Symb_{\Gamma}(\cD_{k_{0}}^{\pri}(\PP^1_{\pri},L))$ attached to the newform $\mathcal{F}$. For brevity, we set $\mathbf{MS}_{\Gamma}(L) \defeq \textup{Symb}_{\Gamma}(\cD_{k_{0}}^{\pri}(\PP^1_{\pri},L))$. In particular $\mathbf{MS}_{\Gamma}(L)_{(\mathcal{F})}$ is a one-dimensional $L$-vector space (See \cite[Remark 3.9]{VW19}). $\Phi_{\cF}^{\mathrm{har}}$ is called the `harmonic modular symbol' attached to the Bianchi eigenform $\cF$. 
\subsection{Double integrals and $\p$-adic Integration} \label{subsec:padicintegrationrecap}
We now recall the theory of double integrals developed in \cite[Section 6]{BW19}. Let $\log_{p}$ denote the branch of the $p$-adic logarithm $\log_{p}:\mathbb{C}_{p}^{\times} \longrightarrow \mathbb{C}_{p}$ such that $\log_{p}(p) = 0$. Recall that the unramified $\p$-adic upper half plane introduced in Section~\ref{sec:intro} is denoted by $\uhp_{\pri}^{\mathrm{ur}}$. 
\begin{definition} \label{defn:doubleintegrals}
	Let $\tau_1, \tau_2 \in \uhp_{\pri}^{\mathrm{ur}}$, $P \in V_{k_{0},k_{0}}(\Cp)$, $\mu \in \textup{Hom}(\Delta_{0}, \cD_{k_{0}}^{\pri}(\PP^1_{\pri},L))$ and $r,s \in \Proj(F)$.
\begin{itemize}
\item[(i)] For each $\sigma : F_{\p} \hookrightarrow L$ an embedding, define the `log' double integral at $\sigma$ by
	\[
	\int_r^s\int_{\tau_{1}}^{\tau_{2}} (P)\omega_{\mu}^{\textup{log}_p,\sigma} \defeq \int_{\Proj_{\pri}}\log_p\left(\frac{t_{\pri}-\tau_1}{t_{\pri}-\tau_2}\right)^{\sigma}P(t)d\mu\{r-s\}(t),
\]
	where $t_{\pri}$ is the projection of $t \in \Proj_{\pri}$ to $\Proj(F_{\pri})$. 
	\item[(ii)] We also define the `normed log' double integral as
	 \begin{multline}	\int_r^s\int_{\tau_{1}}^{\tau_{2}} (P)\omega_{\mu}^{\textup{log}_{p}} \defeq \sum\limits_{\sigma : F_{\p} \hookrightarrow L} \int_r^s\int_{\tau_{1}}^{\tau_{2}} (P)\omega_{\mu}^{\textup{log}_{p},\sigma}  \\ = \int_{\Proj_{\pri}}\log_p\circ\  \mathrm{N}_{F_{\pri}/\QQ_{p}}\left(\frac{t_{\pri}-\tau_1}{t_{\pri}-\tau_2}\right)P(t)d\mu\{r-s\}(t) \end{multline}
\item[(iii)] Define the `ord' double integral by
\begin{align*}
\int_r^s\int_{\tau_{1}}^{\tau_{2}}\ (P)\omega_{\mu}^{\textup{ord}_{\p}} &\defeq \int_{\Proj_{\pri}}\mathrm{ord}_{\pri}\left(\frac{t_{\pri}-\tau_1}{t_{\pri}-\tau_2}\right)P(t)d\mu\{r-s\}(t)\\
&=\sum\limits_{e : \textup{red}_{\mathfrak{p}}(\tau_1) \rightarrow \textup{red}_{\mathfrak{p}}(\tau_2)} \int_{U_{e}} P(t)d\mu\{r-s\}(t),
\end{align*} 
	where $\mathrm{red}_{\pri}:\uhp_{\pri} \rightarrow \mathcal{T}_{\pri} = \mathcal{E}(\mathcal{T}_{\pri}) \sqcup \mathcal{V}(\mathcal{T}_{\pri})$ is the reduction map and $U_e$ is the corresponding open set of $\Proj_{\pri}$.
\end{itemize}
Here we normalise so that $\mathrm{ord}_{\pri}(p) = 1$, noting that $p$ is a uniformiser in $F_{\pri}$. 
\end{definition}

\begin{remark} \label{rem:omegaF}
When $\mu = \Phi_{\cF}^{\mathrm{har}}$, we denote the double integrals defined above by 
$\int_{\tau_{1}}^{\tau_{2}}\int_r^s (P)\omega_{\cF}^{?}$ for $?$ either $\textup{log}_{p},\sigma$ or $\textup{ord}_{\pri}$. This notation will be useful in Section~\ref{sec:p-adicfamilies} below.
\end{remark}

For $?$ either $\textup{log}_{p},\sigma$ or $\textup{ord}_{\pri}$, we think of the \textit{double integrals} defined above as maps 
\begin{align*}
\Phi^{?} : &[\Delta_{0} \otimes \textup{Div}^{0}(\mathcal{H}_{\mathfrak{p}}^{\textup{ur}}) \otimes V_{k_{0},k_{0}}] \otimes 	\textup{Hom}(\Delta_{0}, \cD_{k_{0}}^{\pri}(\PP^1_{\pri},L)) \longrightarrow L\\
	&\big[(r - s)\ \otimes \ (\tau_1 - \tau_2) \ \otimes \ \ P \big]\ \otimes \ \mu \ \ \ \  \ \ \ \longmapsto \int_{\tau_{1}}^{\tau_{2}}\int_r^s (P)\omega_{\mu}^{?}.
\end{align*}
Also let
\[ \Phi^{\textup{log}_{p}} \defeq \Phi^{\textup{log}_p\circ N_{F_\pri/\Qp}} \defeq \sum_{\sigma} \Phi^{\textup{log}_{p},\sigma}. \] 
Since the pairings $\Phi^{\textup{log}_{p},\sigma}$ and $\Phi^{\textup{ord}_{\pri}}$ are $\Gamma$-invariant, we may interpret them as morphisms
\begin{equation} \label{eqn:homologytocohomology1}
\Phi^{\log_p},\Phi^{\log_p,\sigma},\ \Phi^{\ord_{\pri}} : (\Delta_{0} \otimes \textup{Div}^{0}(\uhp_{\mathfrak{p}}^{\textup{ur}}) \otimes V_{k,k})_{\Gamma} \rightarrow \mathbf{MS}_{\Gamma}(L)^{\vee}.
\end{equation}
We will denote the projection onto the $(\mathcal{F})$-isotypic component of $\mathbf{MS}_{\Gamma}(L)^{\vee}$, $\mathrm{pr}_{\mathcal{F}}\circ \Phi^{?}$ simply by $\Phi_{\mathcal{F}}^{?}$. Note that we have an exact sequence 
\begin{equation} \label{eqn:divisorexactsequence}
0 \rightarrow \Delta_0 \otimes \textup{Div}^{0}(\uhp_{\pri}^{\textup{ur}}) \otimes V_{k_{0},k_{0}} \rightarrow \Delta_0 \otimes \textup{Div}(\uhp_{\pri}^{\textup{ur}}) \otimes V_{k_{0},k_{0}} \rightarrow \Delta_0 \otimes V_{k_{0},k_{0}} \rightarrow 0,
\end{equation}
obtained by tensoring $0 \rightarrow \mathrm{Div}^0 \rightarrow \mathrm{Div} \rightarrow \Z \rightarrow 0$ with the flat $\Z$-module $\Delta_0\otimes V_{k_{0},k_{0}}$. On taking $\Gamma$-homology, we have
\begin{multline} \label{eqn:homologylongexactseq1}
\cdots \rightarrow \tupH_{i}(\Gamma, \Delta_0 \otimes \textup{Div}^{0}(\uhp_{\pri}^{\textup{ur}}) \otimes V_{k_{0},k_{0}})  \rightarrow \tupH_{i}(\Gamma, \Delta_0 \otimes \textup{Div}(\uhp_{\pri}^{\textup{ur}}) \otimes V_{k_{0},k_{0}}) \\
\rightarrow \tupH_{i}(\Gamma,\Delta_0 \otimes V_{k_{0},k_{0}}) \rightarrow \cdots
\end{multline}
In particular, we have the connecting morphism
\begin{equation} \label{eqn:connectingmorphism}
\tupH_{1}(\Gamma, \Delta_{0} \otimes V_{k_{0},k_{0}}) \xrightarrow{\delta} (\Delta_0 \otimes \textup{Div}^{0}(\uhp^{\textup{ur}}_{\pri}) \otimes V_{k_{0},k_{0}})_{\Gamma}.
\end{equation}
We record the following results from \cite{VW19}
\begin{theorem} \label{thm:towardsabeljacobi}
\begin{itemize}
\item[(i)]The morphism between the $L$-vector spaces
\[ \Phi^{\textup{ord}_{\pri}}_{\mathcal{F}}\circ\delta : \tupH_{1}(\Gamma, \Delta_{0} \otimes V_{k_{0},k_{0}}) \longrightarrow \mathbf{MS}_{\Gamma}(L)^{\vee}_{(\mathcal{F})} \]
is surjective.
\item[(ii)] For each embedding $\sigma:F_{\pri}\hookrightarrow L$, there exists a unique $\mathcal{L}^{\sigma}_{\pri} \in \Cp$ such that
\[ \Phi^{\log_p,\sigma}_{\mathcal{F}}\circ\delta = \mathcal{L}^{\sigma}_{\pri}\circ\Phi^{\ord_{\pri}}_{\mathcal{F}}\circ\delta : \tupH_{1}(\Gamma,\Delta_{0}\otimes V_{k_{0},k_{0}}) \rightarrow \mathbf{MS}_{\Gamma}(L)_{(\mathcal{F})}^{\vee}.\]
\end{itemize}
\begin{proof}
Part (i) follows from \cite[Thereom 3.15]{VW19} while Part (ii) is \cite[Corollary 3.17]{VW19}.
\end{proof}
\end{theorem}
\begin{remark} \label{rem:bwlinvariant}
We have an equality
\[  \mathcal{L}_{\pri}^{\mathrm{BW}} \defeq \sum_{\sigma} \mathcal{L}_{\pri}^\sigma,\]
where $\sigma$ ranges over all embeddings and $\mathcal{L}_{\pri}^{\mathrm{BW}}$ is the $\mathcal{L}$-invariant of \cite{BW19}. In particular, we also have 
\[ \Phi^{\log_p}_{\mathcal{F}}\circ\delta = \mathcal{L}^{\mathrm{BW}}_{\pri}\circ\Phi^{\ord_{\pri}}_{\mathcal{F}}\circ\delta : \tupH_{1}(\Gamma,\Delta_{0}\otimes V_{k_{0},k_{0}}) \rightarrow \mathbf{MS}_{\Gamma}(L)_{(\mathcal{F})}^{\vee}.\]
\end{remark}
For each $\sigma:F_{\pri} \hookrightarrow L$ we define
\begin{equation} \label{eqn:padicabeljacobi}
\Phi_{\cF}^{\sigma} \defeq -\Phi^{\mathrm{log}_{p},\sigma}_{\mathcal{F}} \oplus \Phi^{\mathrm{ord}_{\p}}_{\mathcal{F}} : (\Delta_0 \otimes \textup{Div}^{0}(\uhp_{\pri}^{\textup{ur}}) \otimes V_{k_{0},k_{0}})_{\Gamma} \rightarrow \mathbf{D}_{\mathcal{F},L}^{\sigma},
 \end{equation}
where we set 
\[\mathbf{D}_{\mathcal{F},L}^{\sigma} \defeq \mathbf{MS}_{\Gamma}(L)_{(\mathcal{F})}^{\vee} \oplus \mathbf{MS}_{\Gamma}(L)_{(\mathcal{F})}^{\vee}.\] 
Further, let
\begin{equation} \label{eqn:monodromymodule}
\mathbf{D}_{\cF} \defeq \bigoplus_{\sigma:F_{\p}\hookrightarrow L} \mathbf{D}_{\cF, L}^{\sigma}
\end{equation}
where each $\mathbf{D}_{\cF, L}^{\sigma}$ is a two dimensional $L$-vector space but with scalar action of $F_{\p}$ (viewed as a subfield of $L$) given by $\sigma$. In \cite[Section 4]{VW19}, we had given $\mathbf{D}_{\mathcal{F}}$ the structure of a rank two filtered $(\varphi, N)$-module over $F_{\p}$ with coefficients in $L$. $\mathbf{D}_{\cF, L}^{\sigma}$  is then a filtered $L$-vector space of dimension two with the filtration given by 
\[ \mathbf{D}_{\cF, L}^{\sigma} = \mathrm{Fil}^{0} \supsetneq \mathrm{Fil}^{1} = \ldots = \mathrm{Fil}^{k_{0}+1} \supsetneq \mathrm{Fil}^{k_{0}+2} = 0 \] 
where 
\begin{align*}
 \mathrm{Fil}^{\frac{k_{0}+2}{2}}\mathbf{D}_{\cF,L}^{\sigma} & \defeq \lbrace (-\mathcal{L}_{\p}^{\sigma}x,x) : x \in \mathbf{MS}_{\Gamma}(L)_{(\mathcal{F})}^{\vee} \rbrace\\
& = \mathrm{Im}(\Phi_{\mathcal{F}}^{\sigma}\circ\delta)  
 \end{align*} 
for each $\sigma : F_{\p} \hookrightarrow L$. 
\begin{remark} \label{rem:dependenceofbranchoflogp}
By definition, the morphisms $\Phi_{\cF}^{\log_{p},\sigma},\Phi_{\cF}^{\log_{p}}$ \& $\Phi_{\cF}^{\sigma}$,  the $\cL$-invariants $\cL_{\p}^{\sigma}$ and the filtered $(\varphi,N)$ module $\mathbf{D}_{\cF}$ all depend on the choice of the branch of the $p$-adic logarithm.
\end{remark}
From the $\Gamma$-homology exact sequence (\ref{eqn:homologylongexactseq1}), we have the connecting morphisms
\begin{equation} \label{eqn:homologylongexactseq2}
 \frac{(\Delta_0 \otimes \textup{Div}^{0}(\uhp_{\pri}^{\textup{ur}}) \otimes V_{k_{0},k_{0}})_{\Gamma}}{\delta(\tupH_{1}(\Gamma,\Delta_0 \otimes V_{k_{0},k_{0}}))} \xhookrightarrow{\partial_{1}} (\Delta_0 \otimes \textup{Div}(\uhp_{\pri}^{\textup{ur}}) \otimes V_{k_{0},k_{0}})_{\Gamma} \xrightarrow{\partial_{2}} (\Delta_0 \otimes V_{k_{0},k_{0}})_{\Gamma}.
 \end{equation}
We recall the definition of a $\p$-adic Abel--Jacobi map from \cite[Definition 5.1]{VW19}
\begin{definition} \label{defn:p-adicAbeliJacobi}
A \emph{$\pri$-adic Abel--Jacobi map} is a morphism
\[ \Phi^{\mathrm{AJ}}_{\sigma} : (\Delta_0 \otimes \textup{Div}(\uhp_{\pri}^{\textup{ur}}) \otimes V_{k_{0},k_{0}})_{\Gamma} \rightarrow \mathbf{D}_{\mathcal{F},L}^{\sigma}/\mathrm{Fil}^{\frac{k_{0}+2}{2}}(\mathbf{D}_{\mathcal{F},L}^{\sigma}) \]
such that the following diagram commutes:
\begin{equation} \label{eqn:commutativaAbelJacobi}
\xymatrix{
 \frac{(\Delta_0 \otimes \textup{Div}^{0}(\uhp_{\pri}^{\textup{ur}}) \otimes V_{k_{0},k_{0}})_{\Gamma}}{\delta(\tupH_{1}(\Gamma,\Delta_0 \otimes V_{k_{0},k_{0}}))} \ar@{^{(}->}[d]^{\partial_1} \ar[r]^<<<<<{\Phi^{\sigma}_{\mathcal{F}}} &     \mathbf{D}_{\mathcal{F},L}^{\sigma}/\mathrm{Fil}^{\frac{k_{0}+2}{2}}(\mathbf{D}_{\mathcal{F},L}^{\sigma}) \\
(\Delta_0 \otimes \textup{Div}(\uhp_{\pri}^{\textup{ur}}) \otimes V_{k_{0},k_{0}})_{\Gamma} \ar@{-->}[ur]^{\Phi^{\mathrm{AJ}}_{\sigma}} 
}
\end{equation}
for $\sigma: F_{\p} \hookrightarrow L$. In other words, a $\pri$-adic Abel--Jacobi map $\Phi^{\mathrm{AJ}}_{\sigma}$ is a lift of the morphism $\Phi_{\mathcal{F}}^{\sigma}$. 
\end{definition}
\begin{remark} \label{rem:uniquenessofabelijacobi}
Note that while there is no unique choice of a  lift of $\Phi^{\sigma}_{\mathcal{F}}$ to a $\p$-adic Abel--Jacobi map, we have shown in \cite[Theorem 6.5]{VW19} that the image of the $\p$-adic Abel--Jacobi image of the Stark--Heegner cycle is independent of such a choice (See also \cite[Remark 5.2]{VW19}).
\end{remark}


\subsection{Stark--Heegner cycles} \label{subsec:starkheegnercycles}
Recall from the Introduction (\S\ref{sec:intro}) that $K/F$ is a quadratic extension of relative discriminant $\mathcal{D}_{K/F}$ prime to the level $\mathcal{N} = \pri\mathcal{M}$ and satisfies the Stark--Heegner hypothesis (\textbf{SH-Hyp}). In particular, the completion $K_{\pri}$ of $K$ at the prime $p$ is the quadratic unramified extension of $F_{\mathfrak{p}}$. We fix $\delta_{K} \in \mathcal{O}_{K}\backslash\mathcal{O}_{F}$ such that $\delta_{K}^{2} \in \mathcal{O}_{F}$ is a generator of the discriminant ideal $(\cD_{K/F})$ (recall that we have assumed $F$ to have class number one). We will regard $\delta_{K}$ as an element of $K_{\mathfrak{p}}$ via $\iota_{p}$. 
Let $\mathcal{O}$ be an $\mathcal{O}_{F}[1/\mathfrak{p}]$-order of conductor $\mathcal{C}$ prime to $\mathcal{D}_{K/F}\mathcal{N}$ and let $\cR$ be the Eichler $\cO_{F}\left[1/\p\right]$--order in $M_{2}(\cO_{F}\left[1/\p\right])$ that are upper triangular modulo $\cM$.
\begin{definition} \label{defn:optimalembedding}
An embedding $\Psi : K \hookrightarrow M_{2}(F)$ is said to be \emph{optimal} if $\Psi(K) \cap \mathcal{R} = \Psi(\mathcal{O})$ and we denote the set of optimal embeddings of $\cO_{F}[1/\p]$-orders $\cO \hookrightarrow \cR$ by $\textup{Emb}(\mathcal{O}, \mathcal{R})$.  
\end{definition}
To an optimal embedding $\Psi \in \textup{Emb}(\mathcal{O}, \mathcal{R})$, we associate the following data.
\begin{itemize}
\item The two points $\tau_{\Psi}$ and $\tau_{\Psi}^{\theta} \in \mathcal{H}_{\mathfrak{p}}^{\textup{ur}}(K) := \mathcal{H}_{\mathfrak{p}}^{\textup{ur}} \cap K$ that are fixed by the action of $\Psi(K^{\times})$. Here $\tau_{\Psi}^{\theta} \defeq \theta(\tau_{\Psi})$ for $\theta \in \textup{Gal}(K/F)$, $\theta \neq \textup{id}$.
\item The fixed vertex $v_{\Psi} \in \mathcal{V}$ in the Bruhat-Tits tree for the action of $\Psi(K^{\times})$ on $\mathcal{V}$. Note that $v_{\Psi} = \mathrm{red}_{\p}(\tau_{\Psi}) = \mathrm{red}_{\p}(\tau_{\Psi}^{\theta})$.
\item The polynomial $P_{\Psi}(x,y) \defeq (cx^2 + (a - d)xy - by^2)(\overline{c}\overline{x}^2 + (\overline{a} - \overline{d})\overline{xy} - \overline{b}\overline{y}^{2})\in V_{2,2}$,
where $\Psi(\delta_{K}) = \smallmatrd{a}{b}{c}{d}$. 
\item Let $u$ be a fixed generator of $\mathcal{O}_{1}^{\times}/\{\textup{torsion}\} \cong \ZZ$ (by Dirichlet's Unit theorem), where $\mathcal{O}_{1} := \{ x \in \mathcal{O}\mid N_{K/F}(x) = 1 \}$. Let $\gamma_{\Psi} := \Psi(u)$ and $\Gamma_{\Psi}$ be the cyclic subgroup of $\Gamma$ generated by $\gamma_{\Psi}$. In particular $\Gamma_{\Psi} = \textup{Stab}(\Psi) \subseteq \Gamma$ and $P_{\Psi} \in (V_{2,2})^{\Gamma_{\Psi}}$.  
\end{itemize}
$\Psi$ is said to have \textit{positive} (resp.\ \emph{negative}) \emph{orientation} (at $\p$) if $v_{\Psi} \in \mathcal{V}^{+}(\mathcal{T}_{\p})$ (resp. $\mathcal{V}^{-}(\mathcal{T}_{\p})$). Then
\[ \textup{Emb}(\mathcal{O}, \mathcal{R}) = \textup{Emb}^{+}(\mathcal{O}, \mathcal{R}) \sqcup \textup{Emb}^{-}(\mathcal{O}, \mathcal{R}) \]
where $\textup{Emb}^{\pm}(\mathcal{O}, \mathcal{R})$ denotes the set of embeddings with positive/negative orientation. $\Gamma$ acts naturally on the set $ \textup{Emb}(\mathcal{O}, \mathcal{R})$ by conjugation and it preserves the subsets $\textup{Emb}^{\pm}(\mathcal{O}, \mathcal{R})$. Further, we know that the association
\[ \Psi \mapsto (\tau_{\Psi}, P_{\Psi}, \gamma_{\Psi}) \]
under conjugation action by any $\gamma \in \Gamma$ satisfies
\begin{equation} \label{eqn:conjugationaction}
(\tau_{\gamma\Psi\gamma^{-1}},P_{\gamma\Psi\gamma^{-1}},\gamma_{\gamma\Psi\gamma^{-1}}) = (\gamma\cdotspace\tau_{\Psi}, \gamma\cdotspace P_{\Psi}, \gamma\gamma_{\Psi}\gamma^{-1}).
\end{equation}
\begin{remark} \label{rem:galoisactiononembeddings}
Let $\theta \in \Gal(K/F)$ be the non-trivial element. We let $\Psi^{\theta} \in \textup{Emb}(\mathcal{O}, \mathcal{R})$ be the embedding defined as $\Psi^{\theta}(-) := \Psi(\theta(-))$. Then, a simple calculation shows that  
\[ 
(\tau_{\Psi^{\theta}},P_{\Psi^{\theta}},\gamma_{\Psi^{\theta}}) = (\tau_{\Psi}^{\theta}, -P_{\Psi}, \gamma_{\Psi}^{-1}).
 \]
 \end{remark}
Once we fix a cusp $x \in \mathbb{P}^{1}(F)$, we define
\[ 
\mathrm{D} : \textup{Emb}(\mathcal{O},\mathcal{R}) \longrightarrow \Delta_{0} \otimes \textup{Div}(\mathcal{H}_{\mathfrak{p}}^{\textup{ur}}) \otimes V_{k_{0},k_{0}},
 \]
\[ 
\mathrm{D}_{\Psi} := \mathrm{D}(\Psi) := (\gamma_{\Psi}\cdotspace x - x)\otimes \tau_{\Psi} \otimes \Big(\frac{1}{\sqrt{\mathrm{N}_{F/\QQ}(\cD_{K/F})}}\Big)^{k_{0}/2}P_{\Psi}^{k_{0}/2}. 
\]
\begin{remark} \label{rem:errorinVW21}
Note that there is a subtle error in \cite[\S6.1]{VW19} in defining the classes $\mathrm{D}_{\Psi}$ that has been corrected above.
\end{remark}
\begin{lemma} [Lemma 6.3, \cite{VW19}] \label{lem:darmoncycles1}
The image of $\mathrm{D}_{\Psi}$ in $(\Delta_{0} \otimes \textup{Div}(\mathcal{H}_{\mathfrak{p}}^{\textup{ur}}) \otimes V_{k_{0},k_{0}})_{\Gamma}$, which is denoted by $[\mathrm{D}_{\Psi}]$ remains the same if we replace $x$ with any $y \in \Gamma x$. Further, $[\mathrm{D}_\Psi]$ is invariant under the conjugation action of $\Gamma$ on $\textup{Emb}(\mathcal{O}, \mathcal{R})$. 
\end{lemma}
In particular, there is a well defined map
\begin{equation} \label{eqn:darmoncycles2}
\mathrm{D} : \Gamma/\textup{Emb}(\mathcal{O},\mathcal{R}) \longrightarrow (\Delta_{0} \otimes \textup{Div}(\mathcal{H}_{\mathfrak{p}}^{\textup{ur}}) \otimes V_{k_{0},k_{0}})_{\Gamma}.
\end{equation}
\begin{definition} \label{defn:stark-heegnercycle1}
We call $[\mathrm{D}_\Psi]$ the \emph{Stark--Heegner cycle} attached to the conjugacy class of optimal embeddings $[\Psi]$.
\end{definition}
Let $\sigma:F_{\p} \hookrightarrow L$ be any embedding as before. Recall from \cite[Theorem 6.5]{VW19} that
\begin{theorem} \label{thm:padicabeljacobiimageofdarmoncycles}
The $\pri$-adic Abel--Jacobi image of the Stark--Heegner cycle $\mathrm{D}_{[\Psi]}$ is independent of the choice of a $\pri$-adic Abel--Jacobi map. In other words, if  
\[ 
\Phi^{\mathrm{AJ}}_{\sigma, i} : (\Delta_0 \otimes \textup{Div}(\uhp_{\pri}^{\textup{ur}}) \otimes V_{k_{0},k_{0}})_{\Gamma} \longrightarrow \mathbf{D}_{\cF, L}^{\sigma}/\mathrm{Fil}^{\frac{k_{0}+2}{2}}(\mathbf{D}_{\cF, L}^{\sigma}) , \hspace{12pt} i = 1,2
\] 
 are any two $\pri$-adic Abel--Jacobi maps lifting $\Phi_{\mathcal{F}}^{\sigma}$, then
\begin{multline*} 
\Phi^{\mathrm{AJ}}_{\sigma,1}\left(\left[(\gamma_{\Psi}\cdotspace x - x) \otimes \tau_{\Psi} \otimes \left(\sqrt{\mathrm{N}_{F/\QQ}(\cD_{K/F})}\right)^{-k_{0}/2}P_{\Psi}^{k_{0}/2}\right]\right)\\ = \Phi^{\mathrm{AJ}}_{\sigma,2}\left(\left[(\gamma_{\Psi}\cdotspace x - x) \otimes \tau_{\Psi} \otimes \left(\sqrt{\mathrm{N}_{F/\QQ}(\cD_{K/F})}\right)^{-k_{0}/2}P_{\Psi}^{k_{0}/2}\right]\right)
\end{multline*}
\end{theorem}
\begin{definition} \label{defn:starkheegnercohomologyclass}
The $\p$-adic Abel--Jacobi image of the Stark--Heegner cycle attached to a conjugacy class of embeddings $[\Psi] \in \Gamma/\mathrm{Emb}(\cO,\mathcal{R})$ is defined as 
\[
	\mathfrak{s}_{[\Psi]}^{\sigma} := \Phi^{\mathrm{AJ}}_{\sigma}(\mathrm{D}_{[\Psi]}) \in \mathbf{D}_{\cF, L}^{\sigma}/\mathrm{Fil}^{\frac{k_{0}+2}{2}}(\mathbf{D}_{\cF, L}^{\sigma}).
\]
where $\Phi^{\mathrm{AJ}}_{\sigma}$ is any $\p$-adic Abel--Jacobi map associated to $\sigma : F_{\p} \hookrightarrow L$.
\end{definition}
There is a natural action of $\textup{Pic}(\mathcal{O})$ (by conjugation) on the set $\Gamma/\textup{Emb}(\mathcal{O},\mathcal{R})$ as in \cite[Proposition 5.8]{Dar01} and \cite[Proposition 2]{Tri06}. By the reciprocity isomorphism of class field theory,
\begin{equation} \label{eqn:cftisomorphism}
\mathrm{rec} : \textup{Pic}(\mathcal{O}) \cong \textup{Gal}(H_{\mathcal{C}}/K)
\end{equation}
we get a transported action of $\textup{Gal}(H_{\mathcal{C}}/K)$ on $\Gamma/\textup{Emb}(\mathcal{O},\mathcal{R})$ (via $\mathrm{rec}^{-1}$). Here $H_{\mathcal{C}}$ is the ring class field of conductor $\mathcal{C}$.

\subsubsection{Picard group torsors} \label{subsubsec:picardgrouptorsors} We will now fix orientations (at $\cM$) for the set of optimal embeddings. Note that by the Stark--Heegner hypothesis (\textbf{SH-Hyp}), there exists $\mathscr{M}' \subseteq \cO_{K}$ such that $\mathrm{N}_{K/F}\mathscr{M}' = \cM$. Now the ideal $\mathscr{M} \defeq \mathscr{M}'\cO_{F}\left[1/\p\right]$ is the kernel of a unique surjective $\cO_{F}$-algebra homomorphism (since $(\cC\p, \cM) = 1$)
\begin{equation} \label{eqn:orientation}
\mathfrak{o}: \cO \rightarrow \cO_{F}/\cM  
\end{equation}
The homomorphism $\mathfrak{o}$ is called an \emph{orientation} of the order $\cO$ (at $\cM$) which we fix along with the ideal $\mathscr{M}$. To an optimal embedding $\Psi \in \mathrm{Emb}(\cO, \cR)$, we can associate an orientation 
\[ \mathfrak{o}_{\Psi} : \cO \rightarrow \cO_{F}/\cM \]
which maps $a \in \cO$ to the upper-left hand entry of the matrix $\Psi(a)$ which can easily be verified to be an $\cO_{F}$-algebra morphism since $\Psi(a)$ is upper triangular modulo $\cM$.
\begin{definition} \label{def:orientedembeddings}
An optimal embedding $\Psi$ is said to be oriented if $\mathfrak{o}_{\Psi} = \mathfrak{o}$. We denote the set of oriented optimal embeddings by $\mathrm{Emb}^{\mathfrak{o}}(\cO, \cR) \subset \mathrm{Emb}(\cO, \cR)$. 
\end{definition}
The conjugation action of $\Gamma$ on $\mathrm{Emb}(\cO, \cR)$ in fact preserves $\mathrm{Emb}^{\mathfrak{o}}(\cO, \cR)$. Furthermore, we have
\begin{proposition} \label{prop:picardgrouptorsor}
There exists a bijection
\[ \mathrm{Pic}(\cO) \cong \mathrm{Gal}(H_{\cC}/K) \cong \Gamma/\mathrm{Emb}^{\mathfrak{o}}(\cO,\cR) \]
\end{proposition} 
\begin{proof}
See \cite[Proposition 2]{Tri06}.
\end{proof}
In particular, Proposition~\ref{prop:picardgrouptorsor} shows that the set $\Gamma/\mathrm{Emb}^{\mathfrak{o}}(\cO,\cR)$ is a $\mathrm{Gal}(H_{\cC}/K)$-torsor with the group action as described above.
\begin{remark} \label{rem:atkinlehnerorientation} For primes $\mathfrak{l} \mid \cN$, let $\alpha_{\mathfrak{l}}$ be the Atkin-Lehner matrices that define the Atkin-Lehner operators on the space of Bianchi eigenforms of level $\cN$ (See \cite[Section 3.3]{Cre81} and \cite[Section 5.3]{Ling05} for a precise definition), i.e.
\[ \cF \mid W_{\mathfrak{l}} \defeq \cF \mid \alpha_{\mathfrak{l}} = \omega_{\mathfrak{l}}\cF \]
where $\omega_{\mathfrak{l}} \in \lbrace \pm 1 \rbrace$.
For an oriented optimal embedding $\Psi \in \mathrm{Emb}^{\mathfrak{o}}(\cO, \cR)$, the optimal embedding $\alpha_{\cM}\Psi^{\theta}\alpha_{\cM}^{-1}$ has the same orientation as $\Psi$, i.e.
\[ \alpha_{\cM}\Psi^{\theta}\alpha_{\cM}^{-1} \in \mathrm{Emb}^{\mathfrak{o}}(\cO, \cR)\]
\end{remark}
\begin{definition} \label{defn:stark-heegnercycle2}
Let $\chi : \textup{Gal}(H_{\mathcal{C}}/K) \rightarrow \mathbb{C}^{\times}$ be any character. The \emph{$\chi$-twisted Stark--Heegner cycle} is then defined as 
\[ 
\mathrm{D}_{\chi} := \sum\limits_{\sigma \in \textup{Gal}(H_{\mathcal{C}}/K)} \chi^{-1}(\sigma)\mathrm{D}_{\sigma\Psi} \in (\Delta_{0} \otimes \textup{Div}(\mathcal{H}_{\mathfrak{p}}^{\textup{ur}}) \otimes V_{k_{0},k_{0}})_{\Gamma} \otimes \chi,
 \]
 where $(-) \otimes \chi$ denotes suitable scalar extension by $\chi$. We may also similarly define 
\[ 
\mathrm{D}_{\chi}^{\theta} \defeq \sum\limits_{\sigma \in \textup{Gal}(H_{\mathcal{C}}/K)} \chi^{-1}(\sigma)\mathrm{D}_{(\sigma\Psi)^{\theta}} \in (\Delta_{0} \otimes \textup{Div}(\mathcal{H}_{\mathfrak{p}}^{\textup{ur}}) \otimes V_{k_{0},k_{0}})_{\Gamma} \otimes \chi.
\]
Here $\sigma\Psi$ denotes the action of $\textup{Gal}(H_{\mathcal{C}}/K)$ on $\Gamma/\mathrm{Emb}(\cO,\cR)$ described above.  
\end{definition}
Further, we also set $\mathfrak{s}_{\chi}^{\sigma} \defeq \Phi^{\mathrm{AJ}}_{\sigma}(\mathrm{D}_{\chi})$ (resp. $\mathfrak{s}_{\chi}^{\sigma, \theta}\defeq \Phi^{\mathrm{AJ}}_{\sigma}(\mathrm{D}_{\chi}^{\theta})$). 
\begin{remark} \label{rem:notfixinganembedding}
In \cite[Section 5 and 6]{VW19}, we had dropped the dependence of the embedding $\sigma : F_{\pri} \hookrightarrow L$ from the notation in the $\pri$-adic Abel--Jacobi map as well as the image of the Stark--Heegner cycle under it. However, in the sequel we will need to consider the $\pri$-adic Abel--Jacobi image of the Stark--Heegner cycles over all possible embeddings when we relate them to $p$-adic $L$-functions, thus making the notation used here slightly different from that in \cite{VW19}. We explain this in detail in \S\ref{sec:reviewofpadicAJ} below.
\end{remark}

\section{Review of the $\p$-adic Abel--Jacobi map} \label{sec:reviewofpadicAJ}
For $\sigma : F_{\p} \hookrightarrow L$, we consider the following commutative diagram:
\begin{equation} \label{eqn:padicAJcommutativediagram}
\begin{tikzcd}
(\Delta_0 \otimes \textup{Div}^{0}(\uhp_{\pri}^{\textup{ur}}) \otimes V_{k_{0},k_{0}})_{\Gamma}  \arrow[r, "\Phi_{\cF}^{\sigma}"] \arrow[d, equal]
& \mathbf{D}_{\mathcal{F},L}^{\sigma}/\mathrm{Fil}^{\frac{k_{0}+2}{2}}(\mathbf{D}_{\mathcal{F},L}^{\sigma}) \arrow[d, "\mathrm{Pr}^{\sigma}" ] \\
(\Delta_0 \otimes \textup{Div}^{0}(\uhp_{\pri}^{\textup{ur}}) \otimes V_{k_{0},k_{0}})_{\Gamma}  \arrow[r, "\mathrm{log}\: \Phi_{\cF}^{\sigma}"]
& \mathbf{MS}_{\Gamma}(L)_{(\mathcal{F})}^{\vee} 
\end{tikzcd}
\end{equation}
where $\mathrm{Pr}^{\sigma}(x,y) \defeq -x - \cL_{\p}^{\sigma}y$ is well--defined \& an isomorphism and 
\begin{equation} \label{eqn:logAJ1}
\mathrm{log}\:\Phi_{\cF}^{\sigma} \defeq \Phi_{\cF}^{\mathrm{log_{p}}, \sigma} - \cL_{\p}^{\sigma}\Phi_{\cF}^{\mathrm{ord}_{\p}}. 
\end{equation}
By Theorem~\ref{thm:towardsabeljacobi}, we know that $\mathrm{log}\:\Phi_{\cF}^{\sigma} \circ \delta = 0$. Since $\mathrm{Pr}^{\sigma}$ is an isomorphism, we may identify $\mathrm{log}\;\Phi_{\cF}^{\sigma}$ and $\Phi_{\cF}^{\sigma}$. We also define
\begin{equation} \label{eqn:logAJ1.1}
(\Delta_0 \otimes \textup{Div}^{0}(\uhp_{\pri}^{\textup{ur}}) \otimes V_{k_{0},k_{0}})_{\Gamma} \xrightarrow{\mathrm{log}\: \Phi_{\cF}} \mathbf{MS}_{\Gamma}(L)_{(\mathcal{F})}^{\vee} 
\end{equation}
as $\mathrm{log}\:\Phi_{\cF} \defeq \sum\limits_{\sigma} \mathrm{Pr}^{\sigma}\circ\Phi_{\cF}^{\sigma}$. In particular, we have 
\begin{equation} \label{eqn:logAJ2} \mathrm{log}\:\Phi_{\cF} = \sum\limits_{\sigma}\mathrm{log}\:\Phi^{\sigma}_{\cF} = \Phi_{\cF}^{\mathrm{log_{p}}} - \cL_{\p}^{\mathrm{BW}}\Phi_{\cF}^{\mathrm{ord}_{\p}} 
\end{equation}
\paragraph*{} Recall that we have fixed a branch of the $p$-adic logarithm such that $\mathrm{log}_{p}(p) = 0$. For every $\ell \in L$, we define
\begin{equation} \label{eqn:branchp-adiclog} \mathrm{log}_{\ell} \defeq \mathrm{log}_{p} - \ell\mathrm{ord}_{\p} : L^{\times} \rightarrow L 
\end{equation}
to be the branch of the $p$-adic logarithm such that $\mathrm{log}_{\ell}(p) = -\ell$ (Note that in our convention $\mathrm{log}_{p}(p) = 0$ and not $-p$).  Corresponding to this branch of the $p$-adic logarithm, for each $\sigma : F_{\p} \hookrightarrow L$, we may define 
\[
\Phi^{\log_{\ell},\sigma} : (\Delta_{0} \otimes \textup{Div}^{0}(\mathcal{H}_{\mathfrak{p}}^{\textup{ur}}) \otimes V_{k_{0},k_{0}})_{\Gamma}  \longrightarrow \mathbf{MS}_{\Gamma}(L)^{\vee}
\]
and
\[
\Phi_{\cF}^{\log_{\ell},\sigma} : (\Delta_{0} \otimes \textup{Div}^{0}(\mathcal{H}_{\mathfrak{p}}^{\textup{ur}}) \otimes V_{k_{0},k_{0}})_{\Gamma}  \longrightarrow \mathbf{MS}_{\Gamma}(L)_{(\cF)}^{\vee}
\]
as in Section~\ref{subsec:padicintegrationrecap} above. As in Theorem~\ref{thm:towardsabeljacobi}, let $\cL_{\p}^{\ell,\sigma} \in \mathbb{C}_{p}$ be the unique scalar such that 
\[ \Phi_{\cF}^{\log_{\ell},\sigma} \circ \delta = \cL_{\p}^{\ell,\sigma} \circ \Phi_{\cF}^{\ord_{\p}} \circ \delta. \]
We may then set 
\[ \Phi_{\cF}^{\ell,\sigma} \defeq -\Phi_{\cF}^{\log_{\ell},\sigma} \oplus \Phi_{\cF}^{\ord_{\p}} : (\Delta_0 \otimes \textup{Div}^{0}(\uhp_{\pri}^{\textup{ur}}) \otimes V_{k_{0},k_{0}})_{\Gamma} \rightarrow \mathbf{D}_{\mathcal{F},L}^{\ell,\sigma} \]
where
\[\mathbf{D}_{\mathcal{F},L}^{\ell,\sigma} \defeq \mathbf{MS}_{\Gamma}(L)_{(\mathcal{F})}^{\vee} \oplus \mathbf{MS}_{\Gamma}(L)_{(\mathcal{F})}^{\vee}\]
is as before but with filtration given by
\begin{align*}
 \mathrm{Fil}^{\frac{k_{0}+2}{2}}\mathbf{D}_{\cF,L}^{\ell,\sigma} & \defeq \lbrace (-\mathcal{L}_{\p}^{\ell,\sigma}x,x) : x \in \mathbf{MS}_{\Gamma}(L)_{(\mathcal{F})}^{\vee} \rbrace\\
& = \mathrm{Im}(\Phi_{\mathcal{F}}^{\ell,\sigma}\circ\delta).  
 \end{align*}
 Similar to (\ref{eqn:padicAJcommutativediagram}) above, we have a commutative diagram 
\begin{equation} \label{eqn:padicAJcommutativediagram2}
\begin{tikzcd}
(\Delta_0 \otimes \textup{Div}^{0}(\uhp_{\pri}^{\textup{ur}}) \otimes V_{k_{0},k_{0}})_{\Gamma}  \arrow[r, "\Phi_{\cF}^{\ell,\sigma}"] \arrow[d, equal]
& \mathbf{D}_{\mathcal{F},L}^{\ell,\sigma}/\mathrm{Fil}^{\frac{k_{0}+2}{2}}(\mathbf{D}_{\mathcal{F},L}^{\ell,\sigma}) \arrow[d, "\mathrm{Pr}^{\ell,\sigma}" ] \\
(\Delta_0 \otimes \textup{Div}^{0}(\uhp_{\pri}^{\textup{ur}}) \otimes V_{k_{0},k_{0}})_{\Gamma}  \arrow[r, "\mathrm{log}\: \Phi_{\cF}^{\ell,\sigma}"]
& \mathbf{MS}_{\Gamma}(L)_{(\mathcal{F})}^{\vee} 
\end{tikzcd}
\end{equation}
where $\mathrm{Pr}^{\ell,\sigma}(x,y) = -x - \cL_{\p}^{\ell,\sigma}y$ and 
\begin{equation} \label{eqn:logAJ3}
\mathrm{log}\:\Phi_{\cF}^{\ell,\sigma} \defeq \Phi_{\cF}^{\mathrm{log_{\ell}}, \sigma} - \cL_{\p}^{\ell,\sigma}\Phi_{\cF}^{\mathrm{ord}_{\p}}. 
\end{equation}
We then have, 
\begin{proposition} \label{prop:p-adiclogarithmbranch1}
For each $\sigma : F_{\p} \hookrightarrow L$ and for every $\ell \in L$,
\[ \Phi_{\cF}^{\mathrm{log}_{\ell},\sigma} = \Phi_{\cF}^{\mathrm{log}_{p},\sigma} - \ell\Phi_{\cF}^{\mathrm{ord}_{\p}} \in \mathrm{Hom} \left( (\Delta_0 \otimes \textup{Div}^{0}(\uhp_{\pri}^{\textup{ur}}) \otimes V_{k_{0},k_{0}})_{\Gamma}, \mathbf{MS}_{\Gamma}(L)_{(\cF)}^{\vee} \right)  \]
and 
\[ \cL_{\p}^{\ell,\sigma} = \cL_{\p}^{\sigma} - \ell. \]
\end{proposition}
\begin{proof} See Proposition 3.1 and Lemma 3.2 of \cite{Sev12}.
\end{proof}
Now, let us choose the branch of the $p$-adic logarithm, $\mathrm{log}_{\cL_{\p}^{\sigma}}$, corresponding to $\ell = \cL_{\p}^{\sigma} \in L$. Then, by Proposition~\ref{prop:p-adiclogarithmbranch1}, the commutative diagram (\ref{eqn:padicAJcommutativediagram2}) above simplifies as
\begin{itemize}
\item[•] $\mathrm{Pr}^{\cL_{\p}^{\sigma},\sigma}(x,y) = -x$. 
\item[•] $\mathrm{log}\;\Phi_{\cF}^{\cL_{\p}^{\sigma},\sigma} = \mathrm{log}\;\Phi_{\cF}^{\sigma} = \Phi_{\cF}^{\mathrm{log}_{\cL_{\p}^{\sigma}}, \sigma}$.
\end{itemize}
In particular, $\mathrm{log}\;\Phi_{\cF}^{\ell,\sigma}$ is independent of the choice of a branch of the $p$-adic logarithm and
\begin{equation} \label{eqn:logPhiFbranch}
 \mathrm{log}\;\Phi_{\cF} = \sum\limits_{\sigma} \Phi_{\cF}^{\mathrm{log}_{\cL_{\p}^{\sigma}},\sigma} 
 \end{equation}
\subsection{Galois Action}
Since we consider only the case when $p$ inert in $F$, we have $F_{\p} \cong \QQ_{p^{2}}$ -- the unique quadratic unramifed extension of $\QQ_{p}$. In this setting, we have two (Galois) embeddings, viz. the identity and the (lift of the) Frobenius, $\sigma_{\mathrm{Id}}, \sigma_{\mathrm{Fr}}: F_{\p} \hookrightarrow L$. We fix an identification
\[ \sigma^{*} : \frac{\mathbf{D}_{\cF, L}^{\sigma_{\mathrm{Id}}}}{\mathrm{Fil}^{\frac{k_{0}+2}{2}}(\mathbf{D}_{\cF, L}^{\sigma_{\mathrm{Id}}})} \xlongrightarrow{\cong} \frac{\mathbf{D}_{\cF, L}^{\sigma_{\mathrm{Fr}}}}{\mathrm{Fil}^{\frac{k_{0}+2}{2}}(\mathbf{D}_{\cF, L}^{\sigma_{\mathrm{Id}}})} \]
such that the following diagram commutes 
\begin{equation} \label{eqn:projections}
\begin{tikzcd}
\frac{\mathbf{D}_{\cF, L}^{\sigma_{\mathrm{Id}}}}{\mathrm{Fil}^{\frac{k_{0}+2}{2}}(\mathbf{D}_{\cF, L}^{\sigma_{\mathrm{Id}}})} \arrow[d, "\mathrm{Pr}^{\sigma_{\mathrm{Id}}}"] \arrow[r, "\sigma^{*}"] & \frac{\mathbf{D}_{\cF, L}^{\sigma_{\mathrm{Fr}}}}{\mathrm{Fil}^{\frac{k_{0}+2}{2}}(\mathbf{D}_{\cF, L}^{\sigma_{\mathrm{Fr}}})} \\
 \mathbf{MS}_{\Gamma}(L)_{(\cF)}^{\vee} \arrow[r, equal] & \mathbf{MS}_{\Gamma}(L)_{(\cF)}^{\vee} \arrow[u, "(\mathrm{Pr}^{\sigma_{\mathrm{Fr}}})^{-1}"] 
\end{tikzcd}
 \end{equation}
 Via the identification $\sigma^{*}$, we have
 \[ \mathrm{Pr}^{\sigma_{\mathrm{Fr}}}\circ\sigma^{*}\circ\Phi_{\cF}^{\sigma_{\mathrm{Id}}} = \mathrm{Pr}^{\sigma_{\mathrm{Id}}}\circ\Phi_{\cF}^{\sigma_{\mathrm{Id}}} \in \mathrm{Hom}\left( \left( \Delta_{0} \otimes \mathrm{Div}^{0}(\cH_{\p}^{\mathrm{ur}}) \otimes V_{k_{0},k_{0}} \right)_{\Gamma}, \mathbf{MS}_{\Gamma}(L)_{(\cF)}^{\vee} \right) \]
 Then,
 \begin{equation} \label{eqn:logPhiF}
 \mathrm{log}\;\Phi_{\cF} = \mathrm{Pr}^{\sigma_{\mathrm{Fr}}}\circ\left( \Phi_{\cF}^{\sigma_{\mathrm{Fr}}} + \sigma^{*}\Phi_{\cF}^{\sigma_{\mathrm{Id}}} \right) \in \mathrm{Hom}\left( \left( \Delta_{0} \otimes \mathrm{Div}^{0}(\cH_{\p}^{\mathrm{ur}}) \otimes V_{k_{0},k_{0}} \right)_{\Gamma}, \mathbf{MS}_{\Gamma}(L)_{(\cF)}^{\vee} \right).
 \end{equation}
This motivates us to set
 \begin{equation} \label{eqn:PhiF}
 \Phi_{\cF} \defeq \Phi_{\cF}^{\sigma_{\mathrm{Fr}}} + \sigma^{*}\Phi_{\cF}^{\sigma_{\mathrm{Id}}} \in \mathrm{Hom}\left( \left( \Delta_{0} \otimes \mathrm{Div}^{0}(\cH_{\p}^{\mathrm{ur}}) \otimes V_{k_{0},k_{0}} \right)_{\Gamma}, \frac{\mathbf{D}_{\cF, L}^{\sigma_{\mathrm{Fr}}}}{\mathrm{Fil}^{\frac{k_{0}+2}{2}}(\mathbf{D}_{\cF, L}^{\sigma_{\mathrm{Fr}}})} \right)
 \end{equation}
 and 
 \begin{equation} \label{eqn:PhiAJ}
 \Phi^{\mathrm{AJ}} \defeq \Phi^{\mathrm{AJ}}_{\sigma_{\mathrm{Fr}}} + \sigma^{*}\Phi^{\mathrm{AJ}}_{\sigma_{\mathrm{Id}}} \in \mathrm{Hom}\left( \left( \Delta_{0} \otimes \mathrm{Div}(\cH_{\p}^{\mathrm{ur}}) \otimes V_{k_{0},k_{0}} \right)_{\Gamma}, \frac{\mathbf{D}_{\cF, L}^{\sigma_{\mathrm{Fr}}}}{\mathrm{Fil}^{\frac{k_{0}+2}{2}}(\mathbf{D}_{\cF, L}^{\sigma_{\mathrm{Fr}}})} \right)
 \end{equation}
 so that we have the following commutative diagrams :-
 \begin{equation} \label{eqn:PhiFlogPhiFcommutativediagram}
\begin{tikzcd}
(\Delta_0 \otimes \textup{Div}^{0}(\uhp_{\pri}^{\textup{ur}}) \otimes V_{k_{0},k_{0}})_{\Gamma}  \arrow[r, "\Phi_{\cF}"] \arrow[d, equal]
& \mathbf{D}_{\mathcal{F},L}^{\sigma_{\mathrm{Fr}}}/\mathrm{Fil}^{\frac{k_{0}+2}{2}}(\mathbf{D}_{\mathcal{F},L}^{\sigma_{\mathrm{Fr}}}) \arrow[d, "\mathrm{Pr}^{\sigma_{\mathrm{Fr}}}" ] \\
(\Delta_0 \otimes \textup{Div}^{0}(\uhp_{\pri}^{\textup{ur}}) \otimes V_{k_{0},k_{0}})_{\Gamma}  \arrow[r, "\mathrm{log}\: \Phi_{\cF}"]
& \mathbf{MS}_{\Gamma}(L)_{(\mathcal{F})}^{\vee} 
\end{tikzcd}
\end{equation} 
and
 \begin{equation} \label{eqn:PhiAJlogPhiAJcommutativediagram}
\begin{tikzcd}
(\Delta_0 \otimes \textup{Div}(\uhp_{\pri}^{\textup{ur}}) \otimes V_{k_{0},k_{0}})_{\Gamma}  \arrow[r, "\Phi^{\mathrm{AJ}}"] \arrow[d, equal]
& \mathbf{D}_{\mathcal{F},L}^{\sigma_{\mathrm{Fr}}}/\mathrm{Fil}^{\frac{k_{0}+2}{2}}(\mathbf{D}_{\mathcal{F},L}^{\sigma_{\mathrm{Fr}}}) \arrow[d, "\mathrm{Pr}^{\sigma_{\mathrm{Fr}}}" ] \\
(\Delta_0 \otimes \textup{Div}(\uhp_{\pri}^{\textup{ur}}) \otimes V_{k_{0},k_{0}})_{\Gamma}  \arrow[r, "\mathrm{log}\: \Phi^{\mathrm{AJ}}"]
& \mathbf{MS}_{\Gamma}(L)_{(\mathcal{F})}^{\vee} 
\end{tikzcd}
\end{equation} 
$\Phi_{\cF}$ (resp. $\Phi^{\mathrm{AJ}}$) should be considered as the sum of $\Phi_{\cF}^{\sigma}$ (resp. $\Phi^{\mathrm{AJ}}_{\sigma}$) over the possible embeddings $\sigma : F_{\p} \hookrightarrow L$.
\begin{remark} \label{rem:galoisequivarianceofL-invs}
Since $\cF$ is the base--change to $F$ of $f \in S_{k_{0}+2}(\Gamma_{0}(N))^{\mathrm{new}}$, we know by \cite[Lemma 4.4]{VW19} that 
\[ \cL_{\p}^{\sigma_{\mathrm{Id}}} = \cL_{\p}^{\sigma_{\mathrm{Fr}}} = \cL_{p}(f) \]
where $\cL_{p}(f)$ is the Darmon--Orton $\cL$-invariant attached to the newform $f$. Then, by (\ref{eqn:logPhiFbranch}),
\begin{equation} \label{eqn:logPhiFbasechange}
 \mathrm{log}\;\Phi_{\cF} = \Phi_{\cF}^{\mathrm{log}_{\cL_{p}(f)},\sigma_{\mathrm{Id}}} + \Phi_{\cF}^{\mathrm{log}_{\cL_{p}(f)},\sigma_{\mathrm{Fr}}} =  \Phi_{\cF}^{\mathrm{log}_{\cL_{p}(f)}}
\end{equation}
i.e. the $\mathbf{D}_{\cF, L}^{\sigma_{\mathrm{Fr}}}/\mathrm{Fil}^{\frac{k_{0}+2}{2}}$-valued integration theory $\Phi_{\cF}$ is equivalent, via $\mathrm{Pr}^{\sigma_{\mathrm{Fr}}}$, to the $\mathbf{MS}_{\Gamma}(L)_{(\cF)}^{\vee}$-valued integration theory
\[ \mathrm{log}\;\Phi_{\cF} = \Phi_{\cF}^{\mathrm{log}_{p}} - \cL_{\p}^{\mathrm{BW}}\Phi_{\cF}^{\mathrm{ord}_{\p}} =  \Phi_{\cF}^{\mathrm{log}_{\cL_{p}(f)}} \]
\end{remark}
\begin{remark} \label{rem:logAJ}
Similar to Definition~\ref{defn:p-adicAbeliJacobi}, we may think of $\mathrm{log}\;\Phi_{\cF}^{\mathrm{AJ}}$ as a lift of $\mathrm{log}\;\Phi_{\cF}$
\begin{equation} \label{eqn:commutativaAbelJacobi2}
\xymatrix{
 \frac{(\Delta_0 \otimes \textup{Div}^{0}(\uhp_{\pri}^{\textup{ur}}) \otimes V_{k_{0},k_{0}})_{\Gamma}}{\delta(\tupH_{1}(\Gamma,\Delta_0 \otimes V_{k_{0},k_{0}}))} \ar@{^{(}->}[d]^{\partial_1} \ar[r]^<<<<<{\mathrm{log}\;\Phi_{\mathcal{F}}} &     \mathbf{MS}_{\Gamma}(L)_{(\mathcal{F})}^{\vee}  \\
(\Delta_0 \otimes \textup{Div}(\uhp_{\pri}^{\textup{ur}}) \otimes V_{k_{0},k_{0}})_{\Gamma} \ar@{-->}[ur]^{\mathrm{log}\;\Phi^{\mathrm{AJ}}} 
}
\end{equation}

\end{remark}

\begin{remark} \label{rem:choiceofembedding}
Alternatively, we could also have set \[\Phi_{\cF} \defeq \Phi_{\cF}^{\sigma_{\mathrm{Id}}} + (\sigma^{*})^{-1}\Phi_{\cF}^{\sigma_{\mathrm{Fr}}} \in \mathrm{Hom}\left( \left( \Delta_{0} \otimes \mathrm{Div}^{0}(\cH_{\p}^{\mathrm{ur}}) \otimes V_{k_{0},k_{0}} \right)_{\Gamma}, \frac{\mathbf{D}_{\cF, L}^{\sigma_{\mathrm{Id}}}}{\mathrm{Fil}^{\frac{k_{0}+2}{2}}(\mathbf{D}_{\cF, L}^{\sigma_{\mathrm{Id}}})} \right)\] 
and
\[\Phi^{\mathrm{AJ}} \defeq \Phi^{\mathrm{AJ}}_{\sigma_{\mathrm{Id}}} + (\sigma^{*})^{-1}\Phi^{\mathrm{AJ}}_{\sigma_{\mathrm{Fr}}} \in \mathrm{Hom}\left( \left( \Delta_{0} \otimes \mathrm{Div}(\cH_{\p}^{\mathrm{ur}}) \otimes V_{k_{0},k_{0}} \right)_{\Gamma}, \frac{\mathbf{D}_{\cF, L}^{\sigma_{\mathrm{Id}}}}{\mathrm{Fil}^{\frac{k_{0}+2}{2}}(\mathbf{D}_{\cF, L}^{\sigma_{\mathrm{Id}}})} \right)\]
which would then fit in the following commutative diagrams :-
 \begin{equation} \label{eqn:PhiFlogPhiFcommutativediagram2}
\begin{tikzcd}
(\Delta_0 \otimes \textup{Div}^{0}(\uhp_{\pri}^{\textup{ur}}) \otimes V_{k_{0},k_{0}})_{\Gamma}  \arrow[r, "\Phi_{\cF}"] \arrow[d, equal]
& \mathbf{D}_{\mathcal{F},L}^{\sigma_{\mathrm{Id}}}/\mathrm{Fil}^{\frac{k_{0}+2}{2}}(\mathbf{D}_{\mathcal{F},L}^{\sigma_{\mathrm{Id}}}) \arrow[d, "\mathrm{Pr}^{\sigma_{\mathrm{Id}}}" ] \\
(\Delta_0 \otimes \textup{Div}^{0}(\uhp_{\pri}^{\textup{ur}}) \otimes V_{k_{0},k_{0}})_{\Gamma}  \arrow[r, "\mathrm{log}\: \Phi_{\cF}"]
& \mathbf{MS}_{\Gamma}(L)_{(\mathcal{F})}^{\vee} 
\end{tikzcd}
\end{equation} 
and
 \begin{equation} \label{eqn:PhiAJlogPhiAJcommutativediagram2}
\begin{tikzcd}
(\Delta_0 \otimes \textup{Div}(\uhp_{\pri}^{\textup{ur}}) \otimes V_{k_{0},k_{0}})_{\Gamma}  \arrow[r, "\Phi^{\mathrm{AJ}}"] \arrow[d, equal]
& \mathbf{D}_{\mathcal{F},L}^{\sigma_{\mathrm{Id}}}/\mathrm{Fil}^{\frac{k_{0}+2}{2}}(\mathbf{D}_{\mathcal{F},L}^{\sigma_{\mathrm{Id}}}) \arrow[d, "\mathrm{Pr}^{\sigma_{\mathrm{Id}}}" ] \\
(\Delta_0 \otimes \textup{Div}(\uhp_{\pri}^{\textup{ur}}) \otimes V_{k_{0},k_{0}})_{\Gamma}  \arrow[r, "\mathrm{log}\: \Phi^{\mathrm{AJ}}"]
& \mathbf{MS}_{\Gamma}(L)_{(\mathcal{F})}^{\vee} 
\end{tikzcd}
\end{equation} 
In particular, we may think of $\Phi_{\cF}$ (resp. $\Phi^{\mathrm{AJ}}$) as being $\left(\mathbf{D}_{\cF, L}^{\sigma}/\mathrm{Fil}^{\frac{k_{0}+2}{2}}\right)$--valued for either choice of an embedding $\sigma : F_{\p} \hookrightarrow L$.
\end{remark}

\section{Families of Bianchi modular forms and families of Bianchi modular symbols} \label{sec:p-adicfamilies}
The goal of this section is to prove Theorem~\ref{thm:padicAJimage} below which is a crucial ingredient in the proof of Theorem~\ref{thm:padicGZformula1} in Section~\ref{sec:p-adicLfns}. We first recall some of the requisite results on $p$-adic families of Bianchi modular forms, mainly following the exposition in \cite[\S3]{BW20a} (See also \cite{Han17}). Let $L/\QQ_{p}$ be a sufficiently large finite extension of $\QQ_{p}$ as before. Since $p$ is inert in $F$ under our running hypothesis (\textbf{Heeg-Hyp}), we have 
\[ \cO_{F} \otimes_{\ZZ} \ZZ_{p} \cong \cO_{F_{\p}}. \]
\begin{definition}
The Bianchi weight space of level $U_{0}(\mathcal{N})$ is defined as the rigid analytic space whose $L$--points are given by 
\[ \mathcal{W}_{F,\mathcal{N}}(L) \defeq \mathrm{Hom}_{\mathrm{cts}}\left(\cO_{F_{\p}}^{\times}/E(\mathcal{N}), L^{\times}\right)\]
where $E(\mathcal{N}) \defeq F^{\times} \cap U_{0}(\mathcal{N}) \cong \mathcal{O}_{F}^{\times} \defeq \mu(\cO_{F})$ -- the roots of unity in $\cO_{F}$.
\end{definition}    
A weight $\lambda_{\kappa} \in \mathcal{W}_{F,\mathcal{N}}(L)$ is said to be \emph{classical} if it is of the form $\epsilon\lambda_{\kappa}^{\mathrm{alg}}$ for $\epsilon$ a finite order character and $\lambda_{\kappa}^{\mathrm{alg}}(z) = z^{\mathbf{\kappa}} \defeq (z^{k_{1}})(\overline{z}^{k_{2}})$ for $\mathbf{\kappa} = (k_1, k_2)$ with $k_1,k_2 \in \ZZ$. Here the over--line indicates complex conjugation (i.e. action by the non-trivial element of $\mathrm{Gal}(F/\QQ)$).
\begin{remark} \label{rem:nolevel}
When the level of the $p$-adic weight space is clear, we shall drop it from the notation and denote the weight space simply by $\mathcal{W}_{F}$.
\end{remark} 
\begin{remark} \label{rem:parallelweights}
When $\epsilon$ is the trivial character and $\mathbf{\kappa} = (k,k)$ for $k \in \ZZ$, we shall call $\lambda_{\kappa} = \lambda_{k}$ a \emph{parallel weight}. Note that in this case
\[ \lambda_{k}(z) = (z^{k})(\overline{z}^{k}) = \mathrm{N}_{F/\QQ}(z)^{k} \]
We shall see later that the parallel weights $\lambda_{k}$ parametrize classical Bianchi modular forms of (parallel) weight $k+2$.
\end{remark}
\begin{remark} \label{rem:roleofmuOF}
Note that if $z \in \mu(\cO_{F})$, then for all parallel weights $\lambda_{k}$, we have that
\[ \lambda_{k} = \mathrm{N}_{F/\QQ}(z)^{k} = 1.\]
This is the reason that we consider the space of `null weights' as in \cite{BW20a} than in \cite{Han17} who considers characters on the split torus $\mathbf{T}$ of diagonal matrices in $\GL_{2}(\cO_{F_{\p}})$. See \cite[Remark 3.2]{BW20a}. 
\end{remark}
\begin{definition} \label{def:weightlambdaaction} 
Let $\mathcal{A}(\cO_{F_{\p}},L)$ denote the ring of $L$-valued \emph{locally analytic} functions on $\cO_{F_{\p}}$. For $\lambda_{\kappa} \in \mathcal{W}_{F}(L)$, we equip this space with a natural weight $\lambda_{\kappa}$ left action of the semi-group 
\[ \Sigma_{0}(\mathfrak{p}) \defeq \left\{ \begin{pmatrix}a & b\\c & d\end{pmatrix} \in \mathrm{M}_{2}(\cO_{F_{\p}}) : \nu_{\mathfrak{p}}(c) > 0, \nu_{\mathfrak{p}}(a) = 0, ad - bc \neq 0 \right\} \] 
given by 
\[\begin{pmatrix}a & b\\c & d\end{pmatrix}\cdot_{\lambda_{\kappa}} f(z) \defeq \lambda_{\kappa}(a+cz)f\left(\frac{b+dz}{a+cz}\right) \]
This transcends to a dual weight $\lambda_{\kappa}$ right action on $\mathcal{D}(L) \defeq \mathrm{Hom}_{\mathrm{cts}}(\mathcal{A}(\cO_{F_{\p}},L),L)$ - the space of $L$-valued \emph{locally analytic distributions} on $\cO_{F_{\p}}$. We will denote this space by $\mathcal{D}_{\lambda_{\kappa}}(L)$ to make the weight $\lambda_{\kappa}$ action implicit.
\end{definition}
Let $U \subset \mathcal{W}_{F}$ be an affinoid with associated universal character $\lambda_{U}^{\mathrm{un}} : \cO_{F_{\p}}^{\times} \longrightarrow \mathcal{O}(U)^{\times}$
i.e. any weight $\lambda_{\kappa} : \cO_{F_{\p}}^{\times}\rightarrow L^{\times}$ in $U(L)$ factors via 
\[ \cO_{F_{\p}}^{\times} \xrightarrow{\lambda_{U}^{\mathrm{un}}} \mathcal{O}(U)^{\times} \xrightarrow{\mathrm{ev}_{\lambda_{\kappa}}} L^{\times} \]
where the last map is evaluation at $\lambda_{\kappa}$.  Let $\mathcal{A}_{U} \defeq \cA(\cO(U))$ denote the space of $\cO(U)$-valued locally analytic functions on $\cO_{F_{\p}}$. The universal character enables us to equip $\mathcal{A}_{U}$ with a `weight $U$' action of the semi-group $\Sigma_{0}(\p)$ as follows
\[ \begin{pmatrix}a & b\\c & d\end{pmatrix}\cdot_{U}f(z) = \lambda_{U}^{\mathrm{un}}(a+cz)f \left( \frac{b+dz}{a+cz} \right). \]
Correspondingly, we get a dual `weight $U$' right action on $\cD_{U} \defeq \cD(\cO(U))$ -- the space of $\cO(U)$-valued locally analytic distributions on $\cO_{F_{\p}}$. For $W \subset U$ a closed subset, we have an isomorphism $\cD_{U} \otimes_{\cO(U)} \cO(W) \cong \cD_{W}$ (See \cite[Section 2.2]{Han17}). Particularly, if $\lambda_{\kappa} \in U(L)$ corresponds to a maximal ideal $\mathfrak{m}_{\lambda_{\kappa}} \subset \cO(U)$, then we have $\cD_{U} \otimes_{\cO(U)} \cO(U)/\mathfrak{m}_{\lambda_{\kappa}} \cong \cD_{\lambda_{\kappa}}$.
\paragraph*{} For the rest of this paper, we shall fix an affinoid $U \subset \mathcal{W}_{F}$ that contains the classical (parallel) weight $\lambda_{k_{0}}$. We will also denote $E(\mathcal{N})$ simply by $\mu(\cO_{F})$ henceforth. By Remark~\ref{rem:parallelweights}, we have 
\[ \lambda_{k_{0}}(z) = \mathrm{N}_{F/\QQ}(z)^{k_{0}} = \mathrm{N}_{F_{\p}/\QQ_{p}}(z)^{k_{0}}. \]
Further since $\cO_{F_{\p}}^{\times} \cong (\mathbb{F}_{p^{2}})^{\times} \times (1+p\cO_{F_{\p}})$, we may write any $z \in \cO_{F_{\p}}^{\times}$ in the form $z \defeq [z]\langle z \rangle$, where $[z] \in (\mathbb{F}_{p^{2}})^{\times}$ and $\langle z \rangle \in (1+p\cO_{F_{\p}})$ is the projection to the group of principal units. Up to shrinking the affinoid $U \subset \mathcal{W}_{F}$, we may assume that any $\lambda_{\kappa} \in U(L)$ is of the form $\lambda_{\kappa}(z) = [z]^{k_{0}}\langle z \rangle^{s} \defeq [z]^{k_{0}}\mathrm{exp}(s.\mathrm{log}_{p}(z))$ for $s \in \cO_{F_{\p}}$. In particular, any classical (parallel) weight $\lambda_{k} \in U(L)$ is of the form 
\[\lambda_{k}(z) = [N_{F_{\p}/\QQ_{p}}(z)]^{k_{0}}\langle N_{F_{\p}/\QQ_{p}}(z)\rangle^{k}.\]
Note that for all $\lambda_{k} \in U$ we have $k \equiv k_{0}\;\mathrm{mod}\:(p^{2}-1)$.
\begin{remark} \label{rem:parallelweightspace}
Let $\mathcal{W}_{F,\mathrm{par}} \subset \mathcal{W}_{F}$ be the parallel weight line defined as the image of $\mathcal{W}_{\QQ} \defeq \mathrm{Hom}_{\mathrm{cont}}(\ZZ_{p}^{\times},L^{\times})$ in $\mathcal{W}_{F}$. Our choice of the affinoid $U$ will be such that $U \subseteq \mathcal{W}_{F,\mathrm{par}}$. In particular, the pair $(U, k_{0})$ will be a \emph{slope-$k_{0}$ adapted} affinoid as defined in \cite[Section 4.1]{BW20a}.  By abuse of notation, we shall also denote the pre-image of $U$ in $\cW_{\QQ}$ by $U$. 
\end{remark}
\paragraph*{}
Let $\chi^{\mathrm{cycl}}_{\QQ}:\mathrm{G}_{\QQ} \rightarrow \ZZ_{p}^{\times}$ be the $p$-adic cyclotomic character. We denote by $\chi^{\mathrm{cycl}}_{F}$ its restriction to $\mathrm{G}_{F}$ which corresponds, via global class field theory, to a character 
\[\chi^{\mathrm{cycl}}_{F}:F^{\times}/\mathbb{A}_{F}^{\times} \longrightarrow \ZZ_{p}^{\times}\]
normalized in such a way that the restriction of $\chi^{\mathrm{cycl}}_{F}$ to $\cO_{F_{\p}}^{\times}$ is the local norm $\mathrm{N}_{F_{\p}/\QQ_{p}}$.
 \paragraph*{} Set $\mathscr{W} \defeq (F_{\p})^{2} - \left\{ 0,0 \right\}$ and denote by $\mathscr{Y}$ the space of orbits $\mathscr{Y} \defeq \mu(\cO_{F}) \backslash \mathscr{W}$, where $\mu(\cO_{F})$ acts diagonally. Consider the projection given by
\[
 \pi : \mathscr{W} \longrightarrow \mathbb{P}^{1}(F_{\p})
 \]
 \[ \pi((x,y)) \defeq [x:y] \]
which $\pi$ factors via $\mathscr{Y}$. 

For $\mathscr{L} \subset (F_{\p})^{2}$ any $\cO_{F_{\p}}$-lattice, let $\mathscr{L}'$ denote the set of primitive vectors of $\mathscr{L}$ i.e. vectors in $\mathscr{L}$ that are not divisible by $p$. For $g$ any $\cO_{F_{\p}}$--basis of $\mathscr{L}$, we set $|\mathscr{L}| \defeq (\mathrm{N}_{F/\QQ}(p))^{\mathrm{ord}_{p}(\mathrm{det}(g))}$. As before, let $\mathscr{L}_{*}$ denote the standard lattice $\cO_{F_{\p}}\oplus\cO_{F_{\p}}$ and let $\mathscr{L}_{\infty} = p\cO_{F_{\p}}\oplus \cO_{F_{\p}}$ which correspond to the vertices $v_{*}$ and $v_{\infty}$ in the Bruhat--Tits tree $\mathcal{T}_{\p}$ respectively. Let $e_{\infty} \in \mathcal{E}(\mathcal{T}_{\p})$ denote the (oriented) edge joining $v_{*}$ and $v_{\infty}$. In fact, for any $e \in \mathcal{E}(\mathcal{T_{\p}})$, we fix lattices $\mathscr{L}_{s(e)}$ and $\mathscr{L}_{t(e)}$ such that the homothety classes $[\mathscr{L}_{s(e)}]$ and $[\mathscr{L}_{t(e)}]$ represent the source and target vertices in the Bruhat--Tits tree respectively. 

For every $e \in \mathcal{E}(\mathcal{T}_{\p})$, denote by $W_{e} \defeq \mathscr{L}'_{s(e)}\cap\mathscr{L}'_{t(e)}$ and its image in $\mathscr{Y}$ by $Y_{e}$. Let $U_{e} \subset \mathbb{P}^{1}(F_{\p})$ be the open compact subset as in \cite[Proposition 2.4]{BW19}. Note that we have $p^{\infty}W_{e} \defeq \bigcup\limits_{n} p^{n}W_{e} = \pi^{-1}(U_{e})$. In particular $U_{e_{\infty}} = \cO_{F_{\p}}$. For brevity, we denote $W_{e_{\infty}} = p\cO_{F_{\p}} \oplus \cO_{F_{\p}}^{\times}\oplus$ and $Y_{e_{\infty}}$ simply by $W_{\infty}$ and $Y_{\infty}$ respectively.

\begin{definition} For $\mathscr{X}$ any open compact subset of $\mathscr{W}, \mathscr{Y}$ or $\mathbb{P}^{1}(F_{\p})$, we denote by $\mathcal{A}(\mathscr{X})$ the space of $L$-valued locally analytic functions on $\mathscr{X}$ and accordingly by $\mathcal{D}(\mathscr{X})$ the space of locally analytic distributions. 
\end{definition}

For any $\mu \in \mathcal{D}(\mathscr{X})$ and any $F \in \mathcal{A}(\mathscr{X})$, we use the measure theoretic definition $\int_{\mathscr{X}}Fd\mu$ to denote $\mu(F)$. Further, if $\mathscr{X}' \subset \mathscr{X}$ is any subset, then by $\int_{\mathscr{X}'}$ we mean $\mu(F.\chi_{\mathscr{X}'})$ where $\chi_{\mathscr{X}'}$ is the \emph{characteristic function} on $\mathscr{X}'$. 

\paragraph*{} By viewing elements of $(F_{\p})^{2}$ as column vectors, we have a natural left action of $\mathrm{GL}_{2}(F_{\p})$ on $(F_{\p})^{2}$ which induces a left action on the spaces $\mathscr{W}, \mathscr{Y}$ and the Bruhat--Tits tree $\mathcal{T}_{\p}$.  For $\mathscr{L}$ any lattice, we have an induced left action of $\mathrm{GL}_{2}(\cO_{F_{\p}})$ on $\mathscr{L}'$ and on $\widetilde{L}' \defeq \mu(\cO_{F}) \backslash \mathscr{L}' \subset \mathscr{Y}$. The diagonal action of $\cO_{F_{\p}}^{\times}$ on $\mathscr{L}'$, given by $t.(x,y)\defeq (t.x,t.y)$, descends to an action of $\cO_{F_{\p}}^{\times}/\mu(\cO_{F})$ on $\widetilde{L}'$. Note that we have a natural $\mathrm{GL}_{2}(\cO_{F_{\p}})$ action on $\mathcal{A}(\widetilde{L}')$ and an induced left $\mathrm{GL}_{2}(\cO_{F_{\p}})$ action on $\mathcal{D}(\widetilde{L}')$. We set $\mathbb{D} \defeq \mathcal{D}(\widetilde{L}_{*}')$ where $\widetilde{L}_{*}' \defeq \mu(\cO_{F})\backslash \mathscr{L}_{*}'$.

Similarly, let $\mathbb{D}^{\dagger} \defeq \cD(\mathscr{Y})$ where $\mathscr{Y} = \mu(\cO_{F})\backslash \mathscr{W}$. We give $\mathbb{D}^{?}$ for $? \in \lbrace \dagger, \emptyset \rbrace$ a $\mathcal{D}(\cO_{F_{\p}}^{\times}/\mu(\cO_{F}))$-module structure as follows
\[ \mathcal{D}(\cO_{F_{\p}}^{\times}/\mu(\cO_{F})) \times \mathbb{D}^{?} \longrightarrow \mathbb{D}^{?} \]
\[ (\nu, \mu) \mapsto \nu\mu \]
where
\[ \int_{\widetilde{L}_{*}'} F(x,y)d\nu\mu(x,y) \defeq \int_{\cO_{F_{\p}}^{\times}/\mu(\cO_{F})}\left(\int_{\widetilde{L}_{*}'} F(zx,zy)d\mu(x,y)\right) d\nu(z) \]
for $z \in \cO_{F_{\p}}^{\times}/\mu(\cO_{F})$. Further, we also give $\mathcal{O}(U)$ a $\mathcal{D}(\cO_{F_{\p}}^{\times}/\mu(\cO_{F}))$-module structure via the Amice--Velu Fourier transform given by 
\[ \nu \mapsto \left[ \lambda_{\kappa} \mapsto \int_{\cO_{F_{\p}}^{\times}/\mu(\cO_{F})}\lambda_{\kappa}(z)d\nu(z) \right]. \]
In particular, under the Amice--Velu Fourier transform (See \cite[Section 3.5]{AS08} for more details), we have 
\begin{equation} \label{eqn:Amice-VeluIsomorphism}
\mathcal{D}(\cO_{F_{\p}}^{\times}/\mu(\cO_{F})) \cong \mathcal{O}(\cW_{F}).
\end{equation}  
We set $\mathbb{D}_{U}^{?} \defeq \mathcal{O}(U) \widehat{\otimes}_{\mathcal{O}(\cW_{F})} \mathbb{D}^{?}$ which by (\ref{eqn:Amice-VeluIsomorphism}) is isomorphic to $\mathcal{O}(U) \widehat{\otimes}_{\mathcal{D}(\cO_{F_{\p}}^{\times}/\mu(\cO_{F}))} \mathbb{D}^{?}$. Note that that space $\mathbb{D}_{U}$ is contained in $\mathbb{D}_{U}^{\dagger}$ as locally analytic distributions with support contained in $\widetilde{L}_{*}'$.
\begin{definition} \label{defn:bigcell}
Let $\mathbf{N}^{\mathrm{opp}}$ (resp. $\mathbf{N}$) denote the set of unipotent lower triangular (resp. unipotent upper triangular) matrices in $\GL_{2}(\cO_{F_{\p}})$ and let $\mathbf{T} \defeq \cO_{F_{\p}}^{\times}/\mu(\cO_{F})$ viewed as diagonal matrices in $GL_{2}(\cO_{F_{\p}})$. Then, the ``big cell'' of \cite{AS08} is defined as $\mathbf{N}^{\mathrm{opp}}\mathbf{T}\mathbf{N}$. 
\end{definition}
\begin{lemma} \label{lem:bigcellprimitivevectors}
There is a bijective correspondence between the big cell $\mathbf{N}^{\mathrm{opp}}\mathbf{T}\mathbf{N}$ and the set $\widetilde{L}_{*}'$.
\end{lemma}
\begin{proof}
Any element in $\mathbf{N}^{\mathrm{opp}}\mathbf{T}\mathbf{N}$ is of the form \[\begin{pmatrix} b & bc \\ ab & b+abc \end{pmatrix}\] for $a,c \in \cO_{F_{\p}}$ and $b \in \cO_{F_{\p}}^{\times}/\mu(\cO_{F})$. Note that at most one of $ab$ or $b+abc$ is divisible by $p$ and in particular we have that $(ab, b+abc) \in \widetilde{L}_{*}'$. It is easy to see that this map is a bijection. 
\end{proof}
\begin{remark} \label{rem:DUvsDU}
The distribution module $\mathbb{D} = \cD(\widetilde{L}_{*}')$ should thus be considered as the space of locally analytic distributions on ``the big cell'' as in \cite{AS08} and \cite{AS00} (Similar spaces make an appearance in \cite{GS93}, \cite{BD09}, \cite{BDI10}, \cite{Sev12} as well as in \cite{Mok11} and \cite{BL11} in similar contexts) whilst the distribution module $\cD_{\lambda_{\kappa}}$ introduced in Definition~\ref{def:weightlambdaaction} is the classical one considered in \cite{BW20a}. By the \emph{universal property of the highest weight module/vector pair}, there exists a unique $\Sigma_{0}(\p)$-equivariant morphism $(\mathbb{D},\delta) \rightarrow (\cD_{\lambda_{\kappa}},\delta)$ that is compatible with the ``weight $k$'' specialisation to $(V_{k,k},\nu)$ for $\delta$ (resp. $\nu$) a highest weight vector in $\cD_{\lambda_{\kappa}}$ and $\mathbb{D}$ (resp. $V_{k,k}$). This extends to a $\Sigma_{0}(\p)$-equivariant morphism $(\mathbb{D}_{U},\delta) \rightarrow (\cD_{U},\delta)$.  

\end{remark}
\begin{definition} \label{def:homogeneous}
Let $\lambda_{\kappa} \in U(L)$ and $\mathscr{X}$ an $\cO_{F_{\p}}^{\times}$-stable open compact subset of $\mathscr{W},\mathscr{Y}$ or $\PP^{1}(F_{\p})$ (eg. $\mathscr{X} = \widetilde{L}_{*}'$). A function $f \in \mathcal{A}(\mathscr{X})$ is said to be homogeneous of weight $\lambda_{\kappa}$ if 
\[ f(zx,zy) = \lambda_{\kappa}(z)f(x,y) = [z]^{k_{0}}\langle\mathrm{N}_{F_{\p}/\QQ_{p}}(z)\rangle^{s}f(x,y) \]
for all $z \in \cO_{F_{\p}}^{\times}/\mu(\cO_{F})$. We denote the subspace of `homogeneous of weight $\lambda_{\kappa}$' functions by $\mathcal{A}^{\lambda_{\kappa}}(\mathscr{X}) \subset \mathcal{A}(\mathscr{X})$. 
\end{definition}
\begin{remark} \label{rem:homogeneouspolynomials}
Let $k \in \ZZ^{\geq 0}$ be any integer and let $P \in V_{k,k}(L)$.  Then the function $\widetilde{P}(x,y) \defeq P(x,y,\overline{x},\overline{y})$ for $(x,y) \in \widetilde{L}_{*}'$
is homogeneous of weight $\lambda_{k}$ on $\widetilde{L}_{*}'$, i.e. $\widetilde{P}(x,y) \in \cA^{\lambda_{k}}(\widetilde{L}_{*}').$ 
\end{remark}
For $\lambda_{\kappa} \in U(L)$, we define 
\[ B_{\lambda_{\kappa}} : \mathcal{O}(U) \times \mathbb{D} \longrightarrow \mathrm{Hom}_{\mathrm{cont}}(\mathcal{A}^{\lambda_{\kappa}}(\widetilde{L}_{*}'), L) \]
\[ B_{\lambda_{\kappa}}(\alpha,\mu)(F) \defeq \alpha(\lambda_{\kappa})\int_{\widetilde{L}_{*}'}F(x,y)d\mu(x,y) \]
In particular, we have a bilinear pairing
\[ B : \mathbb{D}_{U} \times \mathcal{A}^{\lambda_{\kappa}}(\widetilde{L}_{*}') \rightarrow L \]
We once again use the measure-theoretic notation to denote $B(\mu_{U},F)$ as $\int_{\widetilde{L}_{*}'}F(x,y)d\mu_{U}(x,y)$. For $\tau \in \mathcal{H}_{\p}^{\mathrm{ur}}$ and $\mathbf{P} \in \mathcal{A}^{\lambda_{k_{0}}}(\widetilde{L}_{*}')$ 
define
\[ F: U(L) \times \widetilde{L}_{*}' \longrightarrow \mathbb{C}_{p} \]
\[ F(\lambda_{\kappa}, (x,y)) \defeq \mathbf{P}(x,y)\left(\langle x-\tau y \rangle \langle \overline{x}-\overline{\tau
}\overline{y})\rangle \right)^{\lambda_{\kappa}-\lambda_{k_{0}}}. \]
Here $\langle x - \tau y\rangle^{\lambda_{\kappa}-\lambda_{k_{0}}} =\mathrm{exp}((s - k_{0})\mathrm{log}_{p}(x - \tau y))$, where $s \in \cO_{F_{\p}}$ is such that $\lambda_{\kappa}(z) = [z]^{k_{0}}\langle z \rangle^{s}$.  In particular $F(\lambda_{\kappa},(x,y)) \in \cA^{\lambda_{\kappa}}(\widetilde{L}_{*}')$.

\begin{lemma} \label{lem:analyticity}
For $\mu_{U} \in \mathbb{D}_{U}$, the function $U(L) \longrightarrow \mathbb{C}_{p}$ given by 
\[ \lambda_{\kappa} \mapsto \int_{\widetilde{L}_{*}'} F(\lambda_{\kappa}, (x,y))d\mu_{U}(x,y) \]
is analytic around the point $\lambda_{k_{0}} \in U(L)$.
\end{lemma}
\begin{proof}
The same proof as in \cite[Lemma 4.5]{BDI10} goes through.
\end{proof}
With notations in place as above, Lemma~\ref{lem:analyticity} motivates us to make the following definition.
\begin{definition} \label{def:derivative}
We define $ \int_{\widetilde{L}_{*}'} \mathbf{P}(x,y)\log_{p}\big(\langle x-\tau y \rangle \langle \overline{x}-\overline{\tau}\overline{y} \rangle \big)d\mu_{U}(x,y) $
to be the derivative
\[ \left( \frac{d}{d\lambda_{\kappa}}\int_{\widetilde{L}_{*}'}F(\lambda_{\kappa},(x,y))d\mu_{U}(x,y)\right)_{\lambda_{\kappa} = \lambda_{k_{0}}} \]
\end{definition}
The following result will be useful to compute derivatives of $p$-adic $L$-functions later on in the sequel.
\begin{proposition} \label{prop:derivatives}
Let $P(x,y) \in \cA^{\lambda_{k_{0}}}(\widetilde{L}_{*}')$ and $\tau_{1}, \tau_{2} \in \mathcal{H}_{\p}^{\mathrm{ur}}$. Then for $\mu \in \mathbb{D}_{U}$, we have 
\begin{align*}
\frac{d}{d\lambda_{\kappa}}\left(\int_{\widetilde{L}_{*}'}P(x,y)\left(\langle x-\tau_{1}y\rangle \langle \overline{x}-\tau_{1}\overline{y})\rangle \right)^{\frac{\lambda_{\kappa}-\lambda_{k_{0}}}{2}}\left( \langle x-\tau_{2} y \rangle \langle \overline{x}-\tau_{2}\overline{y}) \rangle \right)^{\frac{\lambda_{\kappa}-\lambda_{k_{0}}}{2}}d\mu_{U}(x,y)\right)_{\lambda_{\kappa} = \lambda_{k_{0}}} \\ =
\frac{1}{2}\frac{d}{d\lambda_{\kappa}}\left(\int_{\widetilde{L}_{*}'}P(x,y)\left( \langle x-\tau_{1}y \rangle \langle \overline{x}-\tau_{1}\overline{y} \rangle \right)^{\lambda_{\kappa}-\lambda_{k_{0}}}d\mu_{U}(x,y)\right)_{\lambda_{\kappa} = \lambda_{k_{0}}} \\ +
\frac{1}{2}\frac{d}{d\lambda_{\kappa}}\left(\int_{\widetilde{L}_{*}'}P(x,y)\left( \langle x-\tau_{2} y \rangle \langle \overline{x}-\tau_{2}\overline{y} \rangle \right)^{\lambda_{\kappa}-\lambda_{k_{0}}}d\mu_{U}(x,y)\right)_{\lambda_{\kappa} = \lambda_{k_{0}}}
\end{align*}
\end{proposition}
\begin{proof}
This follows from the explicit formulas of \cite[Remark 4.7]{BDI10}. See also \cite[Proposition 4.2]{Sev12}. 
\end{proof}
\begin{remark} \label{rem:smoothpoints}
There is a paucity of classical points in the Bianchi eigenvariety. However striking results of Barrera Salazar--Williams (c.f. \cite[Sections 5.2 \& 5.3]{BW20a}) show that base--change points are smooth in $\mathcal{E}_{\mathrm{par}}$ -- the parallel weight eigenvariety, which is the case that we are interested in.
\end{remark}
\subsection{Families of Bianchi modular symbols} Recall that $\mathcal{F} \in S_{\underline{k_{0}}+2}(U_{0}(\mathcal{N}))^{\mathrm{new}}$ is a (parallel) weight $k_{0}+2$ newform, where $\mathcal{N} = p\mathcal{M}$ as ideals in $\cO_{F}$ and $(p,\mathcal{M}) = 1$. By \cite[Corollary 4.8]{BW19}, we know that 
\begin{equation} \label{eqn:Up-eigenvalue} \mathcal{F}\mid U_{\p} = \omega_{\p}\mathrm{N}_{F/\QQ}(p)^{k_{0}/2}\mathcal{F} = \pm p^{k_{0}}\mathcal{F} 
\end{equation}
where $U_{\p}$ is the Hecke operator at $p$ and $-\omega_{\p} \in \lbrace \pm 1 \rbrace$ is its Atkin--Lehner eigenvalue. In particular $\mathcal{F}$ is of non-critical slope in the sense of \cite{BW20a}. Let $x_{k_{0}} \in \mathcal{E}(L)$ denote the classical cuspidal point in the Bianchi eigenvariety $\mathcal{E}$ defined over $L/\QQ_{p}$. Up to shrinking $U$, we may assume that $x_{k_{0}}$ varies in a family over the affinoid $U$ (See \cite[Definition 4.1 and Theorem 3.8]{BW20a}). We will denote the connected component containing $x_{k_{0}}$ in $\mathcal{E}$ by $V = \mathrm{Sp}(T)$. It follows from \cite{BW20a} that $U \subset \mathcal{W}_{F,\mathrm{par}}$ and $V \subset \mathcal{E}_{\mathrm{BC}} \subset \mathcal{E}_{\mathrm{par}}$ where $\mathcal{E}_{\mathrm{par}}$ (resp. $\mathcal{E}_{\mathrm{BC}}$) is the \emph{parallel weight eigenvariety} (resp. \emph{base--change eigenvariety}) of \cite[Section 5]{BW20a}. 
\begin{remark} \label{rem:Up-eigenvalue}
Note that since $\cF$ is, by assumption, the base--change to $F$ of an elliptic cuspidal newform $f \in S_{k_{0}+2}(\Gamma_{0}(N))$, we have that $\mathcal{F}\mid U_{\p} = (a_{p}(f))^{2} = p^{k_{0}}\cF$. In particular $\omega_{\p} = 1$. This will be crucial when we consider the \emph{trivial zero phenomenon} in \S\ref{sec:mainresult} later. 
\end{remark}
We recall the following result of Barrera Salazar and Williams.
\begin{proposition}[Barrera Salazar--Williams] \label{prop:zariskidensity}
Every irreducible component of $\mathcal{E}_{\mathrm{par}}$ is one-dimensional and contains a Zariski--dense set of classical points.
\end{proposition}
\begin{proof}
See \cite[Proposition 5.1]{BW20a}.
\end{proof}
By assumption, $(U,k_{0})$ is a slope-adapted affinoid in $\mathcal{W}_{F,\mathrm{par}}$ containing the classical weight $\lambda_{k_{0}}$. In particular, for each classical (parallel) weight $\lambda_{k} \in U$, the point $x_{k} \in V \subset \mathcal{E}_{\mathrm{BC}}$ lying above $\lambda_{k}$ (i.e. $\lambda_{k} = w(x_{k})$ for the weight map $w$) corresponds to a classical cuspidal (base--change) Bianchi eigenform of slope $h=k_{0}$ by the Control Theorem of \cite{Wil17}. We will denote the forms corresponding to $x_k \in V$ by $\mathcal{F}_{k} \in S_{\underline{k}+2}(U_{0}(\mathcal{N}))$ and correspondingly the elliptic cuspidal eigenforms by $f_{k} \in S_{k}(\Gamma_{0}(N))$. For all $k \neq k_{0}$, the eigenforms $\mathcal{F}_{k}$ (resp. $f_{k})$ are old at $p$ and there exists a unique normalized newform $\mathcal{F}_{k}^{\#} \in S_{\underline{k}+2}(U_{0}(\mathcal{M}))^{\mathrm{new}}$ (resp. $f_{k}^{\#} \in S_{k+2}(\Gamma_{0}(M)^{\mathrm{new}}$), such that for all $g \in \mathrm{GL}_{2}(\mathbb{A}_{F})$;
\begin{equation} \label{eqn:pstabilisation}
 \mathcal{F}_{k}(g) = \mathcal{F}_{k}^{\#}(g) - \frac{1}{a_{\p}(\cF_{k})}\mathcal{F}_{k}^{\#}\left(g\begin{pmatrix}
1 & 0\\
0 & p
\end{pmatrix}\right) 
\end{equation}
and 
\begin{equation} \label{eqn:pstabilistaion2}
f_{k}(q) = f_{k}^{\#}(q) - \frac{p^{k+1}}{a_{p}(f_{k})}f_{k}^{\#}(q^{p})
\end{equation}
where $a_{\p}(\cF_{k})$ (resp. $a_{p}(f_k)$) is the $U_{\p}$-eigenvalue of $\cF_{k}$ (resp. $U_{p}$-eigenvalue of $f_{k}$). We also set $\mathcal{F}_{k_{0}}^{\#} \defeq \mathcal{F}_{k_{0}} = \cF$ (resp. $f_{k_{0}}^{\#} \defeq f_{k_{0}} = f$). To ease our notation, we will denote by $\mathbfcal{F}(g)$ (resp. $\mathbf{f}(q)$) to denote the Coleman family, over the affinoid $U$, of cupidal Bianchi eigenforms (resp. cuspidal elliptic eigenforms) passing through the forms $\cF_{k}$ (resp. $f_{k}$) for classical weights $\lambda_{k} \in U$. We will also occasionally denote by $\mathbf{f}/F$ to denote the (base--change) family $\mathbfcal{F}$.
\begin{remark}
Note that we follow Palacios' definition of $p$-stabilisation of Bianchi eigenforms as in \cite[Section 3.3]{Pal21} which is slightly different from that in \cite{BW20a}.
\end{remark}
\paragraph*{} For each classical weight $\lambda_{k} \in U(L)$, we can define a weight $\lambda_{k}$-specialization map
\[\widetilde{\rho}_{\lambda_{k}} : \mathbb{D} \longrightarrow V_{k,k}(L)^{\vee}\]
given by
\begin{equation} \label{eqn:weightkspecialisation0}
\widetilde{\rho}_{\lambda_{k}}\left(\mu \right)(P) \defeq \int_{\mu(\cO_{F})\backslash(\cO_{F_{\p}}^{\times} \oplus \cO_{F_{\p}})} \widetilde{P}(x,y)d\mu(x,y)
\end{equation}
for $\mu \in \mathbb{D}$, which gives rise to a homomorphism at the level of Bianchi modular symbols
\begin{equation} \label{eqn:weightkspecialization}
\rho_{\lambda_{k}} : \mathrm{Symb}_{\Gamma_{0}(\mathcal{N})}(\mathbb{D}_{U}) \longrightarrow \mathrm{Symb}_{\Gamma_{0}(\mathcal{N})}(\mathbb{D}) \longrightarrow \mathrm{Symb}_{\Gamma_{0}(\mathcal{N})}(V_{k,k}(L)^{\vee}) 
\end{equation}
where the first map is given by $\Phi \mapsto \Phi \mbox{ mod }\mathfrak{m}_{\lambda_{k}}$, for $\Phi \in \mathrm{Symb}_{\Gamma_{0}(\mathcal{N})}(\mathbb{D}_{U})$. The following variant of a result of Barrera Salazar and Williams will be crucial to the theory developed in this section.
\begin{theorem}[Williams, Barrera Salazar--Williams] \label{thm:weightkspecialization}
There exists $\Phi_{\infty} \in \mathrm{Symb}_{\Gamma_{0}(\mathcal{N})}(\mathbb{D}_{U})$ such that
\begin{itemize}
\item[•] For every classical weight $\lambda_{k} \in U(L)$, we have $\rho_{\lambda_{k}}(\Phi_{\infty}) = C(k)\phi_{k}$ for some $p$-adic periods $C(k) \in L^{\times}$.
\item[•] $\rho_{\lambda_{k_{0}}}(\Phi_{\infty}) = \phi_{k_{0}}$
\end{itemize}
\end{theorem}
\begin{proof}
This follows from a combination of \cite[Proposition 6.7]{BW20a} and the Control Theorem \cite[Corollary 4.3]{Wil17}.
\end{proof}
\begin{remark} \label{rem:DUvsDU2}
Whilst Williams works with overconvergent modular symbols in \cite{Wil17}, the authors use overconvergent cohomology in \cite{BW20a}. However, we can easily identify the two spaces (See \cite[Eqn. (2.2)]{BW20a} or \cite[Lemma 8.2]{BW19} for instance). Moreover, in \cite{BW20a}, the authors use the ``classical'' distribution module $D_{U}$ in contrast to the space of distributions on the ``big cell'' $\mathbb{D}_{U}$ described above. However, we can obtain a lifting as in Theorem~\ref{thm:weightkspecialization} by Remark~\ref{rem:DUvsDU}. See also \cite[Theorem 6.2.1 \& Theorem 6.4.1]{AS08}. 

\end{remark}
In particular for all $\lambda_{k} \in U(L)$, we can write down the weight $\lambda_{k}$ specialization as 
\begin{equation} \label{eqn:specialization}
\int_{\mu(\cO_{F})\backslash(\cO_{F_{\p}}^{\times} \oplus \cO_{F_{\p}})} \widetilde{P}(x,y)d\Phi_{\infty}\lbrace r - s \rbrace (x,y) = C(k)\phi_{k}\lbrace r - s \rbrace \left(P(x,y)\right).
\end{equation}
for all $P(x,y) \in V_{k,k}(L)$ and $r,s \in \PP^{1}(F)$.

Recall that for a lattice $\mathscr{L} \subset (F_{\p})^{2}$, we denote by $\widetilde{L} = \mu(\cO_{F})\backslash \mathscr{L}$ its image in $\mathscr{Y}$.Then,
\begin{proposition} \label{prop:latticesymbols}
There exists a family of $\mathbb{D}_{U}^{\dagger}$-valued modular symbols $\lbrace \Phi_{\widetilde{L}} \rbrace_{\widetilde{L}}$ indexed by the sets $\widetilde{L} =  \mu(\cO_{F})\backslash \mathscr{L} \subset \mathscr{Y}$ that satisfy
\begin{itemize}
\item[•] $\Phi_{\widetilde{L}_{*}} = \Phi_{\infty}$
\item[•] For all $\gamma \in \widetilde{\Gamma}$ and for all compact open sets $Y \subset \mathscr{Y}$, 
\[ \int_{\gamma Y} (\gamma\cdot F)d\Phi_{\gamma\widetilde{L}}\lbrace \gamma r - \gamma s \rbrace = \int_{Y} (F)d\Phi_{\widetilde{L}}\lbrace r - s \rbrace \]
\end{itemize}
\end{proposition} 
\begin{proof}
The proof follows \cite[Proposition 1.8]{BD07}. Since $\widetilde{\Gamma}$ acts transitively on the set of lattices $\mathscr{L} \subset (F_{\p})^{2}$, the induced action on the sets $\widetilde{L}$ is also transitive. The stabilizer of $\widetilde{L}_{*}$ under the action is $\Gamma_{0}(\mathcal{N})$. Since $\Phi_{\infty}$ is invariant under the action of $\Gamma_{0}(\mathcal{N})$, the family of distributions is well-defined and determines $\Phi_{\widetilde{L}}$ uniquely once we set $\Phi_{\widetilde{L}_{*}} \defeq \Phi_{\infty}$.
\end{proof}
\begin{lemma} \label{lem:latticesymbolssupport}
Let $\mathscr{L}$ be a lattice and $\widetilde{L} = \mu(\cO_{F})\backslash \mathscr{L}$ Then the distributions $\Phi_{\widetilde{L}}\lbrace r - s \rbrace$ are supported on $\widetilde{L}'$ for all $r, s \in \mathbb{P}^{1}(F)$.
\end{lemma}
\begin{proof}
The locally analytic distributions $\Phi_{\widetilde{L}_{*}}\lbrace r - s \rbrace = \Phi_{\infty}\lbrace r - s \rbrace$ are supported on $\widetilde{L}_{*}'$. The lemma now follows from the $\widetilde{\Gamma}$-equivariance property defining the modular symbols $\Phi_{\widetilde{L}}$ at the other sets $\widetilde{L} \subset \mathscr{Y}$. 
\end{proof}
\begin{lemma} \label{lem:indexpmodularsymbols}
Let $\mathscr{L}_{1}, \mathscr{L}_{2}$ be two lattices in $(F_{\p})^{2}$ such that $\mathscr{L}_{2} \subset \mathscr{L}_{1}$ is a sub-lattice of index $\mathrm{N}_{F/\QQ}(p)$. Let $e \in \mathcal{E}(\mathcal{T})$ denote the (ordered) edge joining the vertices corresponding to the homothety classes $[\mathscr{L}_{1}]$ and $[\mathscr{L}_{2}]$. Then for all $\lambda_{k} \in U(L)$ and for all $F \in \mathcal{A}^{\lambda_{k}}(Y_{e})$, we have 
\[ \int_{Y_{e}}F(x,y)d\Phi_{\widetilde{L}_{2}}\lbrace r - s \rbrace = \frac{a_{\p}(\mathcal{F}_{k})}{\mathrm{N}_{F/\QQ}(p)^{k}}\int_{Y_{e}}F(x,y)d\Phi_{\widetilde{L}_{1}}\lbrace r - s \rbrace \]
where $Y_{e} = \widetilde{L}'_{1} \cap \widetilde{L}'_{2}$.
\end{lemma}
\begin{proof}
The proof is similar to that of \cite[Lemma 1.10]{BD07}. Since the group $\widetilde{\Gamma}$ acts transitively on pairs $(\mathscr{L}_{1},\mathscr{L}_{2})$ of lattices satisfying $[\mathscr{L}_{1}:\mathscr{L}_{2}] = \mathrm{N}_{F/\QQ}(p)$, it suffices to establish the result when 
\[ \mathscr{L}_{1} = \frac{1}{p}\cO_{F_{\p}} \oplus \cO_{F_{\p}},\mbox{                  }\mathscr{L}_{2} = \cO_{F_{\p}} \oplus \cO_{F_{\p}}\]
so that
\[ \mathscr{L}_{1}' \cap \mathscr{L}_{2}' = \cO_{F_{\p}} \oplus \cO_{F_{\p}}^{\times}. \]
For each class $a \: \mbox{mod}\:p$, let $\mathscr{L}_{(a)}$ denote the $\cO_{F_{\p}}$-lattice
\[ \mathscr{L}_{(a)} \defeq \lbrace (x,y) \in \cO_{F_{\p}}^{2}\mbox{ such that }x+ay \in p\cO_{F_{\p}} \rbrace .\]
Then,
\[ \mathscr{L}_{2}' \cap \mathscr{L}_{(a)}' = \lbrace (x,y) \in \cO_{F_{\p}} \oplus \cO_{F_{\p}}^{\times} \mbox{ such that } xy^{-1} \equiv -a\:\mbox{mod}\: p \rbrace.\]
Hence, we have a disjoint union
\[ \mathscr{L}_{1}' \cap \mathscr{L}_{2}' = \bigcup\limits_{a\:\mathrm{mod}\:p}\mathscr{L}_{2}' \cap \mathscr{L}_{(a)}'\]
which implies
\begin{equation} \label{eqn:indexpeqn1}
\int_{Y_{e}}F(x,y)d\Phi_{\widetilde{L}_{2}}\lbrace r - s \rbrace = \sum\limits_{a\:\mathrm{mod}\:p}\int_{Y_{(a)}}F(x,y)d\Phi_{\widetilde{L}_{2}}\lbrace r - s \rbrace
\end{equation}
where $Y_{(a)} \defeq \widetilde{L}_{2}' \cap \widetilde{L}_{(a)}'$. Let $\gamma_{a} \defeq \begin{pmatrix} 1 & a \\ 0 & p \end{pmatrix}$ be the matrices, as $a$ ranges over classes $\mathrm{mod}\:p$, that are used to define the Hecke operator $U_{p}$ in \cite{Wil17}. Then note that
\[ \gamma_{a}\mathscr{L}_{2} = p\mathscr{L}_{1},\mbox{    }\gamma_{a}\mathscr{L}_{a} = p\mathscr{L}_{2},\mbox{ so that }\gamma_{a}Y_{(a)} = pY_{e} \]
using which we re-write (\ref{eqn:indexpeqn1}), upon applying Proposition~\ref{prop:latticesymbols}, as
\begin{align*} \label{eqn:indexpeqn2}
\int_{Y_{e}}F(x,y)d\Phi_{\widetilde{L}_{2}}\lbrace r - s \rbrace & = \sum\limits_{a\:\mathrm{mod}\:p} \int_{\gamma_{a}Y_{(a)}}\gamma_{a}\cdot F(x,y)d\Phi_{\gamma_{a}\widetilde{L}_{2}}\lbrace \gamma_{a}r - \gamma_{a}s \rbrace \\
 & =  \sum\limits_{a\:\mathrm{mod}\:p} \int_{pY_{e}}\gamma_{a}\cdot F(x,y)d\Phi_{p\widetilde{L}_{1}}\lbrace \gamma_{a}r - \gamma_{a}s \rbrace\\
 & = \sum\limits_{a\:\mathrm{mod}\:p} \int_{Y_{e}}\mathbb{P}^{-1}\gamma_{a}\cdot F(x,y)d\Phi_{\widetilde{L}_{1}}\lbrace \mathbb{P}^{-1}\gamma_{a}r - \mathbb{P}^{-1}\gamma_{a}s \rbrace\\
\end{align*}
where $\mathbb{P} \defeq \begin{pmatrix} p & 0 \\ 0 & p \end{pmatrix}$. Since $F(x,y) \in \mathcal{A}^{\lambda_{k}}(Y_{e})$, we have $\mathbb{P}^{-1}F(x,y) = F(x,y)/\mathrm{N}_{F/\QQ}(p)^{k}$,
\begin{align*}
\int_{Y_{e}}F(x,y)d\Phi_{\widetilde{L}_{2}}\lbrace r - s \rbrace & = \frac{1}{\mathrm{N}_{F/\QQ}(p)^{k}}\sum\limits_{a\:\mathrm{mod}\:p} \int_{Y_{e}}\gamma_{a}\cdot F(x,y)d\Phi_{\widetilde{L}_{1}}\lbrace \gamma_{a}r - \gamma_{a}s \rbrace\\
 & = \frac{1}{\mathrm{N}_{F/\QQ}(p)^{k}}\int_{Y_{e}}F(x,y)dU_{p}.\Phi_{\widetilde{L}_{1}}\lbrace r - s \rbrace \\
 & = \frac{a_{\p}(\cF_{k})}{\mathrm{N}_{F/\QQ}(p)^{k}}\int_{Y_{e}}F(x,y)d\Phi_{\widetilde{L}_{1}}\lbrace r - s \rbrace
\end{align*}
where the penultimate equality follows from the definition of the $U_{p}$ operator acting on the $\mathbb{D}_{U}^{\dagger}$-valued modular symbols $\lbrace \Phi_{\widetilde{L}} \rbrace$.
\end{proof} 
For each $k \neq k_{0}$ such that $\lambda_{k} \in U(L)$, we will denote by 
\begin{equation} \label{eqn:period} \phi_{k}^{\#} \defeq \frac{\widetilde{\phi}_{k}^{\#}}{\Omega_{k}^{\#}} \in \mathrm{Symb}_{\Gamma_{0}(\mathcal{M})}(V_{k,k}(E_{k})^{\vee})
\end{equation}
for the \emph{Bianchi modular symbol} attached to $\mathcal{F}_{k}^{\#}$ in \cite{Wil17}, where $\Omega_{k}^{\#} \in \mathbb{C}^{\times}$ is some complex period and $E_{k}/\QQ$ is some number field. Similarly, let $\phi_{k} \in \mathrm{Symb}_{\Gamma_{0}(\mathcal{N})}(V_{k,k}(E_{k})^{\vee})$ denote the Bianchi modular symbol attached to $\mathcal{F}_{k}$. Note that, similar to (\ref{eqn:pstabilisation}), we have a relation between the modular symbols $\phi_{k}$ and $\phi_{k}^{\#}$,\footnote{We thank Luis Santiago Palacios for providing us this calculation.}
\begin{equation} \label{eqn:pstabilisationmodularsymbols1}
\phi_{k} = \phi_{k}^{\#} - \frac{1}{a_{\p}(\cF_{k})}\phi_{k}^{\#}\mid_{\begin{pmatrix}p & 0\\ 0 & 1 \end{pmatrix}}.  
\end{equation}
In other words, for all $r, s \in \PP^{1}(F)$ and $P(x,y) \in V_{k,k}(L)$
\begin{equation} \label{eqn:pstabilisationmodularsymbols2}
\phi_{k}\lbrace r - s \rbrace(P(x,y)) = \phi_{k}^{\#}\lbrace r - s \rbrace(P(x,y)) - \frac{1}{a_{\p}(\cF_{k})}\phi_{k}^{\#}\lbrace pr - ps \rbrace(P(x,py)). 
\end{equation}
\begin{lemma} \label{lem:pstabilisationmodularsymbols}
For all $\lambda_{k} \in U(L)$ and all $P(x,y) \in V_{k,k}(L)$, we have
\begin{equation} \label{eqn:pstabilisationmodularsymbols}
\int\limits_{\widetilde{L}_{*}'}\widetilde{P}(x,y) d\Phi_{\infty}\lbrace r - s \rbrace(x,y) = C(k)\left(1 - \frac{\mathrm{N}_{F/\QQ}(p)^{k}}{a_{\p}(\mathcal{F}_{k})^{2}}\right)(\phi_{k}^{\#})\lbrace r - s \rbrace(P)
\end{equation} 
\end{lemma}
\begin{proof}
The proof is inspired from \cite[Proposition 2.4]{BD09}. Note that we can write the set of primitive vectors of the lattice $\mathscr{L}_{*}'$ as a disjoint union
\[ \mathscr{L}_{*}' = (\cO_{F_{\p}}^{\times} \oplus \cO_{F_{\p}}) \bigsqcup (p\cO_{F_{\p}} \oplus \cO_{F_{\p}}^{\times}). \]
Let $\theta \in \cR$ be any matrix of determinant $p$ such that
\[ \theta(\mathscr{L}_{*}) = \mathscr{L}_{\infty},\mbox{     \&     } \theta(\mathscr{L}_{\infty}) = p\mathscr{L}_{*}.\]
Then,
\begin{equation} \label{eqn:theta} 
\theta(p\cO_{F_{\p}} \oplus \cO_{F_{\p}}^{\times}) = p(\cO_{F_{\p}}^{\times} \oplus \cO_{F_{\p}}). 
\end{equation} 
We write 
\[ \int\limits_{\widetilde{L}_{*}'}\widetilde{P}(x,y) d\Phi_{\infty}\lbrace r - s \rbrace(x,y) = I_{1} + I_{2} \]
where 
\begin{align*} 
 I_{1} & = \int\limits_{\mu(\cO_{F})\backslash(\cO_{F_{\p}}^{\times} \oplus \cO_{F_{\p}})} \widetilde{P}(x,y) d\Phi_{\infty}\lbrace r - s \rbrace(x,y) = C(k)\phi_{k}\lbrace r - s \rbrace(P) 
\end{align*}
and
\begin{align*} 
I_{2} & =  \int\limits_{\mu(\cO_{F})\backslash(p\cO_{F_{\p}} \oplus \cO_{F_{\p}}^{\times})}\widetilde{P}(x,y)d\Phi_{\widetilde{L}_{*}}\lbrace r - s \rbrace(x,y)\\
 & = \int\limits_{\theta(\mu(\cO_{F})\backslash(p\cO_{F_{\p}} \oplus \cO_{F_{\p}}^{\times}))}\theta\cdot\widetilde{P}(x,y)d\Phi_{\widetilde{L}_{\infty}}\lbrace \theta r - \theta s \rbrace(x,y) 
\end{align*}
where the second equality follows from Proposition~\ref{prop:latticesymbols} upon noting that $\theta(\widetilde{L}_{*}) = \widetilde{L}_{\infty}$. 
Let $\PP = \begin{pmatrix} p & 0\\0 & p\end{pmatrix}$ as before. Then by (\ref{eqn:theta}) above, we have
\begin{align*} 
I_{2} & = \int\limits_{\mu(\cO_{F})\backslash(\cO_{F_{\p}}^{\times} \oplus \cO_{F_{\p}})}(\PP^{-1}\theta\cdot\widetilde{P}(x,y))d\Phi_{\PP^{-1}\widetilde{L}_{\infty}}\lbrace \PP^{-1}\theta r - \PP^{-1}\theta s \rbrace(x,y) \\
 & = \frac{1}{\mathrm{N}_{F/\QQ}(p)^{k}}\int\limits_{\mu(\cO_{F})\backslash(\cO_{F_{\p}}^{\times} \oplus \cO_{F_{\p}})}(\theta\cdot\widetilde{P}(x,y)) d\Phi_{\frac{1}{p}\cdot \widetilde{L}_{\infty}} \lbrace \theta r - \theta s \rbrace(x,y).
\end{align*}
Since $\mathscr{L}_{*} \subset \frac{1}{p}\mathscr{L}_{\infty}$ is a sub-lattice of index $\mathrm{N}_{F/\QQ}(p)$, by Lemma~\ref{lem:indexpmodularsymbols}, we get
\begin{align*}
I_{2} & = \frac{1}{a_{\p}(\cF_{k})}\int\limits_{\mu(\cO_{F})\backslash(\cO_{F_{\p}}^{\times} \oplus \cO_{F_{\p}})}(\theta\cdot\widetilde{P}(x,y)) d\Phi_{\widetilde{L}_{*}} \lbrace \theta r - \theta s \rbrace(x,y) \\
& =  \frac{C(k)}{a_{\p}(\cF_{k})} \phi_{k}\lbrace \theta r - \theta s \rbrace (\theta\cdot P(x,y)).
\end{align*}
where the final equality follows from (\ref{eqn:specialization}) above. Note that (\ref{eqn:pstabilisationmodularsymbols2}) above may also be re-written in terms of the matrix $\theta$ as 
\begin{equation} \label{eqn:pstabilisationmodularsymbols3}
\phi_{k}\lbrace r - s \rbrace(P(x,y)) = \phi_{k}^{\#}\lbrace r - s \rbrace(P(x,y)) - \frac{1}{a_{\p}(\cF_{k})}\phi_{k}^{\#}\lbrace \theta r - \theta s \rbrace(\theta\cdot P(x,y)).
\end{equation}
In particular, we may simplify the two integrals considered above as
\begin{align*}
I_{1} & = C(k)\left(\phi_{k}^{\#}\lbrace r - s \rbrace(P(x,y)) - \frac{1}{a_{\p}(\cF_{k})}\phi_{k}^{\#}\lbrace \theta r - \theta s \rbrace(\theta\cdot P(x,y))\right)
\end{align*}
and 
\begin{align*}
I_{2} & = \frac{C(k)}{a_{\p}(\cF_{k})}\left(\phi_{k}^{\#}\lbrace \theta r - \theta s \rbrace(\theta\cdot P(x,y)) - \frac{1}{a_{\p}(\cF_{k})}\phi_{k}^{\#}\lbrace \theta^{2} r - \theta^{2} s \rbrace(\theta^{2}\cdot P(x,y))\right).
\end{align*}
Note that we can find a $\gamma \in \Gamma_{0}(\cM)$ such that $\theta^{2} = p\gamma$. Since $\phi_{k}^{\#}$ is $\Gamma_{0}(\cM)$-invariant, it follows that (again by the homogeneity of $P(x,y)$)
\[ \phi_{k}^{\#}\lbrace \theta^{2} r - \theta^{2} s \rbrace(\theta^{2}P(x,y)) = \mathrm{N}_{F/\QQ}(p)^{k}\cdot\phi_{k}^{\#} \lbrace r - s \rbrace (P(x,y)). \]
The Lemma now follows since
\[ I_{1} + I_{2} = C(k)\left( 1 - \frac{\mathrm{N}_{F/\QQ}(p)^{k}}{a_{\p}(\cF_{k})^{2}} \right) (\phi_{k}^{\#})\lbrace r - s \rbrace (P(x,y)) \]
\end{proof}
Let $\pi$ be the projection of $\mathscr{W}$ onto $\PP^{1}(F_{\p})$ as before. For every $\cO_{F_{\p}}$-lattice $\mathscr{L}$, we can define a \emph{pushforward modular symbol} $\pi_{*}(\Phi_{\widetilde{L}}) \in \mathbf{MS}_{\Gamma}(L)$ given by
\begin{equation} \label{eqn:pushforward}
\int_{\PP^{1}(F_{\p})}F(t)d\pi_{*}(\Phi_{\widetilde{L}})\lbrace r - s \rbrace(t) \defeq |\mathscr{L}|^{k_{0}/2}\int_{\mathscr{Y}}\widetilde{F}(x,y)d\Phi_{\widetilde{L}}\lbrace r - s \rbrace(x,y)
\end{equation}
where $F(t) \in \mathcal{A}_{k_{0}}(\mathbb{P}^{1}(F_{\p}), L)$ is a locally analytic function on $\mathbb{P}^{1}(F_{\p})$ except for a pole of order at most $k_{0}$ at $\infty$ and $\widetilde{F}(x,y) \defeq y^{k_{0}}F(x/y) \in \cA^{\lambda_{k_{0}}}(\mathscr{Y})$ is a locally analytic `homogeneous of weight $\lambda_{k_{0}}$' function on $\mathscr{Y}$. Recall that there exists a unique harmonic modular symbol $\Phi_{\cF}^{\mathrm{har}} \in \mathbf{MS}_{\Gamma}(L)$ that lifts $\phi_{k_{0}} = \phi_{k_{0}}^{\#}$. The following result relates the pushforward modular symbols with the harmonic modular symbol $\Phi_{\cF}^{\mathrm{har}}$.

\begin{corollary} \label{cor:pushforwardtoharmonic}
For all lattices $\mathscr{L}$, $\pi_{*}(\Phi_{\widetilde{L}}) = \Phi_{\cF}^{\mathrm{har}} \in \mathbf{MS}_{\Gamma}(L)$
\end{corollary}

\begin{proof}
The Corollary follows by a repeated application of Lemma~\ref{lem:indexpmodularsymbols} combined with the weight $\lambda_{k}$-specialization of (\ref{eqn:specialization}). Compare with \cite[Corollary 4.7]{Sev12} and \cite[Proposition 2.12]{BD07}. Since our form $\cF_{k_{0}}$ is split--multiplicative (See Remark~\ref{rem:Up-eigenvalue}), we don't need to restrict to even vertices as in \cite{Sev12}.
\end{proof}
\begin{definition} \label{def:semidefiniteintegral}
The semidefinite integral attached to $r, s \in \mathbb{P}^{1}(F)$; $\tau \in \cH_{\p}^{\mathrm{ur}}$ and $P(x,y) \in V_{k_{0},k_{0}}$ is defined as
\begin{align*} \label{eqn:semidefinitedefinition}
 \int_{r}^{s}\int^{\tau}P(x,y)\omega_{\cF} & \defeq |\mathscr{L}_{\tau}|^{k_{0}/2} \int_{\mathscr{Y}}\widetilde{P}(x,y)\log_{p}\left(\mathrm{N}_{F/\QQ}(x-\tau y)\right)d\Phi_{\widetilde{L_{\tau}}}\lbrace{r - s\rbrace}(x,y)
 \end{align*}
where $[\mathscr{L}_{\tau}] = \mathrm{red}_{\p}(\tau)$.
\end{definition}
\begin{lemma} \label{lem:derivativelemma}
Let $\alpha(\lambda_{\kappa}) \in \cO(U)$. For all $e \in \cE(\cT)$, $\tau \in \cH_{\p}$, $r,s \in \PP^{1}(F)$ and $P(x,y) \in V_{k_{0},k_{0}}$, we have
\begin{multline*} \int_{Y_{e}} \widetilde{P}(x,y) \log_{p}\left(\mathrm{N}_{F/\QQ}(x-\tau y)\right)d\alpha\Phi_{\widetilde{L_{\tau}}}\lbrace r - s \rbrace (x,y) \\ = \alpha'(\lambda_{k_{0}})|\mathscr{L}_{\tau}|^{-k_{0}/2}\int_{U_{e}}P(t)d\Phi_{\cF}^{\mathrm{har}}\lbrace r - s \rbrace (t) \\
+ \alpha(\lambda_{k_{0}})\int_{Y_{e}} \widetilde{P}(x,y) \log_{p}\left(\mathrm{N}_{F/\QQ}(x-\tau y)\right)d\Phi_{\widetilde{L_{\tau}}}\lbrace r - s \rbrace (x,y).
\end{multline*}
\end{lemma}
\begin{proof}
By Definition~\ref{def:derivative},

\begin{align*} 
& \int_{Y_{e}} \widetilde{P}(x,y) \log_{p}\left(\mathrm{N}_{F/\QQ}(x-\tau y)\right)d\alpha(\lambda_{\kappa})\Phi_{\widetilde{L_{\tau}}}\lbrace r - s \rbrace (x,y) \\
& = \frac{d}{d\lambda_{\kappa}}\left(\alpha(\lambda_{\kappa})\int_{Y_{e}} \widetilde{P}(x,y)\left(\mathrm{N}_{F/\QQ}(x-\tau y)\right)^{\lambda_{\kappa}-\lambda_{k_{0}}}d\Phi_{\widetilde{L_{\tau}}}\lbrace r - s \rbrace (x,y)\right)_{\lambda_{\kappa} = \lambda_{k_{0}}}\\
& = \alpha'(\lambda_{k_{0}})\int_{Y_{e}} \widetilde{P}(x,y)d\Phi_{\widetilde{L_{\tau}}}\lbrace r - s \rbrace (x,y)\\
& + \alpha(\lambda_{k_{0}})\int_{Y_{e}} \widetilde{P}(x,y) \log_{p}\left(\mathrm{N}_{F/\QQ}(x-\tau y)\right)d\Phi_{\widetilde{L_{\tau}}}\lbrace r - s \rbrace (x,y) \\
& = \alpha'(\lambda_{k_{0}})|\mathscr{L}_{\tau}|^{-k_{0}/2}\int_{U_{e}}P(t)d\Phi_{\cF}^{\mathrm{har}}\lbrace r - s \rbrace (t)\\
& + \alpha(\lambda_{k_{0}})\int_{Y_{e}} \widetilde{P}(x,y) \log_{p}\left(\mathrm{N}_{F/\QQ}(x-\tau y)\right)d\Phi_{\widetilde{L_{\tau}}}\lbrace r - s \rbrace (x,y).
\end{align*}

\end{proof}

\begin{proposition} \label{prop:definiteord+log}
For $r, s \in \PP^{1}(F)$ and $\tau_{1}, \tau_{2} \in \cH_{\p}^{\mathrm{ur}}$, we have
\begin{multline*}
 \int_{r}^{s}\int^{\tau_{2}}P(x,y)\omega_{\cF} - \int_{r}^{s}\int^{\tau_{1}}P(x,y)\omega_{\cF} = \\
 \int_{r}^{s}\int_{\tau_{1}}^{\tau_{2}}P(x,y)\omega^{\mathrm{log}_{p}}_{\cF} + 2\frac{a_{\p}'(\cF)}{a_{\p}(\cF)}\int_{r}^{s}\int_{\tau_{1}}^{\tau_{2}}P(x,y)\omega^{\mathrm{ord}_{\p}}_{\cF}
 \end{multline*}
 where $a_{\p}'(\cF)$ is defined as the derivative of the $U_{\p}$-eigenvalue of the specialisations of the Coleman family $\mathbfcal{F}(g)$ at $\lambda_{k_{0}}$. 
 \end{proposition}
 \begin{proof} 
We may suppose without loss of generality that $\mathscr{L}_{\tau_{2}} \subset \mathscr{L}_{\tau_{1}}$ and that $[\mathscr{L}_{\tau_{1}}:\mathscr{L}_{\tau_{2}}] = \mathrm{N}_{F/\QQ}(p)$. We denote by $e \in \cE(\cT)$ to be the ordered edge between $\mathscr{L}_{\tau_{1}}$ and $\mathscr{L}_{\tau_{2}}$.
By Definition~\ref{def:semidefiniteintegral} above, we have
\begin{align*}
 \int_{r}^{s}\int^{\tau_{2}}P(x,y)\omega_{\cF} & -  \int_{r}^{s}\int^{\tau_{1}}P(x,y)\omega_{\cF} \\ 
 & =  |\mathscr{L}_{\tau_{2}}|^{k_{0}/2} \int_{\mathscr{Y}}\widetilde{P}(x,y)\log_{p}\left( \mathrm{N}_{F/\QQ}(x-\tau_{2} y)\right) d\Phi_{\widetilde{L_{\tau_{2}}}}\lbrace{r - s\rbrace}(x,y) \\  
 & -  |\mathscr{L}_{\tau_{1}}|^{k_{0}/2} \int_{\mathscr{Y}}\widetilde{P}(x,y)\log_{p}\left( \mathrm{N}_{F/\QQ}(x-\tau_{1} y)\right) d\Phi_{\widetilde{L_{\tau_{1}}}}\lbrace{r - s\rbrace}(x,y) \\
 & =  |\mathscr{L}_{\tau_{2}}|^{k_{0}/2} \int_{\mathscr{Y}}\widetilde{P}(x,y)\log_{p}\left( \mathrm{N}_{F/\QQ}(x-\tau_{2} y)\right) d\Phi_{\widetilde{L_{\tau_{2}}}}\lbrace{r - s\rbrace}(x,y) \\  
 & -  |\mathscr{L}_{\tau_{2}}|^{k_{0}/2} \int_{\mathscr{Y}}\widetilde{P}(x,y)\log_{p}\left( \mathrm{N}_{F/\QQ}(x-\tau_{1} y)\right) d\Phi_{\widetilde{L_{\tau_{2}}}}\lbrace{r - s\rbrace}(x,y) \\
& +  |\mathscr{L}_{\tau_{2}}|^{k_{0}/2} \int_{\mathscr{Y}}\widetilde{P}(x,y)\log_{p}\left( \mathrm{N}_{F/\QQ}(x-\tau_{1} y)\right) d\Phi_{\widetilde{L_{\tau_{2}}}}\lbrace{r - s\rbrace}(x,y) \\
& -  |\mathscr{L}_{\tau_{1}}|^{k_{0}/2} \int_{\mathscr{Y}}\widetilde{P}(x,y)\log_{p}\left( \mathrm{N}_{F/\QQ}(x-\tau_{1} y)\right) d\Phi_{\widetilde{L_{\tau_{1}}}}\lbrace{r - s\rbrace}(x,y) \\
& = I_{\ord} + I_{\log}
\end{align*}
 where 
\begin{align*}
 I_{\ord} & \defeq |\mathscr{L}_{\tau_{2}}|^{k_{0}/2} \int_{\mathscr{Y}}\widetilde{P}(x,y)\log_{p}\left( \mathrm{N}_{F/\QQ}(x-\tau_{1} y)\right) d\Phi_{\widetilde{L_{\tau_{2}}}}\lbrace{r - s\rbrace}(x,y) \\
 & - |\mathscr{L}_{\tau_{1}}|^{k_{0}/2} \int_{\mathscr{Y}}\widetilde{P}(x,y)\log_{p}\left( \mathrm{N}_{F/\QQ}(x-\tau_{1} y)\right) d\Phi_{\widetilde{L_{\tau_{1}}}}\lbrace{r - s\rbrace}(x,y).
 \end{align*} 
and  
 \begin{align*}
 I_{\log} & \defeq |\mathscr{L}_{\tau_{2}}|^{k_{0}/2} \int_{\mathscr{Y}}\widetilde{P}(x,y)\log_{p}\circ\mathrm{N}_{F/\QQ}\left(\frac{x-\tau_{2} y}{x-\tau_{1} y}\right) d\Phi_{\widetilde{L_{\tau_{2}}}}\lbrace{r - s\rbrace}(x,y). \\
 \end{align*}
By Corollary~\ref{cor:pushforwardtoharmonic} above, we have 
\begin{align*}  
I_{\log} & = \int_{\PP^{1}_{\p}} P(t)\log_{p}\circ \mathrm{N}_{F/\QQ} \left(\frac{t - \tau_{2}}{t - \tau_{1}}\right) d\Phi_{\cF}^{\mathrm{har}}\lbrace r - s \rbrace (t) \\
 & = \int_{r}^{s}\int_{\tau_{1}}^{\tau_{2}}P(x,y)\omega^{\mathrm{log}_{p}}_{\cF}.
\end{align*}
We express $I_{\ord}$ as the sum of two contributions $I_{e}$ and $I_{\overline{e}}$ obtained by evaluating the integrals over disjoint subsets $Y_{e}$ and $Y_{\overline{e}}$ of $\mathscr{Y}$ respectively. By Lemma~\ref{lem:indexpmodularsymbols},
\begin{align*}
I_{e} & = |\mathscr{L}_{\tau_{2}}|^{k_{0}/2}\frac{a_{\p}(\cF_{k_{0}})}{\mathrm{N}_{F/\QQ}(p)^{k_{0}}} \int_{{Y_{e}}}\widetilde{P}(x,y)\log_{p}\left( \mathrm{N}_{F/\QQ}(x-\tau_{1} y)\right) d\Phi_{\widetilde{L_{\tau_{1}}}}\lbrace{r - s\rbrace}(x,y) \\
 & - |\mathscr{L}_{\tau_{1}}|^{k_{0}/2} \int_{{Y_{e}}}\widetilde{P}(x,y)\log_{p}\left( \mathrm{N}_{F/\QQ}(x-\tau_{1} y)\right) d\Phi_{\widetilde{L_{\tau_{1}}}}\lbrace{r - s\rbrace}(x,y) \\
 & = |\mathscr{L}_{\tau_{1}}|^{k_{0}/2}\left(\frac{a_{\p}(\cF_{k_{0}})}{\mathrm{N}_{F/\QQ}(p)^{k_{0}/2}}-1\right) \int_{{Y_{e}}}\widetilde{P}(x,y)\log_{p}\left( \mathrm{N}_{F/\QQ}(x-\tau_{1} y)\right) d\Phi_{\widetilde{L_{\tau_{1}}}}\lbrace{r - s\rbrace}(x,y).
\end{align*}
Applying Lemma~\ref{lem:derivativelemma} with $\alpha(\lambda_{\kappa}) \defeq a_{\p}(\cF_{\lambda_{\kappa}})/\mathrm{N}_{F/\QQ}(p)^{k_{0}/2}-1$ upon noting that $\alpha(\lambda_{k_{0}}) = 0$, we get
\[ I_{e} = \frac{a_{\p}'(\cF)}{\mathrm{N}_{F/\QQ}(p)^{k_{0}/2}}\sum\limits_{e : \textup{red}_{\mathfrak{p}}(\tau_1) \rightarrow \textup{red}_{\mathfrak{p}}(\tau_2)}\int_{U_{e}}P(t)d\Phi_{\cF}^{\mathrm{har}}\lbrace r - s \rbrace (t) \] 
Similarly for $\overline{e}$ - the ordered edge between $(1/p)\mathscr{L}_{2}$ and $\mathscr{L}_{1}$, noting that $[(1/p)\mathscr{L}_{2}:\mathscr{L}_{1}] = \mathrm{N}_{F/\QQ}(p)$, the same computation using Lemma~\ref{lem:indexpmodularsymbols} gives
\begin{align*}
I_{\overline{e}} & = |\mathscr{L}_{\tau_{2}}|^{k_{0}/2}\frac{a_{\p}(\cF_{k_{0}})}{\mathrm{N}_{F/\QQ}(p)^{k_{0}}} \int_{Y_{\overline{e}}}\widetilde{P}(x,y)\log_{p}\left( \mathrm{N}_{F/\QQ}(x-\tau_{1} y)\right) d\Phi_{\widetilde{L_{\tau_{1}}}}\lbrace{r - s\rbrace}(x,y) \\
 & - |\mathscr{L}_{\tau_{1}}|^{k_{0}/2} \int_{Y_{\overline{e}}}\widetilde{P}(x,y)\log_{p}\left( \mathrm{N}_{F/\QQ}(x-\tau_{1} y)\right) d\Phi_{\widetilde{L_{\tau_{1}}}}\lbrace{r - s\rbrace}(x,y) \\
 = & -|\frac{1}{p}\mathscr{L}_{\tau_{2}}|^{k_{0}/2}\left(\frac{a_{\p}(\cF_{k_{0}})}{\mathrm{N}_{F/\QQ}(p)^{k_{0}/2}}-1\right) \int_{Y_{\overline{e}}}\widetilde{P}(x,y)\log_{p}\left( \mathrm{N}_{F/\QQ}(x-\tau_{1} y)\right) d\Phi_{\widetilde{(1/p)L_{\tau_{2}}}}\lbrace{r - s\rbrace}(x,y)\\
 & = -\frac{a_{\p}'(\cF)}{\mathrm{N}_{F/\QQ}(p)^{k_{0}/2}}\sum\limits_{e : \textup{red}_{\mathfrak{p}}(\tau_1) \rightarrow \textup{red}_{\mathfrak{p}}(\tau_2)}\int_{U_{\overline{e}}}P(t)d\Phi_{\cF}^{\mathrm{har}}\lbrace r - s \rbrace (t).
\end{align*}
Since $\PP^{1}_{\p}$ is a disjoint union of $U_{e}$ and $U_{\overline{e}}$, we have 
\begin{multline*} \int_{U_{\overline{e}}}P(t)d\Phi_{\cF}^{\mathrm{har}}\lbrace r - s \rbrace (t) + \int_{U_{e}}P(t)d\Phi_{\cF}^{\mathrm{har}}\lbrace r - s \rbrace (t) 
 = \int_{\PP^{1}_{\p}}P(t)d\Phi_{\cF}^{\mathrm{har}}\lbrace r - s \rbrace (t) = 0
 \end{multline*}
where the final equality follows from \cite[Proposition 5.8(i)]{BW19}. Hence
\begin{align*} I_{\ord} = I_{e} + I_{\overline{e}} & =  2\frac{a_{\p}'(\cF)}{\mathrm{N}_{F/\QQ}(p)^{k_{0}/2}}\sum\limits_{e : \textup{red}_{\mathfrak{p}}(\tau_1) \rightarrow \textup{red}_{\mathfrak{p}}(\tau_2)}\int_{U_{\overline{e}}}P(t)d\Phi_{\cF}^{\mathrm{har}}\lbrace r - s \rbrace (t) \\
& = 2\frac{a_{\p}'(\cF)}{\mathrm{N}_{F/\QQ}(p)^{k_{0}/2}}\int_{r}^{s}\int_{\tau_{1}}^{\tau_{2}}P(x,y)\omega^{\mathrm{ord}_{\p}}_{\cF}
\end{align*} 
by Definition~\ref{defn:doubleintegrals}(iii). Since $a_{\p}(\cF) = \mathrm{N}_{F/\QQ}(p)^{k_{0}/2}$, we get that
\begin{align*}
\int_{r}^{s}\int^{\tau_{2}}P(x,y)\omega_{\cF} & - \int_{r}^{s}\int^{\tau_{1}}P(x,y)\omega_{\cF} \\
& = I_{\log} + I_{\ord} \\
& = \int_{r}^{s}\int_{\tau_{1}}^{\tau_{2}}P(x,y)\omega^{\mathrm{log}_{p}}_{\cF} + 2\frac{a_{\p}'(\cF)}{a_{\p}(\cF)}\int_{r}^{s}\int_{\tau_{1}}^{\tau_{2}}P(x,y)\omega^{\mathrm{ord}_{\p}}_{\cF}.
\end{align*}
 \end{proof}
\begin{proposition} \label{prop:basechangeLinv}
The $\cL$-invariant attached to the Bianchi cusp form $\mathcal{F} \in S_{\underline{k_{0}}+2}(U_{0}(\mathcal{N}))^{\mathrm{new}}$ is given by
\[ \cL_{\p}^{\mathrm{BW}} = -2\frac{a_{\p}'(\cF)}{a_{\p}(\cF)} \]
\end{proposition}
\begin{proof}
Recall that the Bianchi cuspform $\cF$ is the base--change to $F$ of the elliptic cuspidal newform $f \in S_{k_{0}+2}(\Gamma_{0}(N))$. By \cite[Lemma 4.4]{VW19}, we know that $\cL_{\p}^{\mathrm{BW}} = 2\cL_{p}(f)$. Moreover, since $p$ is inert in $F$, we know that $a_{\p}(\cF) = a_{p}(f)^{2}$ where $a_{p}$ is the $U_{p}$-eigenvalue of the form $f$. The result follows since it is known that (See for eg. \cite[Theorem 4.11]{Sev12})
\[ \cL_{p}(f) = -2\frac{a_{p}'(f)}{a_{p}(f)} \]
 where $a_{p}'(f)$ is defined as the derivative of the $U_{p}$-eigenvalue of the specialisations of the Coleman family $\mathbf{f}(q)$ at $\lambda_{k_{0}}$. 
\end{proof}
\begin{corollary} \label{cor:semidefinitebranch} The semidefinite integral $\int_{r}^{s}\int^{\tau}P(x,y)\omega_{\cF}$ satisfies
\[ \int_{r}^{s}\int^{\tau_{2}}P(x,y)\omega_{\cF} - \int_{r}^{s}\int^{\tau_{1}}P(x,y)\omega_{\cF} = \int_{r}^{s}\int^{\tau_{2}}_{\tau_{1}}P(x,y)\omega_{\cF}^{\mathrm{log}_{\cL_{p}(f)}} \]
where $\log_{\cL_{p}(f)}$ is the branch of the $p$-adic logarithm such that $\log_{\cL_{p}(f)}(p) = -\cL_{p}(f)$. 

\end{corollary}
\begin{proof}
By Propositions~\ref{prop:definiteord+log} and \ref{prop:basechangeLinv} above, we have 
\begin{multline*} \int_{r}^{s}\int^{\tau_{2}}P(x,y)\omega_{\cF} - \int_{r}^{s}\int^{\tau_{1}}P(x,y)\omega_{\cF} = \\ \int_{r}^{s}\int^{\tau_{2}}_{\tau_{1}}P(x,y)\omega_{\cF}^{\mathrm{log}_{p}} - \cL_{\p}^{\mathrm{BW}}\int_{r}^{s}\int^{\tau_{2}}_{\tau_{1}}P(x,y)\omega_{\cF}^{\mathrm{ord}_{\p}}. \end{multline*}
By Remark~\ref{rem:galoisequivarianceofL-invs} above, we have 
\[\Phi_{\cF}^{\log_{p}} - \cL_{\p}^{\mathrm{BW}}\Phi_{\cF}^{\ord_{\p}} = \Phi_{\cF}^{\log_{\cL_{p}(f)}} \]
or in terms of double integrals, we have 
\begin{equation} \label{eqn:doubleintegrals} \int_{r}^{s}\int^{\tau_{2}}_{\tau_{1}}P(x,y)\omega_{\cF}^{\mathrm{log}_{p}} - \cL_{\p}^{\mathrm{BW}}\int_{r}^{s}\int^{\tau_{2}}_{\tau_{1}}P(x,y)\omega_{\cF}^{\mathrm{ord}_{\p}} = \int_{r}^{s}\int^{\tau_{2}}_{\tau_{1}}P(x,y)\omega_{\cF}^{\mathrm{log}_{\cL_{p}(f)}} 
\end{equation}

\end{proof}
\begin{theorem} \label{thm:padicAJimage}
Let $\Psi \in \mathrm{Emb}(\cO, \cR)$ and let $r \in \mathbb{P}^{1}(F)$ be an arbitrary base point. Then
\[ \int_{r}^{\gamma_{\Psi}r}\int^{\tau_{\Psi}}P_{\Psi}^{k_{0}/2}\omega_{\cF} = \bigg( \sqrt{\mathrm{N}_{F/\QQ}(\cD_{K/F})} \bigg)^{k_{0}/2}\mathrm{log}\:\Phi^{\mathrm{AJ}}(\mathrm{D}_{[\Psi]})(\Phi_{\cF}^{\mathrm{har}})\]
\end{theorem}

\begin{proof} By Corollary~\ref{cor:semidefinitebranch},  
 
\begin{multline*}
\int_{r}^{\gamma_{\Psi}r}\int^{\tau_{\Psi}}P_{\Psi}^{k_{0}/2}\omega_{\cF} - \int_{r}^{\gamma_{\Psi}r}\int^{\tau_{\Psi}^{\theta}}P_{\Psi}^{k_{0}/2}\omega_{\cF} = \int_{r}^{\gamma_{\Psi}r}\int^{\tau_{\Psi}}_{\tau_{\Psi}^{\theta}}P_{\Psi}^{k_{0}/2}\omega_{\cF}^{\mathrm{log}_{\cL_{p}(f)}} \\
 = \Phi_{\cF}^{\mathrm{log}_{\cL_{p}(f)}}\left( \lbrace \gamma_{\Psi}r - r \rbrace \otimes \lbrace \tau_{\Psi} - \tau_{\Psi}^{\theta} \rbrace \otimes P_{\Psi}^{k_{0}/2} \right)
 \end{multline*} 
which by Remark~\ref{rem:galoisequivarianceofL-invs} is equivalent to 
\begin{equation}\label{eqn:Section4eqn1} \mathrm{log}\;\Phi_{\cF}\left( \lbrace \gamma_{\Psi}r - r \rbrace \otimes \lbrace \tau_{\Psi} - \tau_{\Psi}^{\theta} \rbrace \otimes P_{\Psi}^{k_{0}/2} \right) \in  \mathbf{MS}_{\Gamma}(L)_{(\cF)}^{\vee}. 
\end{equation}  
Then, 
\begin{equation} \label{eqn:Section4eqn2}
\mathrm{log}\;\Phi^{\mathrm{AJ}}\left( \lbrace \gamma_{\Psi}r - r \rbrace \otimes \lbrace \tau_{\Psi} \rbrace \otimes P_{\Psi}^{k_{0}/2} \right) \in  \mathbf{MS}_{\Gamma}(L)_{(\cF)}^{\vee} 
\end{equation}
is a lift of (\ref{eqn:Section4eqn1}) above as in Remark~\ref{rem:logAJ}. The theorem now follows Theorem~\ref{thm:padicabeljacobiimageofdarmoncycles} (See \cite[Theorem 6.5]{VW19} for more details) since the $\p$-adic Abel--Jacobi image of Stark--Heegner cycles are independent of the choice of the $\p$-adic Abel--Jacobi map.
\end{proof}



\section{$p$-adic $L$-functions} \label{sec:p-adicLfns}
In this section, we briefly recall the construction and interpolation properties of the various $p$-adic $L$-functions that are of relevance to this article.
\subsection{Bellaiche-Stevens-Mazur-Kitagawa $p$-adic $L$-functions} \label{subsec:smk}
Let $\chi$ be a quadratic Dirichlet character of conductor $c(\chi)$ relatively prime to $N$. For $\lambda_{k} \in U$ a classical weight, which we view as a classical weight on the weight space $\mathcal{W}_{\QQ} \defeq \mathrm{Hom}_{\mathrm{cts}}(\ZZ_{p}^{\times}, L^{\times})$ when the context is clear, let $L(f_{k}^{\#}, \chi, s)$ be the analytic continuation of the $\chi$-twisted Dirichlet $L$-series $\sum\limits_{n \geq 1}\frac{\chi(n)a_{n}(f_{k}^{\#})}{n^{s}}$ for $s \in \mathbb{C}$. We denote by $\Omega_{\infty}(f_{k}^{\#})^{\pm} \in \mathbb{C}^{\times}$ to be the Shimura periods that satisfy
\[  \Omega_{\infty}(f_{k}^{\#})^{+}\cdot \Omega_{\infty}(f_{k}^{\#})^{-} = \langle f_{k}^{\#}, f_{k}^{\#}\rangle_{Mp^{r(k)}} \]
where $r(k_{0}) = 1$ and zero otherwise, and $\langle,\rangle$ is the usual Petersson norm. 
Then the \emph{algebraic part} of the special $L$-values, for $1 \leq j \leq k+1$, are given by 
\[ L^{\mathrm{alg}}(f_{k}^{\#}, \chi, j) = \frac{(j-1)!\cdot\tau(\chi)\cdot c(\chi)^{j}\cdot L(f_{k}^{\#}, \chi, j)}{(-2\pi i)^{j-1}\cdot \Omega_{\infty}(f_{k}^{\#})^{\pm}} \in \QQ(f_{k}^{\#}, \chi) \]
where $\tau(\chi) = \sum_{a \in (\ZZ/c(\chi)\ZZ)^*} \chi(a)\cdot\zeta^{a}_{c(\chi)}$ is the Gauss sum of $\chi$. The sign of the Shimura period is chosen according to the sign of 
\[ \chi(-1) = \pm(-1)^{j-1} \]
\begin{theorem}[Bellaiche, Stevens, Mazur, Kitagawa] \label{thm:mazurkitagawa1} There exists $p$-adic periods $\Lambda^{\pm}(k) \in L$ for $\lambda_{k} \in U \cap \ZZ$ (the set of classical weights in the affinoid $U$) such that for any quadratic Dirichlet character $\chi$, there is a locally analytic $p$-adic function $L_{p}(\mathbf{f},\chi)$ on $U \times \ZZ_{p}$ that satisfies, for $1 \leq j \leq k+1$, 
\begin{equation} \label{eqn:mazurkitagawa1}
L_{p}(\mathbf{f},\chi)(\lambda_{k},j) = \begin{cases} 
\Lambda^{\pm}(k_0)\Big( 1 - \frac{\chi(p)p^{j-1}}{a_{p}(k)}\Big) L^{\mathrm{alg}}(f_{k}, \chi, j), &\text{if }k = k_{0}\\
\Lambda^{\pm}(k)\Big( 1 - \frac{\chi(p)p^{j-1}}{a_{p}(k)}\Big)\Big( 1 - \frac{\chi(p)p^{k-j+1}}{a_{p}(k)}\Big) L^{\mathrm{alg}}(f_{k}^{\#}, \chi, j), &\text{if }k \neq k_{0}
\end{cases} 
\end{equation}
\end{theorem}
\begin{proof} See \cite{Bel12} for instance.
\end{proof}
\begin{remark} \label{rem:mazurkitagawanormalization}
We can normalize $L_{p}(\mathbf{f},\chi)$ such that $\Lambda^{\pm}(k_{0}) = 1$. 
\end{remark}
$L_{p}(\mathbf{f},\chi)(\lambda_{\kappa}, s)$ is the two-variable $p$-adic $L$-function attached to the Coleman family $\mathbf{f}(q)$ and the character $\chi$ constructed by Mazur (unpublished) and Kitagawa (\cite{Kit94}) in the slope zero case (Hida families) and by Stevens in the finite slope scenario (\cite{PS11} \& \cite{PS12}). We will be primarily interested in the slice of the two variable $p$-adic $L$-function along the central critical line $s = k/2 + 1$. Namely, set
\[ \mathcal{L}_{p}(\mathbf{f},\chi,\lambda_{k}) \defeq L_{p}(\mathbf{f},\chi)(\lambda_{k}, k/2+1) \]
for all classical weights $\lambda_{k} \in U^{\mathrm{cl}}$. Theorem~\ref{thm:mazurkitagawa1} above implies that
\begin{equation} \label{eqn:mazurkitagawa2}
\mathcal{L}_{p}(\mathbf{f},\chi,\lambda_{k}) = 
\begin{cases} 
\Big( 1 - \frac{\chi(p)p^{k/2}}{a_{p}(k)}\Big) L^{\mathrm{alg}}(f_{k}, \chi, k/2+1), &\quad\text{if }k = k_{0}\\
\Lambda^{\pm}(k)\Big( 1 - \frac{\chi(p)p^{k/2}}{a_{p}(k)}\Big)^{2}L^{\mathrm{alg}}(f_{k}^{\#}, \chi, k/2+1), &\quad\text{if }k \neq k_{0}
\end{cases}
\end{equation}


\subsection{$p$-adic $L$-functions attached to Bianchi modular forms} \label{subsec:bianchip-adicLfns}
We briefly sumarize the construction of $p$-adic $L$-functions attached to Bianchi modular forms in \cite{Wil17} following the exposition of \cite[\S3.4]{BW19} where the $p$-adic $L$-function is described in terms of analytic functions on $\cO_{F} \otimes_{\ZZ} \ZZ_{p} \cong \cO_{F_{\p}}$ (since $p$ is inert in $F$) rather than as locally analytic distributions on $\mathrm{Cl}_{F}(p^{\infty})$ - the ray class group of $F$ of conductor $p^{\infty}$. Let $\mathfrak{g} \subseteq \cO_{F}$ be any ideal relatively prime to $p$. It can be shown that
\[ \mathrm{Cl}_{F}(\mathfrak{g}p^{\infty}) \cong [(\cO_{F}/\mathfrak{g})^{\times} \times (\cO_{F} \otimes_{\ZZ} \ZZ_{p})^{\times} ]/ \cO_{F}^{\times} \]
Let $\mathcal{G} \in S_{\underline{k_{0}}+2}(U_{0}(\cN))$ be a small slope cuspidal Bianchi eigenform and let $\Phi_{\mathcal{G}}$ be the \emph{overconvergent modular symbol} of \cite{Wil17} attached to $\mathcal{G}$. Let $\mu'_{a (\textrm{mod  }\mathfrak{g})}$ be a distribution on $\lbrace{[a]\rbrace} \times (\cO_{F} \otimes_{\ZZ} \ZZ_{p}) \subset (\cO_{F}/\mathfrak{g})^{\times} \times (\cO_{F} \otimes_{\ZZ} \ZZ_{p})$ (which can be seen as a copy of $(\cO_{F} \otimes_{\ZZ} \ZZ_{p})$) defined as 
\begin{equation} \label{eqn:distributionp-adicLfn}
   \mu'_{a (\textrm{mod  }\mathfrak{g})} \defeq (g\overline{g})^{k_{0}/2}\Bigg[ \Phi_{\mathcal{G}} \Biggl\lvert \begin{pmatrix} 1 & b\\0 & g\end{pmatrix} \Bigg] \lbrace 0 - \infty \rbrace
   \end{equation}
   where $b$ is some lift of $a (\textrm{mod  }\mathfrak{g})$ and $\mathfrak{g}\cO_{F} = (g)$. Combining the distributions for different $a \in (\cO_{F}/\mathfrak{g})^{\times}$, we get a distribution $\mu_{p}$ on $(\cO_{F}/\mathfrak{g})^{\times} \times (\cO_{F} \otimes_{\ZZ} \ZZ_{p})$. On restricting to units in the second variable and then to restricting to functions invariant under $\cO_{F}^{\times}$, we obtain a distribution on $\mathrm{Cl}_{F}(\mathfrak{g}p^{\infty})$ afforded by the identification above.
   \begin{definition} \label{defn:p-adicLfn}  
   Let $\chi$ be a finite order Hecke character of conductor $\mathfrak{gf}$ where $\mathfrak{g}$ is coprime to $p$ and $\mathfrak{f}\mid p^{\infty}$ (which can be seen as a finite order character of $\mathrm{Cl}_{F}(\mathfrak{g}p^{\infty})$). The $p$-adic $L$-function associated to $\mathcal{G} \in S_{\underline{k_{0}}+2}(U_{0}(\cN))$ is defined to be the analytic function on $(\cO_{F} \otimes_{\ZZ} \ZZ_{p})$ given by
   \[ L_{p}(\mathcal{G}, \chi, s) \defeq \int_{\mathrm{Cl}(\mathfrak{g}p^{\infty})} \langle \mathbf{z}_{p} \rangle^{s} \chi(\mathbf{z}) d\mu_{p}(\mathbf{z}), \]
 where $s \in \cO_{F} \otimes_{\ZZ} \ZZ_{p} \cong \cO_{F_{\p}}$ and $\mathbf{z}_{p}$ is the projection of $\mathbf{z} \in \mathrm{Cl}(\mathfrak{g}p^{\infty})$ to $\mathrm{Cl}(p^{\infty})$.   
   \end{definition}
Germane to this article is the two variable base--change Bianchi $p$-adic $L$-function constructed by Seveso in \cite[Section 5.3]{Sev12},
\begin{equation} \label{eqn:SevesopadicLfn}
L_{p}(\mathbfcal{F},\chi,\lambda_{\kappa}) : U \rightarrow \mathbb{C}_{p}
\end{equation}
where $\chi$ is a finite order Hecke character of $F$. This two variable $p$-adic $L$-function interpolates the central critical values of the weight $k$-specializations of the Coleman family $\mathbfcal{F} = \mathbf{f}/F$ which we recall below :-
\begin{theorem} \label{thm:bianchiinterpolation}
Let $\chi$ be a finite order Hecke character of $F$ of conductor $\mathfrak{c}_{\chi}$. Then, for all classical weights $\lambda_{k} \in U$, 
\begin{equation} \label{eqn:bianchiinterpolation1}
L_{p}(\mathbfcal{F}, \chi, \lambda_{k}) =
\begin{cases}
\left( 1 - \chi(p)\frac{\mathrm{N}_{F/\QQ}(p)^{k/2}}{a_{\p}(\cF_{k})}\right)\cdot L^{\mathrm{alg}}(\cF, \chi, k_{0}/2 + 1) &\text{if } k=k_{0} \\
 C(k)\Big( 1 - \chi(p)\frac{\mathrm{N}_{F/\QQ}(p)^{k/2}}{a_{\p}(\cF_{k})}\Big)^{2}\cdot L^{\mathrm{alg}}(\cF_{k}^{\#}, \chi, k/2 + 1) & \text{if }k\neq k_{0}
 \end{cases}
\end{equation}
where $C(k) \in L^{\times}$ is the $p$-adic period of Theorem~\ref{thm:weightkspecialization} and 
\begin{equation} \label{eqn:bianchiperiods}
L^{\mathrm{alg}}(\cF_{k}^{\#}, \chi, k/2 + 1) \defeq u_{F}\frac{D_{F}^{k/2}(k/2)!^{2}\tau(\chi^{-1})(\mathrm{N}_{F/\QQ}(\mathfrak{c}_{\chi}))^{k/2}}{(2\pi i)^{k}\Omega_{k}^{\#}}L(\cF_{k}^{\#}, \chi, k/2 + 1) \in \overline{\QQ}
\end{equation}
for  $u_{F} \defeq \left[\cO_{F}^{\times}:\ZZ^{\times}\right] = |\mu(\cO_{F})|/2$.
\end{theorem}  
\begin{proof} See \cite[Theorem 3.12]{BW19}, \cite[Theorem 5.17]{Sev12} and \cite[Theorem 3.8]{BD07}.
\end{proof}

We then have the following factorisation of the Seveso $p$-adic $L$-function from above as a product of two Stevens--Mazur--Kitagawa $p$-adic $L$-functions of Section~\ref{subsec:smk}.
\begin{theorem} \label{thm:factorisation1}
Let $\epsilon_{F/\QQ}$ denote the quadratic Dirichlet character associated to the imaginary quadratic field $F/\QQ$. Then there exists a $p$-adic analytic function $\eta$ of $\lambda_{\kappa} \in U$ such that
\begin{equation} \label{eqn:factorisation1}
L_{p}(\mathbfcal{F}, \lambda_{\kappa}) = L_{p}(\mathbf{f}/F,\lambda_{\kappa}) = \eta(\lambda_{\kappa})u_{F}\cL_{p}(\mathbf{f}, \lambda_{\kappa})\cL_{p}(\mathbf{f}, \epsilon_{F/\QQ}, \lambda_{\kappa}) 
\end{equation} 
\end{theorem}
\begin{proof}
We refer the reader to \cite[Theorem 5.21]{Sev12} which is based on \cite[Corollary 5.3]{BD07}.
\end{proof}
\subsection{Heegner cycles} \label{subsec:Heegnercycles} We briefly recall the connection between Heegner cycles and the $p$-adic $L$-functions considered above, primarily following the exposition in \cite{IS03}, \cite{Sev12}, \cite{Sev14} and \cite{GSS16}. Recall the factorization $N = pM = pN^{+}N^{-}$. Let $\cB$ (resp. $B$) be the indefinite (resp. definite) quaternion algebra ramified at the primes dividing $pN^{-}$ (resp. $N^{-}\infty$). Let $\cR' = \cR_{N^{+},pN^{-}}$ (resp. $R'$) be a fixed Eichler order of level $N^{+}$ in $\cO_{\cB}$ (resp. of level $pN^{+}$ in $\cO_{B}$) where $\cO_{\cB}$ and $\cO_{B}$ are maximal orders in $\cB$ and $B$ respectively. We set $\widehat{B} \defeq B \otimes \widehat{\ZZ}$ and $\widehat{R'} \defeq R' \otimes \widehat{\ZZ}$. Let $\Sigma = \prod\limits_{\ell} \Sigma_{\ell} \defeq \widehat{R'}^{\times}$. Fix an identification $\iota_{p}:B \otimes \QQ_{p} \cong \mathrm{M}_{2}(\QQ_{p})$ and set 
\[ \widetilde{\Gamma}'_{\Sigma} \defeq \iota_{p}\left( \cO_{B}[1/p] \cap \prod\limits_{\ell \neq p} \Sigma_{\ell} \right) = \iota_{p}\left( R'[1/p]^{\times} \right) \]
and by $\Gamma'_{\Sigma} = \Gamma'$ to be the subgroups of elements of reduced norm one of $\widetilde{\Gamma}'_{\Sigma}$. Denote by $X \defeq X_{N^{+},pN^{-}}$ the Shimura curve attached to $\cB$ and by $f^{\mathrm{JL}} \in M_{k_{0}+2}(X)$ the weight $k_{0}+2$ modular form on the Shimura curve $X$ attached to $f$ via the Jacquet--Langlands correspondence. By the Cerednik--Drinfeld Theorem of $p$-adic uniformization, we have a rigid analytic isomorphism
\[ X_{\Gamma'} \defeq \Gamma'\backslash \cH_{p} \cong X^{\mathrm{an}} \]
This identification between the Mumford curve $X_{\Gamma'}$ and the rigid analytification of the Shimura curve $X^{\mathrm{an}}$ is defined over $F_{\p} \cong \QQ_{p^{2}}$ (See \cite{BC91}). We denote by $f^{\mathrm{rig}} \in M_{k_{0}+2}(\Gamma', F_{\p})$ the rigid analytic modular form associated to $f$ via this identification.
\paragraph*{} Let $\cM_{k_{0}}$ be the Chow motive over $\QQ$, attached to the space of weight $k_{0}+2$ modular forms on the Shimura curve $X$ and let $V(k_{0}/2+1) \defeq H_{p}(\cM_{k_{0},\overline{\QQ}}, \QQ_{p}(k_{0}/2+1))$ be its $p$-adic realization. See \cite[Appendix 10.1]{IS03} for the construction. By \cite[Lemma 5.8]{IS03}, the $G_{\QQ}$-representation $V(k_{0}/2+1)$ maybe realized as the representation attached to weight $k_{0}+2$ cusp forms that are new at the primes dividing $pN^{-}$. In particular, $V_{p}(f)(k_{0}/2+1)$ maybe realized as the idempotent component (corresponding to $f$) of the representation $V(k_{0}/2+1)$. For $H/\QQ$ any number field, we have the global $p$-adic \'{e}tale Abel--Jacobi map
\[ \mathrm{cl}_{H} \defeq \mathrm{cl}_{0, H}^{k_{0}/2+1} : \mathrm{CH}^{k_{0}/2+1}\left( \cM_{k_{0}} \otimes H \right) \rightarrow \mathrm{Sel}_{\mathrm{st}}(H, V(k_{0}/2+1)) \] 
where $\cM_{k_{0}} \otimes H$ is the base--change of $\cM_{k_{0}}/\QQ$ to $H$ and $\mathrm{CH}^{k_{0}/2+1}$ denotes the Chow group of co-dimension $k_{0}/2+1$ cycles. We may also consider the projection $V(k_{0}/2+1) \rightarrow V_{p}(f)(k_{0}/2+1)$ to obtain
\[ \mathrm{cl}_{f, H}:\mathrm{CH}^{k_{0}/2+1}\left( \cM_{k_{0}} \otimes H \right) \rightarrow \mathrm{Sel}_{\mathrm{st}}(H, V_{p}(f)(k_{0}/2+1)) \] 
Let $\mathfrak{P}$ be a prime in $H$ above $p$ and let $H_{\mathfrak{P}}$ be its $\mathfrak{P}$-adic completion. Then we have a commutative diagram
\begin{equation} \label{eqn:ISabeljacobi1}
\begin{tikzcd}
\mathrm{CH}^{k_{0}/2+1}\left( \cM_{k_{0}} \otimes H \right)  \arrow[r, "\mathrm{cl}_{f,H}"] \arrow[d]
&\mathrm{Sel}_{\mathrm{st}}(H, V_{p}(f)(k_{0}/2+1)) \arrow[d, "\mathrm{res}_{\mathfrak{P}}" ] \\
\mathrm{CH}^{k_{0}/2+1}\left( \cM_{k_{0}} \otimes H_{\mathfrak{P}}\right)  \arrow[r, "\mathrm{cl}_{f,H_{\mathfrak{P}}}"]
& \mathrm{H}^{1}_{\mathrm{st}}(H_{\mathfrak{P}}, V_{p}(f)(k_{0}/2+1)) 
\end{tikzcd}
\end{equation}
Let 
\[ \mathbb{D}_{f} \defeq \mathbb{D}_{\mathrm{st}}(V_{p}(f)|_{G_{\QQ_{p}}}) \]
denote the rank-two $(\varphi, N)$-module attached to $V_{p}(f)$. By \cite[(49)]{IS03}, we have the following identification 
\begin{align*} \mathrm{IS}: \mathrm{H}^{1}_{\mathrm{st}}(H_{\mathfrak{P}}, V_{p}(f)(k_{0}/2+1)) & \overset{\mathrm{log_{BK}}}{\cong}  \frac{\mathbb{D}_{f} \otimes H_{\mathfrak{P}}}{\mathrm{Fil}^{k_{0}/2+1}(\mathbb{D}_{f} \otimes H_{\mathfrak{P}})} \\
& \cong  M_{k_{0}+2}(X, H_{\mathfrak{P}})_{(f^{\mathrm{JL}})}^{\vee} \cong M_{k_{0}+2}(\Gamma', H_{\mathfrak{P}})_{(f^{\mathrm{rig}})}^{\vee}  
\end{align*}
where $\mathrm{log}_{\mathrm{BK}}$ is the Bloch--Kato logarithm and the final identification holds assuming $H_{\mathfrak{P}} \supseteq F_{\p} \cong \QQ_{p^{2}}$. Here $(-)^{\vee}$ stands for the $H_{\mathfrak{P}}$-dual and $(f^{?})$ stands for the $f^{?}$-isotypic component. It would be useful to consider the composition
\begin{multline} \label{ISabeljacobi2}
\mathrm{log\;cl}_{f,L}: \mathrm{CH}^{k_{0}/2+1}\left( \cM_{k_{0}} \otimes F \right)  \rightarrow  \mathrm{Sel}_{\mathrm{st}}(F, V_{p}(f)(k_{0}/2+1))  \rightarrow  \mathrm{Sel}_{\mathrm{st}}(K, V_{p}(f)(k_{0}/2+1)) \\  \xrightarrow{\mathrm{res}_{\p}} \mathrm{H}^{1}_{\mathrm{st}}(L, V_{p}(f)(k_{0}/2+1)) 
\rightarrow \frac{\mathbb{D}_{f} \otimes L}{\mathrm{Fil}^{k_{0}/2+1}(\mathbb{D}_{f} \otimes L)} \\
  \rightarrow M_{k_{0}+2}(X, L)_{(f^{\mathrm{JL}})}^{\vee} \rightarrow M_{k_{0}+2}(\Gamma', L)_{(f^{\mathrm{rig}})}^{\vee}
\end{multline}
where $L/\QQ_{p}$ is as before (recall that $L \supseteq K_{\p}$) and 
\[ \mathrm{Sel}_{\mathrm{st}}(F, V_{p}(f)(k_{0}/2+1))  \rightarrow  \mathrm{Sel}_{\mathrm{st}}(K, V_{p}(f)(k_{0}/2+1)) \]
is the usual restriction map.
\paragraph*{} The following result is proved in \cite{Sev14} generalising the weight the $k_{0} = 0$ setting of \cite{BD07}. See also \cite[Section 5.3.2]{Sev12} for more details.
\begin{theorem} \label{thm:HeegnercyclespadicAJ}
There exists a global cycle
\[ \mathcal{Y} \in \mathrm{CH}^{k_{0}/2+1}\left( \cM_{k_{0}} \otimes F \right) \]
such that
\[ \frac{d^{2}}{d\lambda_{\kappa}^{2}}\left[ L_{p}(\mathbfcal{F}, \lambda_{\kappa}) \right]_{\lambda_{\kappa} = \lambda_{k_{0}}} = \frac{d^{2}}{d\lambda_{\kappa}^{2}}\left[ L_{p}(\mathbf{f}/F, \lambda_\kappa) \right]_{\lambda_{\kappa} = \lambda_{k_{0}}} = 2 \mathrm{log\;cl}_{f,L}(\mathcal{Y})(f^{\mathrm{rig}})^{2} \]
\end{theorem}
\begin{proof}
This is \cite[Corollary 5.27]{Sev12} which is in turn a special case of \cite[Corollary 9.2]{Sev14}.   
\end{proof}
\begin{remark}
The global cycle $\mathcal{Y} \in \mathrm{CH}^{k_{0}/2+1}\left( \cM_{k_{0}} \otimes F \right)$ is the Heegner cycle (associated to the trivial character $\chi = 1$) constructed in \cite[Section 8]{IS03} using the theory of Complex Multiplication.
\end{remark}

\subsection{$p$-adic $L$-functions over $K$} \label{subsec:p-adicL-functionsoverK}
Recall that $K/F$ is a relative quadratic extension that satisfies the Stark--Heegner hypothesis (\textbf{SH--Hyp})
\begin{itemize}
\item[•] $\p$ is inert in $K$
\item[•] All primes $\mathfrak{l} \mid \cM$ split in $K$
\end{itemize}
Let $\Psi \in \mathrm{Emb}(\cO, \cR)$ be an optimal embedding of conductor $\mathcal{C}$ relatively prime to $\cN\mathcal{D}_{K/F}$. Let $G_{\cC} \defeq \mathrm{Gal}(H_{\cC}/K)$ be the corresponding Galois group of the ring class field of conductor $\cC$ and $(\tau_{\Psi}, P_{\psi}, \gamma_{\Psi})$ be the data attached to the embedding $\Psi$ as in \S\ref{subsec:starkheegnercycles}. Let $\mathscr{L}_{\Psi}$ be an $\cO_{F_{\p}}$-lattice corresponding to the vertex $v_{\Psi}$ and let $\widetilde{L}_{\Psi} \defeq \mu(\cO_{F})\backslash \mathscr{L}_{\Psi}$.  In particular $\mathscr{L}_{\Psi} = \mathscr{L}_{\tau_{\Psi}} = \mathscr{L}_{\tau^{\theta}_{\Psi}}$ since $v_{\Psi} = \mathrm{red}_{\p}(\tau_{\Psi}) = \mathrm{red}_{\p}(\tau^{\theta}_{\Psi})$. 

Following \cite{BD09} and \cite{Sev12}, we define a \emph{partial square root $p$-adic $L$-function} to such an embedding $\Psi$ as follows
\begin{definition} \label{def:squarerootpadicLfn}
Let $r \in \mathbb{P}^{1}(F)$ be any base point. 
\begin{itemize}
\item[(i)] The partial square root $p$-adic $L$-function attached to $(\ccF/K, \Psi)$ is defined as
\[ \cL_{p}(\ccF/K, \Psi, \lambda_{\kappa}) \defeq | \mathscr{L}_{\Psi}|^{k_{0}/2}\int_{\widetilde{L}_{\Psi}'} \langle P_{\Psi}(x,y)\rangle^{\frac{\lambda_{\kappa}-\lambda_{k_{0}}}{2}}P_{\Psi}^{k_{0}/2}(x,y)d\Phi_{\widetilde{L}_{\Psi}}\lbrace r - \gamma_{\Psi}r \rbrace.  \]

\item[(ii)] The partial square root $p$-adic $L$-function attached to $(\ccF/K, \psi_{K})$ for $\psi_{K} : G_{\cC} \rightarrow \mathbb{C}^{\times}$ is then defined as 
\[ \cL_{p}(\ccF/K, \psi_{K}, \lambda_{\kappa}) \defeq \sum\limits_{\sigma \in G_{\cC}} \psi_{K}^{-1}(\sigma)\cL_{p}(\ccF/K, \sigma\Psi, \lambda_{\kappa}) \]
and finally
\item[(iii)] the \emph{$p$-adic $L$-function} attached to $(\ccF/K, \psi_{K})$ is defined as 
\[ L_{p}(\ccF/K, \psi_{K}, \lambda_{k}) \defeq \cL_{p}(\ccF/K, \psi_{K}, \lambda_{\kappa})^{2}. \]
\end{itemize}
\end{definition} 
\begin{remark} \label{rem:choiceofbigmodularsymbol}
A priori the $p$-adic $L$-functions defined above depend on the $\mathbb{D}_{U}^{\dagger}$-valued modular symbol $\Phi_{\widetilde{L}_{\Psi}}$ of Proposition~\ref{prop:latticesymbols} associated to the lattice $\mathscr{L}_{\Psi}$. It can be shown that the definition depends only on the class of optimal embeddings $[\Psi] \in \Gamma\backslash \mathrm{Emb}(\cO, \cR)$ (See \cite[Lemma 5.1]{GSS16}). Following \cite[Remark 5.6]{Sev12}, we choose a lattice $\mathscr{L}_{\Psi}$ as follows. Since $\Gamma$ acts transitively on the set of vertices $\cV(\cT)$, let $\gamma \in \Gamma$ be such that $\gamma v_{\Psi} = v_{*}$. Then $v_{*} = v_{\gamma \Psi \gamma^{-1}}$ and $\mathscr{L}_{*} = \mathscr{L}_{\gamma \Psi \gamma^{-1}}$ is the lattice associated to the optimal embedding $\gamma \Psi \gamma^{-1} \in [\Psi]$. We show later on in Theorem~\ref{thm:padicinterpolationoverK} that this choice of a lattice associated to a class of optimal embeddings $[\Psi]$ is the natural one to consider.
\end{remark}
\paragraph*{} The $p$-adic $L$-function $L_{p}(\ccF/K, \psi_{K}, -)$ defined above interpolates $L^{\mathrm{alg}}(\cF_{k}^{\#}/K, \psi_{K}, k/2 + 1)$ -- the \emph{algebraic part} of central critical $L$-values of the newforms $\cF_{k}^{\#}$, 
\begin{multline} \label{eqn:basechangeperiods}
L^{\mathrm{alg}}(\cF_{k}^{\#}/K, \psi_{K}, k/2 + 1) \defeq \\
\frac{Tu_{K}^{2}}{(\Omega_{k}^{\#})^{2}}\cdot\frac{\prod_{\nu \in \Sigma_{\infty}^{F}} C'_{\nu}(K,\pi_{k},\psi_{K})}{\sqrt{\mathrm{N}_{F/\QQ}(\cD_{K/F})}}\cdot \frac{((k/2)!)^{4}}{(2\pi)^{2k+4}}L(\cF_{k}^{\#}/K, \psi_{K}, k/2 + 1) \in \overline{\QQ}. 
\end{multline}
where $u_{K} \defeq \left[ \mu(\cO_{K}):\mu(\cO_{F}) \right]$ and $T$ and $C'_{\nu}(K,\pi_{k},\psi_{K})$ are explicit constants (See Appendix~\ref{appendix1}). This entails rewriting the adelic toric periods appearing in Waldspurger's formula (See \cite{MW09} and \cite{FMP17}) in terms of certain geodesic cycles. 

Let  
\begin{equation} \label{eqn:popa1} \mathbb{L}(\cF_{k}^{\#}, \psi_{K}) \defeq \Bigg( \sum\limits_{\sigma \in G_{\cC}} \psi_{K}^{-1}(\sigma)\phi_{k}^{\#}\lbrace \tau - \gamma_{\sigma\Psi}\tau \rbrace \Big( (P_{\sigma\Psi}(x,y))^{k/2} \Big) \Bigg)^{2}  
\end{equation}
where $\phi_{k}^{\#}$ is the Bianchi modular symbol attached to the newform $\cF_{k}^{\#}$. Then, the following generalization of Alexandru Popa's result (\cite[Theorem 6.3.1]{Pop06}) holds.
\begin{theorem} \label{thm:popageneralization} Let $\psi_{K}: \mathrm{Gal}(H_{K}/K) \rightarrow \mathbb{C}^{\times}$ be an unramified character. Then
\[
\mathbb{L}(\cF_{k}^{\#}, \psi_{K}) =  L^{\mathrm{alg}}(\cF_{k}^{\#}/K, \psi_{K}, k/2 + 1).
\]
\end{theorem}
\begin{proof} 
The proof follows from Santiago Molina's \emph{Waldspurger formula in higher cohomology} (\cite[Theorem 4.6]{Mol22}) and an explicit form of Waldspurger formula as in \cite[Theorem 4.2]{MW09} or \cite[Theorem 1.8]{CST14}.  Given that such a result could be of independent interest, we explain this in detail in Appendix~\ref{appendix1}.
\end{proof}
\begin{theorem} \label{thm:vanishingofp-adicLfn1}
The $p$-adic $L$-functions, for all $\Psi$ and $\psi_{K}$ as above, vanish at $\lambda_{k_{0}}$, i.e.
\[ \cL_{p}(\ccF/K, \Psi, \lambda_{k_{0}}) = \cL_{p}(\ccF/K, \psi_{K}, \lambda_{k_{0}}) = L_{p}(\ccF/K, \psi_{K}, \lambda_{k_{0}}) = 0. \]
Further, we also have
\[ \frac{d}{d\lambda}[L_{p}(\ccF/K, \psi_{K}, \lambda_{\kappa})]_{\kappa = k_{0}} = 0 \]
\end{theorem}
\begin{proof} 
By Definition~\ref{def:squarerootpadicLfn} above, we have 
\begin{align*}
\cL_{p}(\ccF/K, \Psi, \lambda_{k_{0}}) & = |\mathscr{L}_{\Psi}|^{k_{0}/2}\int_{\widetilde{L}_{\Psi}'}P_{\Psi}^{k_{0}/2}(x,y)d\Phi_{\widetilde{L}_{\Psi}}\lbrace \tau - \gamma_{\Psi}\tau \rbrace \\
& = \int_{\mathbb{P}^{1}(F_{\p})}P_{\Psi}^{k_{0}/2}(t)d\pi_{*}(\Phi_{\widetilde{L}_{\Psi}})\lbrace \tau - \gamma_{\Psi}\tau \rbrace
\end{align*}
where the second equality follows from (\ref{eqn:pushforward}). By Corollary~\ref{cor:pushforwardtoharmonic}, we then have that 
\[ \cL_{p}(\ccF/K, \Psi, \lambda_{k_{0}}) = \int_{\mathbb{P}^{1}(F_{\p})} P_{\Psi}^{k_{0}/2}(t)d\Phi_{\cF}^{\mathrm{har}}\lbrace \tau - \gamma_{\Psi}\tau \rbrace \]
The vanishing now follows from \cite[Proposition 5.8(i)]{BW19}. The defining properties of the other $p$-adic $L$-functions implies the vanishing simultaneously at $\lambda_{k_{0}}$.
\end{proof}
By Theorem~\ref{thm:popageneralization}, we have the following interpolation property of the $p$-adic $L$-functions introduced in this section
\begin{theorem} \label{thm:padicinterpolationoverK}
For all classical weights $\lambda_{k} \in U$, $k \neq k_{0}$, we have
\[ L_{p}(\ccF/K, \psi_{K}, \lambda_{k}) = C(k)^{2}\left(1 - \frac{\mathrm{N}_{F/\QQ}(p)^{k}}{a_{\p}(\mathcal{F}_{k})^{2}}\right)^{2}L^{\mathrm{alg}}(\cF_{k}^{\#}/K, \psi_{K}, k/2 + 1) \]
\end{theorem}
\begin{proof}
By definition,
\[ \cL_{p}(\ccF/K, \Psi, \lambda_{k}) = |\mathscr{L}_{\Psi}|^{k_{0}/2}\int_{\widetilde{L}_{\Psi}'}P_{\Psi}^{k/2}(x,y)d\Phi_{\widetilde{L}_{\Psi}}\lbrace \tau - \gamma_{\Psi}\tau \rbrace \]
since $\langle P_{\Psi} \rangle^{\frac{\lambda_{k}-\lambda_{k_{0}}}{2}}P_{\Psi}^{k_{0}/2} = P_{\Psi}^{k/2}[P_{\Psi}]^{\frac{\lambda_{k}-\lambda_{k_{0}}}{2}} = P_{\Psi}^{k/2}$ (as $[z]^{k} = [z]^{k_{0}}$ for all $\lambda_{k} \in U$ classical). By Remark~\ref{rem:choiceofbigmodularsymbol} above, we may choose $\mathscr{L}_{\Psi}$ to be the lattice $\mathscr{L}_{*} = \cO_{F_{\p}} \oplus \cO_{F_{\p}}$ and hence we get
\[ \cL_{p}(\ccF/K, \Psi, \lambda_{k}) = \int_{\widetilde{L}_{*}'}P_{\Psi}^{k/2}(x,y)d\Phi_{\widetilde{L}_{*}}\lbrace \tau - \gamma_{\Psi}\tau \rbrace \]
which by Lemma~\ref{lem:pstabilisationmodularsymbols} implies that
\[ \cL_{p}(\ccF/K, \Psi, \lambda_{k}) = C(k)\left(1 - \frac{\mathrm{N}_{F/\QQ}(p)^{k}}{a_{\p}(\mathcal{F}_{k})^{2}}\right)\phi_{k}^{\#}\lbrace \tau - \gamma_{\Psi}\tau \rbrace \Big( (P_{\Psi}(x,y))^{k/2} \Big). \]
Hence, by Definition~\ref{def:squarerootpadicLfn} 
\[ L_{p}(\ccF/K, \psi_{K}, \lambda_{k}) = C(k)^{2}\left(1 - \frac{\mathrm{N}_{F/\QQ}(p)^{k}}{a_{\p}(\mathcal{F}_{k})^{2}}\right)^{2}\mathbb{L}(\cF_{k}^{\#}, \psi_{K}). \]
The proof follows from Theorem~\ref{thm:popageneralization}.
\end{proof}
\subsection{Factorization of $p$-adic $L$-functions I} \label{subsec:factorization}
Let $\epsilon_{K/F}$ be the quadratic id\`{e}le class character of $F$ that cuts out the relative quadratic extension $K/F$ and let $\chi_{1}$ and $\chi_{2}$ be a pair of quadratic Hecke characters such that $\chi_{1}\cdot\chi_{2} = \epsilon_{K/F}$. Let $\psi_{K}$ be the \emph{genus character} of $K$ associated to the pair $\chi_{1},\chi_{2}$, i.e.
\[ \mathrm{Ind}_{K}^{F} \psi_{K} = \chi_{1} \oplus \chi_{2}. \]
Then by the classical Artin formalism, we have the following factorization of $L$-functions :-
\begin{equation} \label{eqn:classicalartinformalism}
L(\cF_{k}^{\#}/K, \psi_{K}, s) = L(\cF_{k}^{\#}, \chi_{1}, s)L(\cF_{k}^{\#}, \chi_{2}, s). 
\end{equation}
The goal of this section is to show that a similar factorization ($p$-adic Artin formalisim) also holds at the level of $p$-adic $L$-functions.  Set 
\begin{equation} \label{eqn:introductionofeta}
\eta \defeq \frac{u_{K}^{2}.T.\prod_{\nu \in \Sigma_{\infty}^{F}} C'_{\nu}(K,\pi_{k},\psi_{K})}{u_{F}^{2}(2\pi)^{4}\mathrm{N}_{F/\QQ}(\cD_{K/F})}.
\end{equation}
\begin{theorem} \label{thm:factorizationofpadicLfns1}
For all classical weights $\lambda_{k} \in U$, 
\[ (D_{K})^{k/2}L_{p}(\mathbfcal{F}/K, \psi_{K},\lambda_{k}) = \eta.L_{p}(\mathbfcal{F}, \chi_{1},\lambda_{k}).L_{p}(\mathbfcal{F}, \chi_{2}, \lambda_{k}) \]
\end{theorem}
\begin{proof}
Let $\lambda_{k} \in U (k \neq k_{0})$ be a classical weight. Then by (\ref{eqn:classicalartinformalism}) above, we have 
\[ L(\cF_{k}^{\#}/K, \psi_{K},k/2+1) = L(\cF_{k}^{\#}, \chi_{1},k/2+1)L(\cF_{k}^{\#},\chi_{2}, k/2+1) \]
Recall that we have the following identities (See \cite[Appendix I]{Mok11} for instance)
\[ \mathrm{N}_{F/\QQ}(\mathfrak{c}_{\chi_{1}})\mathrm{N}_{F/\QQ}(\mathfrak{c}_{\chi_{2}}) = \mathrm{N}_{F/\QQ}(\cD_{K/F}) \]
\[ \tau(\chi_{1}^{-1})\tau(\chi_{2}^{-1}) = \tau(\chi_{1})\tau(\chi_{2}) = \sqrt{\mathrm{N}_{F/\QQ}(\cD_{K/F})} \]
Further, upon combining (\ref{eqn:bianchiperiods}) and (\ref{eqn:basechangeperiods}), we get
\begin{multline} \label{eqn:factorizationofalgebraicparts} \frac{(2\pi)^{4}\sqrt{\mathrm{N}_{F/\QQ}(\cD_{K/F})}}{u_{K}^{2}.T.\prod_{\nu \in \Sigma_{\infty}^{F}} C'_{\nu}(K,\pi_{k},\psi_{K})} L^{\mathrm{alg}}(\cF_{k}^{\#}/K, k/2+1) \\ = \frac{1}{u_{F}^{2}\sqrt{\mathrm{N}_{F/\QQ}(\cD_{K/F})}D_{K}^{k/2}} L^{\mathrm{alg}}(\cF_{k}^{\#}, k/2+1)L^{\mathrm{alg}}(\cF_{k}^{\#}, \epsilon_{K/F}, k/2+1) 
\end{multline}
By Theorem~\ref{thm:bianchiinterpolation} and Theorem~\ref{thm:padicinterpolationoverK} above, we have
\[ (D_{K})^{k/2}.L_{p}(\mathbfcal{F}/K, \psi_{K},\lambda_{k}) = \eta.L_{p}(\mathbfcal{F}, \chi_{1},\lambda_{k}).L_{p}(\mathbfcal{F}, \chi_{2}, \lambda_{k}) \]
for all $\lambda_{k} \in U$ classical.  The proof for $\lambda_{k_{0}}$ is simpler as both sides vanish.
\end{proof}
Note that since by assumption that $D_{K}$ is relatively prime to $p$, the function $D_{K}^{k/2}$ extends to an analytic function on $U$, $D_{K}^{\lambda_{\kappa}/2} \defeq \langle D_{K} \rangle^{\lambda_{\kappa}/2}$. 
\begin{corollary} \label{cor:factorizationofpadicLfns}
For all genus characters $\psi_{K}$ of $K$ associated to a pair of quadratic Hecke characters $(\chi_{1},\chi_{2})$ of $F$ and for all $\lambda_{\kappa} \in U$,
\begin{equation} \label{eqn:factorizationofpadicLfns} (D_{K})^{\lambda_{\kappa}/2}L_{p}(\mathbfcal{F}/K, \psi_{K},\lambda_{\kappa}) = \eta.L_{p}(\mathbfcal{F}, \chi_{1},\lambda_{\kappa}).L_{p}(\mathbfcal{F}, \chi_{2}, \lambda_{\kappa}).
\end{equation}
\end{corollary}
\begin{proof}
Since the set of classical points in $U$ is Zariski--dense and the two sides of (\ref{eqn:factorizationofpadicLfns}) are continuous functions on $U$, they agree on all points of $U$.
\end{proof}

\subsection{A $p$-adic Gross--Zagier formula} We will now show a $p$-adic Gross--Zagier formula relating the (second derivative of the) base--change $p$-adic $L$-function $L_{p}(\mathbfcal{F}/K, \psi_{K}, \lambda_{\kappa})$ and the $p$-adic Abel--Jacobi image of Stark--Heegner cycles introduced in \S\ref{subsec:starkheegnercycles}.
\begin{theorem} \label{thm:padicGZformula1}
\begin{multline*} \frac{d}{d\lambda_{\kappa}}[\cL_{p}(\mathbfcal{F}/K, \Psi, \lambda_{\kappa})]_{\kappa = k_{0}} = \\
 \frac{1}{2}\left(\mathrm{N}_{F/\QQ}(\cD_{K/F})\right)^{k_{0}/4}\left(\mathrm{log}\:\Phi^{\mathrm{AJ}}(\mathrm{D}_{[\Psi]})(\Phi_{\cF}^{\mathrm{har}})+ (-1)^{\frac{k_{0}+2}{2}}\mathrm{log}\:\Phi^{\mathrm{AJ}}(\mathrm{D}_{[\Psi^{\theta}]})(\Phi_{\cF}^{\mathrm{har}})\right) 
\end{multline*}
\end{theorem} 
\begin{proof}
Consider the factorization 
\begin{align*} 
P_{\Psi}(x,y) & = A(x - \tau_{\Psi}y)(x - \tau_{\Psi}^{\theta}y)\overline{A}(\overline{x} - \tau_{\Psi}\overline{y})(\overline{x} - \tau_{\Psi}^{\theta}\overline{y}) \\
& = \mathrm{N}_{F/\QQ}(A(x - \tau_{\Psi}y)(x - \tau_{\Psi}^{\theta}y)). 
\end{align*}
Here $\theta$ is the non-trivial automorphism of $\mathrm{Gal}(K/F)$ (and the over line denotes the non-trivial automorphism of $\mathrm{Gal}(F/\QQ)$). Then, we can write
\begin{multline*} 
\cL_{p}(\mathbfcal{F}/K, \Psi, \lambda_{\kappa}) = |\mathscr{L}_{\Psi}|^{k_{0}/2}\langle \mathrm{N}_{F/\QQ}(A) \rangle^{\frac{\lambda_{\kappa}-\lambda_{k_{0}}}{2}} \\ \int_{\widetilde{L}_{\Psi}'} \langle\mathrm{N}_{F/\QQ}(x - \tau_{\Psi}y)\rangle^{\frac{\lambda_{\kappa}-\lambda_{k_{0}}}{2}}\langle\mathrm{N}_{F/\QQ}(x - \tau_{\Psi}^{\theta}y)\rangle^{\frac{\lambda_{\kappa}-\lambda_{k_{0}}}{2}}P_{\Psi}^{k_{0}/2}(x,y)d\Phi_{\widetilde{L}_{\Psi}}\lbrace r - \gamma_{\Psi}r \rbrace  
\end{multline*}
By the product rule for derivatives,
\begin{multline*}
\frac{d}{d\lambda_{\kappa}}[\cL_{p}(\mathbfcal{F}/K, \Psi, \lambda_{\kappa})]_{\kappa = k_{0}} = |\mathscr{L}_{\Psi}|^{k_{0}/2}\langle \mathrm{N}_{F/\QQ}(A) \rangle^{\frac{\lambda_{k_{0}}-\lambda_{k_{0}}}{2}} \\
\frac{d}{d\lambda_{\kappa}}\left[\int_{\widetilde{L}_{\Psi}'} \langle \mathrm{N}_{F/\QQ}((x - \tau_{\Psi}y)(x - \tau_{\Psi}^{\theta}y))\rangle^{\frac{\lambda_{\kappa}-\lambda_{k_{0}}}{2}}P_{\Psi}^{k_{0}/2}(x,y)d\Phi_{\widetilde{L}_{\Psi}}\lbrace r - \gamma_{\Psi}r \rbrace  \right]_{\kappa = k_{0}}  \\ 
+ |\mathscr{L}_{\Psi}|^{k_{0}/2}  \int_{\widetilde{L}_{\Psi}'} P_{\Psi}^{k_{0}/2}(x,y)d\Phi_{\widetilde{L}_{\Psi}}\lbrace r - \gamma_{\Psi}r \rbrace  \cdot \frac{d}{d\lambda_{\kappa}}\left[\langle \mathrm{N}_{F/\QQ}(A) \rangle^{\frac{\lambda_{\kappa}-\lambda_{k_{0}}}{2}}\right] 
\end{multline*}

By the proof of Theorem~\ref{thm:vanishingofp-adicLfn1}, we have the vanishing,
\[ |\mathscr{L}_{\Psi}|^{k_{0}/2}  \int_{\widetilde{L}_{\Psi}'} P_{\Psi}^{k_{0}/2}(x,y)d\Phi_{\widetilde{L}_{\Psi}}\lbrace r - \gamma_{\Psi}r \rbrace = 0 \]
Hence, we get
\begin{multline*}
\frac{d}{d\lambda_{\kappa}}[\cL_{p}(\mathbfcal{F}/K, \Psi, \lambda_{\kappa})]_{\kappa = k_{0}} = |\mathscr{L}_{\Psi}|^{k_{0}/2} \\
\frac{d}{d\lambda_{\kappa}}\left[\int_{\widetilde{L}_{\Psi}'} \langle \mathrm{N}_{F/\QQ}((x - \tau_{\Psi}y)(x - \tau_{\Psi}^{\theta}y))\rangle^{\frac{\lambda_{\kappa}-\lambda_{k_{0}}}{2}}P_{\Psi}^{k_{0}/2}(x,y)d\Phi_{\widetilde{L}_{\Psi}}\lbrace r - \gamma_{\Psi}r \rbrace  \right]_{\kappa = k_{0}}
\end{multline*}
By Proposition~\ref{prop:derivatives}, 
\begin{multline*}
|\mathscr{L}_{\Psi}|^{k_{0}/2}\frac{d}{d\lambda_{\kappa}}\left[\int_{\widetilde{L}_{\Psi}'} \langle \mathrm{N}_{F/\QQ}((x - \tau_{\Psi}y)(x - \tau_{\Psi}^{\theta}y))\rangle^{\frac{\lambda_{\kappa}-\lambda_{k_{0}}}{2}}P_{\Psi}^{k_{0}/2}(x,y)d\Phi_{\widetilde{L}_{\Psi}}\lbrace r - \gamma_{\Psi}r \rbrace  \right]_{\kappa = k_{0}} \\
= \frac{1}{2}|\mathscr{L}_{\Psi}|^{k_{0}/2}\frac{d}{d_{\lambda_{\kappa}}}\left(\int_{\widetilde{L}_{\Psi}'} \langle \mathrm{N}_{F/\QQ}((x - \tau_{\Psi}y)\rangle^{\lambda_{\kappa}-\lambda_{k_{0}}}P_{\Psi}^{k_{0}/2}(x,y)d\Phi_{\widetilde{L}_{\Psi}}\lbrace r - \gamma_{\Psi}r \rbrace  \right)_{\kappa = k_{0}} \\
+ \frac{1}{2}|\mathscr{L}_{\Psi}|^{k_{0}/2}\frac{d}{d_{\lambda_{\kappa}}}\left(\int_{\widetilde{L}_{\Psi}'} \langle \mathrm{N}_{F/\QQ}((x - \tau^{\theta}_{\Psi}y)\rangle^{\lambda_{\kappa}-\lambda_{k_{0}}}P_{\Psi}^{k_{0}/2}(x,y)d\Phi_{\widetilde{L}_{\Psi}}\lbrace r - \gamma_{\Psi}r \rbrace  \right)_{\kappa = k_{0}} 
\end{multline*}
Note now that $\mathscr{L}_{\Psi} = \mathscr{L}_{\tau_{\Psi}} = \mathscr{L}_{\tau^{\theta}_{\Psi}}$ and recall from Remark~\ref{rem:galoisactiononembeddings} that $(\tau_{\Psi^{\theta}},P_{\Psi^{\theta}},\gamma_{\Psi^{\theta}}) = (\tau_{\Psi}^{\theta}, -P_{\Psi}, \gamma_{\Psi}^{-1})$. By Definition~\ref{def:semidefiniteintegral} on semidefinite integrals, we have 
\begin{multline*}
\frac{d}{d\lambda_{\kappa}}[\cL_{p}(\mathbfcal{F}/K, \Psi, \lambda_{\kappa})]_{\kappa = k_{0}} = \\ \frac{1}{2}\left(\int_{r}^{\gamma_{\Psi}r}\int^{\tau_{\Psi}}P_{\Psi}(x,y)^{k_{0}/2}\omega_{\cF} + (-1)^{k_{0}/2} \int_{r}^{\gamma_{\Psi^{\theta}}^{-1}r}\int^{\tau_{\Psi^{\theta}}}P_{\Psi^{\theta}}(x,y)^{k_{0}/2}\omega_{\cF} \right)
\end{multline*}
Replacing the arbitrary base point $r \in \mathbb{P}^{1}(F)$ by $\gamma_{\Psi^{\theta}}r$, the second term on the RHS becomes 
\[ \int_{r}^{\gamma_{\Psi^{\theta}}^{-1}r}\int^{\tau_{\Psi^{\theta}}}P_{\Psi^{\theta}}(x,y)^{k_{0}/2}\omega_{\cF} = \int^{r}_{\gamma_{\Psi^{\theta}}r}\int^{\tau_{\Psi^{\theta}}}P_{\Psi^{\theta}}(x,y)^{k_{0}/2}\omega_{\cF} = - \int_{r}^{\gamma_{\Psi^{\theta}}r}\int^{\tau_{\Psi^{\theta}}}P_{\Psi^{\theta}}(x,y)^{k_{0}/2}\omega_{\cF} \]
The result now follows from Theorem~\ref{thm:padicAJimage} above.
\end{proof}
An immediate consequence of Theorem~\ref{thm:vanishingofp-adicLfn1} and Theorem~\ref{thm:padicGZformula1} above is 
\begin{corollary} \label{cor:padicGZformula2} For all unramified characters $\psi_{K} : \mathrm{Gal}(H_{K}/K) \rightarrow \mathbb{C}^{\times}$,
\begin{multline*} \frac{d^{2}}{d\lambda_{\kappa}^{2}}[L_{p}(\mathbfcal{F}/K, \psi_{K}, \lambda_{\kappa})]_{\kappa = k_{0}} = \\ 
\frac{1}{2}\left(\mathrm{N}_{F/\QQ}(\cD_{K/F})\right)^{k_{0}/2}\left(\mathrm{log}\:\Phi^{\mathrm{AJ}}(\mathrm{D}_{\psi_{K}})(\Phi_{\cF}^{\mathrm{har}})+ (-1)^{\frac{k_{0}+2}{2}}\mathrm{log}\:\Phi^{\mathrm{AJ}}(\mathrm{D}_{\psi_{K}}^{\theta})(\Phi_{\cF}^{\mathrm{har}})\right)^{2} 
\end{multline*}
\end{corollary}
In particular, for $\psi_{K} = \psi_{\mathrm{triv}}$ - the trivial character, we may further simplify the expression of Corollary~\ref{cor:padicGZformula2} above as 
\begin{corollary} \label{cor:padicGZformula3}
\begin{multline*} \frac{d^{2}}{d\lambda_{\kappa}^{2}}[L_{p}(\mathbfcal{F}/K, \psi_{\mathrm{triv}}, \lambda_{\kappa})]_{\kappa = k_{0}} = \\ 
\frac{1}{2}\left(\mathrm{N}_{F/\QQ}(\cD_{K/F})\right)^{k_{0}/2}\left(1 + (-1)^{\frac{k_{0}+2}{2}}\omega_{\cM}\right)^{2}
\left(\mathrm{log}\:\Phi^{\mathrm{AJ}}(\mathrm{D}_{\mathbbm{1}})(\Phi_{\cF}^{\mathrm{har}})\right)^{2}
\end{multline*}
\end{corollary}
\begin{proof}
For $\sigma \in \mathrm{Gal}(H_{K}/K)$, let $\sigma\Psi \in \Gamma/\mathrm{Emb}^{\mathfrak{o}_{\sigma\Psi}}(\cO,\cR)$ be a $\Gamma$-conjugacy class of oriented optimal embeddings. Since $\Phi_{\cF}^{\mathrm{har}}$ is an eigensymbol for the Atkin--Lehner involution $W_{\cM}$, we have
\[ \mathrm{log}\:\Phi^{\mathrm{AJ}}(\mathrm{D}_{(\sigma\Psi)^{\theta}})(\Phi_{\cF}^{\mathrm{har}}\mid W_{\cM}) = \mathrm{log}\:\Phi^{\mathrm{AJ}}(\mathrm{D}_{\alpha_{\cM}(\sigma\Psi)^{\theta}\alpha_{\cM}^{-1}})(\Phi_{\cF}^{\mathrm{har}}) = \omega_{\cM}\mathrm{log}\:\Phi^{\mathrm{AJ}}(\mathrm{D}_{(\sigma\Psi)^{\theta}})(\Phi_{\cF}^{\mathrm{har}}). \]
Whilst $(\sigma\Psi)^{\theta}$ doesn't have the same orientation (at $\cM$) as $\sigma\Psi$, we know that $\alpha_{\cM}(\sigma\Psi)^{\theta}\alpha_{\cM}^{-1} \in \Gamma/\mathrm{Emb}^{\mathfrak{o}_{\sigma\Psi}}(\cO,\cR)$ (See Remark~\ref{rem:atkinlehnerorientation}). By Proposition~\ref{prop:picardgrouptorsor} which exhibits the set of the $\Gamma$-conjugacy class of oriented optimal embeddings as a $\mathrm{Gal}(H_{K}/K)$-torsor, we  know that there exists $\delta_{\sigma\Psi} \in \mathrm{Gal}(H_{K}/K)$ such that
\[\alpha_{\cM}(\sigma\Psi)^{\theta}\alpha_{\cM}^{-1} = \delta_{\sigma\Psi}\sigma\Psi. \] 
Thus we have
\begin{align*}
\sum\limits_{\sigma \in \mathrm{Gal}(H_{K}/K)}\mathrm{log}\:\Phi^{\mathrm{AJ}}(\mathrm{D}_{(\sigma\Psi)^{\theta}})(\Phi_{\cF}^{\mathrm{har}}) 
 = \omega_{\cM}\sum\limits_{\sigma \in \mathrm{Gal}(H_{K}/K)}\mathrm{log}\:\Phi^{\mathrm{AJ}}(\mathrm{D}_{\delta_{\sigma\Psi}\sigma\Psi})(\Phi_{\cF}^{\mathrm{har}})
\end{align*}
Yet again using the fact that the $\mathrm{Gal}(H_{K}/K)$-action on $\Gamma/\mathrm{Emb}^{\mathfrak{o}_{\sigma\Psi}}(\cO,\cR)$ is transitive, we get
\begin{align*} \mathrm{log}\:\Phi^{\mathrm{AJ}}(\mathrm{D}_{\mathbbm{1}}^{\theta})(\Phi_{\cF}^{\mathrm{har}}) & = \sum\limits_{\sigma \in \mathrm{Gal}(H_{K}/K)}\psi_{\mathrm{triv}}^{-1}(\sigma)\mathrm{log}\:\Phi^{\mathrm{AJ}}(\mathrm{D}_{(\sigma\Psi)^{\theta}})(\Phi_{\cF}^{\mathrm{har}}) \\
& = \omega_{\cM}\sum\limits_{\sigma \in \mathrm{Gal}(H_{K}/K)}\psi_{\mathrm{triv}}^{-1}(\sigma)\mathrm{log}\:\Phi^{\mathrm{AJ}}(\mathrm{D}_{\delta_{\sigma\Psi}\sigma\Psi})(\Phi_{\cF}^{\mathrm{har}}) \\
 & = \omega_{\cM}\sum\limits_{\sigma \in \mathrm{Gal}(H_{K}/K)}\psi_{\mathrm{triv}}^{-1}(\sigma)\mathrm{log}\:\Phi^{\mathrm{AJ}}(\mathrm{D}_{\sigma\Psi})(\Phi_{\cF}^{\mathrm{har}}) \\
 & = \omega_{\cM}\mathrm{log}\:\Phi^{\mathrm{AJ}}(\mathrm{D}_{\mathbbm{1}})(\Phi_{\cF}^{\mathrm{har}})
 \end{align*}
 and the result follows from Corollary~\ref{cor:padicGZformula2} above.
\end{proof}
\begin{corollary} \label{cor:padicGZformula4}
\begin{multline*} \frac{d^{2}}{d\lambda_{\kappa}^{2}}[L_{p}(\mathbfcal{F}/K, \psi_{\mathrm{triv}}, \lambda_{\kappa})]_{\kappa = k_{0}} = \\
\begin{cases}
2\left(\mathrm{N}_{F/\QQ}(\cD_{K/F})\right)^{\frac{k_{0}}{2}}\left(\mathrm{log}\:\Phi^{\mathrm{AJ}}(\mathrm{D}_{\mathbbm{1}})(\Phi_{\cF}^{\mathrm{har}})\right)^{2} & \text{if }\omega_{\cM} = (-1)^{\frac{k_{0}+2}{2}}  \\
0 & \text{if }\omega_{\cM} = (-1)^{\frac{k_{0}}{2}}
\end{cases}
\end{multline*}
\end{corollary}

\section{Proof of the main result} \label{sec:mainresult}
We begin by noting the vanishing of the several $p$-adic $L$-functions introduced above. For $\cF \in S_{\underline{k_{0}}+2}(U_{0}(\cN))^{\mathrm{new}}$, we know that the sign of the functional equation of the base--change $L$-function $L(\cF/K, s)$ is $-1$ by the (\textbf{SH--Hyp}). In particular the central critical $L$-value $L(\cF/K, k_{0}/2+1)$ vanishes to odd order. The classical Artin formalism (\ref{eqn:classicalartinformalism}) above;
\[ L(\cF/K, s) = L(\cF, s)L(\cF, \epsilon_{K/F}, s)  \]
along with the Heegner hypothesis (\textbf{Heeg--Hyp}) shows that the sign of the functional equation of $L(f/F, s) = L(\cF, s)$ is $-1$, which forces an odd order of vanishing of the central critical $L$-value $L(\cF, k_{0}/2+1)$, whilst that of $L(f/F, \epsilon_{K/F}, s) = L(\cF, \epsilon_{K/F}, s)$ is $+1$ (See \cite[Theorem 1.1]{Pal21}). We shall assume that $L(\cF, \epsilon_{K/F}, k_{0}/2+1) \neq 0$. 
\paragraph*{}
By Remark~\ref{rem:Up-eigenvalue}, $\omega_{\p} = 1$ and hence by \cite[Theorem 9.3]{BW19}, we know that the $p$-adic $L$-function $L_{p}(\cF, s)$ has a trivial zero at $s = k_{0}/2 + 1$. Similarly, the $p$-adic $L$-function $L_{p}(\cF, \epsilon_{K/F}, s)$ doesn't have a trivial zero at $s = k_{0}/2 + 1$ since $\epsilon_{K/F}(p) = -1$ (recall that $p$ is inert in $K$ by (\textbf{SH--Hyp})). In particular, we may summarize that
\[ \mathrm{ord}_{s = k_{0}/2+1}L_{p}(\cF, s) \geq 2\]
from which we conclude, by Theorem~\ref{thm:bianchiinterpolation} above, that
\begin{equation} \label{eqn:orderofvanishingoftwovariablepadicLfn}
\mathrm{ord}_{\lambda_{\kappa} = \lambda_{k_{0}}} L_{p}(\mathbfcal{F}, \lambda_{\kappa}) \geq 2.
\end{equation}

We know, a priori, that the quantity
\begin{equation} \label{eqn:sF} S_{\cF} \defeq 2\frac{L^{\mathrm{alg}}(\cF, \epsilon_{K/F}, k_{0}/2 + 1)}{\mathrm{N}_{F/\QQ}(\cD_{K/F})}  \in \QQ(\cF)^{\times}. 
\end{equation}
We shall work under the following assumption for the rest of the paper which is consistent with the Birch and Swinnerton--Dyer conjecture (See  \cite[Remark 6.6]{Mok11} for instance) :-
\begin{assumption} \label{ass:sFisasquare} We shall assume that the quantity $S_{\cF}$ is a square, i.e.
\[ S_{\cF} \in \left(\QQ(\cF)^{\times}\right)^{2} \]
\end{assumption}
We can now compare the $\p$-adic Abel--Jacobi image of Stark--Heegner cycles introduced in \S\ref{subsec:starkheegnercycles} with that of the Heegner cycles that appear in \S\ref{subsec:Heegnercycles}
\begin{theorem} \label{thm:comparingAJimages}
Suppose that $\omega_{\cM} = (-1)^{\frac{k_{0}+2}{2}}$. Then, under Assumption~\ref{ass:sFisasquare}, there exists $\mathcal{Y} \in \mathrm{CH}^{k_{0}/2+1}\left( \cM_{k_{0}} \otimes F \right) \subset \mathrm{CH}^{k_{0}/2+1}\left( \cM_{k_{0}} \otimes K \right)$ and $s_{\cF} \in \QQ(\cF)^{\times}$ such that
\[ \mathrm{log}\:\Phi^{\mathrm{AJ}}(\mathrm{D}_{\mathbbm{1}})(\Phi_{\cF}^{\mathrm{har}}) = s_{\cF}\cdot \mathrm{log\;cl}_{f,L}(\mathcal{Y})(f^{\mathrm{rig}}). \]
\end{theorem}
\begin{proof}
By Corollary~\ref{cor:factorizationofpadicLfns} above, we have
\[ (D_{K})^{\lambda_{\kappa}/2}L_{p}(\mathbfcal{F}/K, \psi_{K},\lambda_{\kappa}) = \eta.L_{p}(\mathbfcal{F}, \chi_{1},\lambda_{\kappa}).L_{p}(\mathbfcal{F}, \chi_{2}, \lambda_{\kappa}) \]
Further, by Theorem~\ref{thm:vanishingofp-adicLfn1} and (\ref{eqn:orderofvanishingoftwovariablepadicLfn}) (See also Theorem~\ref{thm:bianchiinterpolation}), we have
\begin{equation}
\frac{d^{2}}{d\lambda_{\kappa}^{2}}\left[L_{p}(\mathbfcal{F}/K, \lambda_{\kappa})\right]_{\lambda_\kappa = \lambda_{k_{0}}} = \frac{\eta}{D_{K}^{k_{0}/2}}\frac{d^{2}}{d\lambda_{\kappa}^{2}}\left[L_{p}(\mathbfcal{F}, \lambda_{\kappa}))\right]_{\lambda_\kappa = \lambda_{k_{0}}} L_{p}(\mathbfcal{F}, \epsilon_{K/F}, \lambda_{k_{0}}) 
\end{equation}
By (\ref{eqn:bianchiinterpolation1}), we know that 
\begin{align*} 
L_{p}(\mathbfcal{F}, \epsilon_{K/F}, \lambda_{k_{0}}) & = \left( 1 - \epsilon_{K/F}(p)\frac{\mathrm{N}_{F/\QQ}(p)^{k_{0}/2}}{a_{\p}(\cF_{k_{0}})}\right)\cdot L^{\mathrm{alg}}(\cF, \epsilon_{K/F}, k_{0}/2 + 1) \\
& = 2L^{\mathrm{alg}}(\cF, \epsilon_{K/F}, k_{0}/2 + 1).
\end{align*}
By Corollary~\ref{cor:padicGZformula4} \& Theorem~\ref{thm:HeegnercyclespadicAJ} above and Remark~\ref{rem:archimedeanconstants} in Appendix~\ref{appendix1}, we know that
\begin{equation} \label{eqn:squaresofAJmaps}
\left(\mathrm{log}\:\Phi^{\mathrm{AJ}}(D_{\mathbbm{1}})(\Phi_{\cF}^{\mathrm{har}})\right)^{2} =  \frac{D_{F}^{k_{0}}u_{K}^{2}T(16\pi)^{2}}{D_{K}^{k_{0}}u_{F}^{2}16\pi^{4}} S_{\cF} \cdot \left(\mathrm{log\;cl}_{f,L}(\mathcal{Y})(f^{\mathrm{rig}})\right)^{2}
\end{equation}
for $\mathcal{Y} \in \mathrm{CH}^{k_{0}/2+1}\left( \cM_{k_{0}} \otimes F \right) \subset \mathrm{CH}^{k_{0}/2+1}\left( \cM_{k_{0}} \otimes K \right)$. Note that $T = \left(\phi,\phi\right)/\left(\phi_{\pi},\phi_{\pi}\right)$ is a square since $\phi$ is a translate of $\phi_{\pi}$ (See Appendix~\ref{appendix1} for more details). The result now follows, under Assumption~\ref{ass:sFisasquare}, upon extracting square-roots on both sides of (\ref{eqn:squaresofAJmaps}).
\end{proof}
Recall that since $\cF$ is the base--change $f/F$, at the level of Galois representations, we have
\begin{equation} \label{eqn:galoisreps}
V_{p}(\cF) = V_{p}(f)|_{G_{F}} 
\end{equation}
Let $\mathbb{D}_{\cF} \defeq \mathbb{D}_{\mathrm{st}}(V_{p}(\cF))$ (resp. $\mathbb{D}_{f} \defeq \mathbb{D}_{\mathrm{st}}(V_{p}(f))$ be Fontaine's semistable Dieudonn\'{e} module attached to the Galois representation $V_{p}(\cF)$ (resp. $V_{p}(f)$). By (\ref{eqn:galoisreps}) above, we have an identification
\[ \mathbb{D}_{\cF} \cong \mathbb{D}_{f} \otimes_{\QQ_{p}} F_{\p} \]
Further by \cite[Theorem 4.5]{VW19}, we have an isomorphism of $(\varphi, N)$-modules over $F_{\p}$ with coefficients in $L$
\[ \bigoplus\limits_{\sigma} \mathbf{D}^{\sigma}_{\cF, L} = \mathbf{D}_{\cF} \overset{\varphi}{\cong} \mathbb{D}_{\cF} = \bigoplus\limits_{\sigma} \mathbb{D}_{\cF, L}^{\sigma} \]
where $\mathbb{D}_{\cF, L}^{\sigma} \defeq \mathbb{D}_{\cF} \otimes_{F_{\p} \otimes L, \sigma} L$, which induces an identification of the tangent spaces
\[ \frac{\mathbf{D}_{\cF, L}^{\sigma}}{\mathrm{Fil}^{\frac{k_{0}+2}{2}}(\mathbf{D}_{\cF, L}^{\sigma})} \cong \frac{\mathbb{D}_{\cF, L}^{\sigma}}{\mathrm{Fil}^{\frac{k_{0}+2}{2}}(\mathbb{D}_{\cF, L}^{\sigma})} \cong \frac{\mathbb{D}_{f, L}}{\mathrm{Fil}^{\frac{k_{0}+2}{2}}(\mathbb{D}_{f, L})} \]
for each $\sigma : F_{\p} \hookrightarrow L$. We fix an isomorphism 
\[ \mathbf{MS}_{\Gamma}(L)_{(\mathcal{F})}^{\vee} \overset{\alpha}{\cong} M_{k_{0}+2}(\Gamma', L)_{(f^{\mathrm{rig}})}^{\vee} \]
defined as follows 
\begin{align} \label{eqn:identificationofdifferenttangentspaces}
\alpha : \mathbf{MS}_{\Gamma}(L)_{(\mathcal{F})}^{\vee} \overset{(\mathrm{Pr}^{\sigma})^{-1}}{\cong} & \frac{\mathbf{D}_{\mathcal{F},L}^{\sigma}}{\mathrm{Fil}^{\frac{k_{0}+2}{2}}(\mathbf{D}^{\sigma}_{\mathcal{F},L})} \overset{\varphi}{\cong} \frac{\mathbb{D}^{\sigma}_{\mathcal{F},L}}{\mathrm{Fil}^{\frac{k_{0}+2}{2}}(\mathbb{D}^{\sigma}_{\mathcal{F},L})}\\
 \cong & \frac{\mathbb{D}_{f,L}}{\mathrm{Fil}^{\frac{k_{0}+2}{2}}(\mathbb{D}_{f,L})} \overset{\mathrm{exp}_{\mathrm{BK}}}{\cong} \mathrm{H}^{1}_{\mathrm{st}}(L, V_{p}(f)(k_{0}/2+1)) \nonumber \\
\overset{\mathrm{IS}}{\cong} & M_{k_{0}+2}(X, L)_{(f^{\mathrm{JL}})}^{\vee} \cong M_{k_{0}+2}(\Gamma', L)_{(f^{\mathrm{rig}})}^{\vee} \nonumber
\end{align}
for either choice of an embedding $\sigma:F_{\p} \hookrightarrow L$ (See Remark~\ref{rem:choiceofembedding}). In particular, we have a commutative diagram
\begin{equation} \label{eqn:commutativediagram2}
\begin{tikzcd}[row sep=large, column sep=large]
(\Delta_0 \otimes \textup{Div}(\uhp_{\pri}^{\textup{ur}}) \otimes V_{k_{0},k_{0}})_{\Gamma}  \arrow[r, "\Phi^{\mathrm{AJ}}"] \arrow[d, equal] & \frac{\mathbf{D}^{\sigma}_{\mathcal{F},L}}{\mathrm{Fil}^{\frac{k_{0}+2}{2}}(\mathbf{D}^{\sigma}_{\mathcal{F},L})}  \arrow[r, "\mathrm{exp}_{\mathrm{BK}}\circ \varphi"] \arrow[d, "\mathrm{Pr^{\sigma}}"] 
& \mathrm{H}^{1}_{\mathrm{st}}(L, V_{p}(f)(k_{0}/2+1)) \arrow[d, "\mathrm{IS}" ] \\
(\Delta_0 \otimes \textup{Div}(\uhp_{\pri}^{\textup{ur}}) \otimes V_{k_{0},k_{0}})_{\Gamma} \arrow[r, "\mathrm{log}\;\Phi^{\mathrm{AJ}}"] & \mathbf{MS}_{\Gamma}(L)_{(\mathcal{F})}^{\vee}  \arrow[r, "\alpha"]
& M_{k_{0}+2}(\Gamma', L)_{(f^{\mathrm{rig}})}^{\vee} 
\end{tikzcd}
\end{equation}
\begin{remark} \label{rem:isomorphismofphiNmodules}
Note that the isomorphism of \cite[Theorem 4.5]{VW19}, $\mathbf{D}_{\cF} \overset{\varphi}{\cong} \mathbb{D}_{\cF}$ is conditional on \cite[Conjecture 4.2]{VW19} in addition to the semistability of the local Galois representation $V_{p}(\cF)|_{G_{F_{\p}}}$. However, in the base--change scenario both these conditions are satisfied, making the isomorphism $\varphi$ unconditional. See \cite[Lemma 4.4]{VW19} for more details. 
\end{remark}
We can now prove our main theorem, Theorem~\ref{thm:maintheoremintroduction} from the Introduction, that sheds evidence towards the global rationality conjectures formulated in \cite[Section 6.2]{VW19} :-
\paragraph*{}
By Theorem~\ref{thm:comparingAJimages} above, we have
\begin{equation} \label{eqn:maintheoremeqn1}
\mathrm{log}\:\Phi^{\mathrm{AJ}}(\mathrm{D}_{\mathbbm{1}})(\Phi_{\cF}^{\mathrm{har}}) = \mathrm{log\;cl}_{f,L}(s_{\cF}\cdot\mathcal{Y})(f^\mathrm{rig}),
\end{equation}
and hence
\begin{equation} \label{eqn:maintheoremeqn2}
\alpha\left( \mathrm{log}\:\Phi^{\mathrm{AJ}}(\mathrm{D}_{\mathbbm{1}})\right) = \mathrm{log\;cl}_{f,L}(s_{\cF}\mathcal{Y}) = \mathrm{IS}\left( \mathrm{cl}_{f,L}(s_{\cF}\mathcal{Y}) \right) 
\end{equation}
where the last equality follows from (\ref{ISabeljacobi2}). The commutative diagram (\ref{eqn:commutativediagram2}) implies that
\[ \alpha\left(\mathrm{log}\:\Phi^{\mathrm{AJ}}(\mathrm{D}_{\mathbbm{1}})\right) = \alpha\left(\mathrm{Pr}^{\sigma}\left(\Phi^{\mathrm{AJ}}(\mathrm{D}_{\mathbbm{1}})\right)\right) = \mathrm{IS}\left(\mathrm{exp}_{\mathrm{BK}}\circ\varphi\left(\Phi^{\mathrm{AJ}}(\mathrm{D}_{\mathbbm{1}})\right)\right). \]
In particular,  
\[\mathrm{IS}\left(\mathrm{exp}_{\mathrm{BK}}\circ\varphi\left(\Phi^{\mathrm{AJ}}(\mathrm{D}_{\mathbbm{1}})\right)\right) = \mathrm{IS}\left(\mathrm{cl}_{f,L}(s_{\cF}\mathcal{Y})\right) \] 
from which we conclude that
\[ \mathrm{exp}_{\mathrm{BK}}\circ\varphi\left(\Phi^{\mathrm{AJ}}(\mathrm{D}_{\mathbbm{1}})\right) = \mathrm{cl}_{f,L}(s_{\cF}\mathcal{Y}) \in \mathrm{H}^{1}_{\mathrm{st}}(L, V_{p}(f)(k_{0}/2+1)). \]
Let $\mathcal{S}_{K} \in \mathrm{Sel}_{\mathrm{st}}(K, V_{p}(f)(k_{0}/2+1))$ denote the image of the global cycle $s_{\cF}\mathcal{Y}$ under
\[ \mathrm{cl}_{f, K}:\mathrm{CH}^{k_{0}/2+1}\left( \cM_{k_{0}} \otimes K \right) \rightarrow \mathrm{Sel}_{\mathrm{st}}(K, V_{p}(f)(k_{0}/2+1)). \] 
Then,
\[\mathrm{exp}_{\mathrm{BK}}\circ\varphi\left(\Phi^{\mathrm{AJ}}(\mathrm{D}_{\mathbbm{1}})\right) =\mathrm{res}_{\p}\left(\mathcal{S}_{K}\right) \in \mathrm{H}^{1}_{\mathrm{st}}(L, V_{p}(f)(k_{0}/2+1)).\]
Theorem~\ref{thm:maintheoremintroduction} now follows since we have an identification of the Bloch--Kato Selmer groups owing to (\ref{eqn:galoisreps}) above
\[ \mathrm{Sel}_{\mathrm{st}}(K, V_{p}(\cF)(k_{0}/2+1)) = \mathrm{Sel}_{\mathrm{st}}(K, V_{p}(f)(k_{0}/2+1)). \]
\begin{remark} \label{rem:heckeaction}
Let $\mathbb{T}_{(\cF)}$ denote the $\cF$-isotypic component of the usual Hecke algebra $\mathbb{T}$ acting on $S_{\underline{k_{0}}+2}(U_{0}(\cN))^{\mathrm{new}}$. Then, via the isomorphism $\mathbb{T}_{(\cF)} \cong \QQ(\cF)$, we may regard $s_{\cF} \in \QQ(\cF)^{\times}$ as a Hecke operator in $\mathbb{T}_{(\cF)}$ acting on the Chow groups. 
\end{remark}
\section{Concluding Remarks} \label{sec:conclusion}
\begin{itemize}
\item[(1)] Conjectures of Calegari--Mazur (\cite[Conjecture 1.3]{CM09}) and Barrera Salazar--Williams (\cite[Conjecture 5.13]{BW20a}) predict that the only cuspidal $p$-adic families of Bianchi eigenforms come from (twisted) base--change or CM families over $\QQ$. Since the ideas explored in this article heavily rely on Hida/Coleman families of Bianchi eigenforms, the crucial assumption that the Bianchi eigenform $\cF$ is the base--change to $F$ of a classical cuspidal eigenform $f$ is indispensable. In fact, the reader will realize that no genuine Stark--Heegner cycles are constructed in this article. Similar to \cite{BD09} and \cite{Sev12}, in scenarios where the theory of Heegner cycles overlaps with that of their Stark--Heegner counterparts such as base--change,  we show that the Stark--Heegner cycles can be expressed in terms of Heegner cycles. 
\item[(2)] It would be evident to the reader that we have restricted ourselves to the \emph{analytic rank one} setting in this article. In a forthcoming work with Lennart Gehrmann, we will consider \emph{Plectic Stark--Heegner cycles} (See \cite{FG21} for example) that would account for higher orders of vanishing of the $L$-series $L(\cF/K, s)$ at the central critical point.
\item[(3)] It would be interesting to give some computational evidence when the Bianchi eigenform $\cF$ corresponds to the quadratic base--change of an elliptic curve $E_{/F}$ as in \S\ref{subsec:literature}. We hope to get back to this in the future. 
\end{itemize}
\appendix
\section{A Central $L$-value formula} \label{appendix1}
The goal of this appendix is to prove Theorem~\ref{thm:popageneralization}. In particular, we have to show that, for each classical weight $\lambda_{k} \in U$, 
\[
\mathbb{L}(\cF_{k}^{\#}, \psi_{K}) = L^{\mathrm{alg}}(\cF_{k}^{\#}/K, \psi_{K}, k/2 + 1) \in E_{k}(\psi_{K})
\]
where 
\begin{equation} \label{eqn:popaperiod} \mathbb{L}(\cF_{k}^{\#}, \psi_{K}) = \Bigg( \sum\limits_{\sigma \in \mathrm{Gal}(H_{K}/K)} \psi_{K}^{-1}(\sigma)\phi_{k}^{\#}\lbrace \tau - \gamma_{\sigma\Psi}\tau \rbrace \Big( (P_{\sigma\Psi}(x,y))^{k/2} \Big) \Bigg)^{2}.  
\end{equation}
\subsection{Classical formulation} \label{subsec:classicalformulation}
Recall that $\psi_{K} : \mathrm{Gal}(H_{K}/K) \rightarrow \mathbb{C}^{\times}$ is an unramified character and that 
\[\phi_{k}^{\#} \in \tupH^{0}(\Gamma_{0}(\cM),\Delta(V_{k,k}(E_{k})^{\vee}))\]
is the Bianchi modular symbol attached to the normalised newform $\cF_{k}^{\#} \in S_{\underline{k}+2}(U_{0}(\cM))^{\new}.$ Note that we have a natural injective map
\[  \tupH^{0}(\Gamma_{0}(\cM),\Delta(V_{k,k}(E_{k})^{\vee})) \hookrightarrow \tupH^{1}(\Gamma_{0}(\cM), V_{k,k}(E_{k})^{\vee}) \] 
and hence we may consider $\phi_{k}^{\#}$ as a cohomology class in $\tupH^{1}(\Gamma_{0}(\cM), V_{k,k}(E_{k})^{\vee})$. 

For an optimal embedding $\Psi \in \mathrm{Emb}(\cO_{K},\cR)$, let $P_{\Psi}(x,y) \in (V_{2,2})^{\Gamma_{\Psi}}$ be as in Section~\ref{subsec:starkheegnercycles}.  By Dirichlet's Unit theorem, we have that $\cO_{K}^{\times}/\lbrace\mathrm{torsion}\rbrace \cong \ZZ. $
In particular,  $\tupH_{1}(\Gamma_{\Psi},\ZZ) \cong \ZZ$ where $\Gamma_{\Psi}$ is the cyclic group generated by $\Psi(u)$ for $u$ a fundamental unit of $\cO_{K}^{\times}$. We fix a generator $\eta \in \tupH_{1}(\Gamma_{\Psi},\ZZ)$.  Consider the cap product
\[ C_{\Psi} := \eta \cap (P_{\Psi}(x,y))^{k/2} \in \tupH_{1}(\Gamma_{\Psi},\ZZ) \times \tupH^{0}(\Gamma_{\Psi}, V_{k,k}) \cong \tupH_{1}(\Gamma_{\Psi},V_{k,k})\]
Note that $\Gamma_{\Psi} \subseteq \Gamma_{0}(\cM)$. We set \[C_{[\Psi]} \defeq \mathrm{corres}^{\Gamma_{0}(\cM)}_{\Gamma_{\Psi}}[C_{\Psi}] \in \tupH_{1}(\Gamma_{0}(\cM), V_{k,k}).\] For $\psi_{K} : \Gal(H_{K}/K) \rightarrow \mathbb{C}^{\times}$ as above, we define the $\psi_{K}$-twisted cycle
\[ C^{\psi_{K}}_{[\Psi]} \defeq \sum\limits_{\sigma \in \mathrm{Cl}(K)} \psi_{K}^{-1}(\sigma)C_{\sigma\cdot[\Psi]} \in \tupH_{1}(\Gamma_{0}(\cM), V_{k,k}).  \]
Under the pairing given by cap product, 
\begin{equation} \label{eqn:capproductclassical} \langle\cdot,\cdot\rangle : \tupH_{1}(\Gamma_{0}(\cM), V_{k,k}) \times \tupH^{1}(\Gamma_{0}(\cM),V_{k,k}(E_{k})^{\vee}) \rightarrow E_{k}(\psi_{K}), 
\end{equation}
we have 
\[ \mathbb{L}(\cF_{k}^{\#}, \psi_{K}) = \langle C^{\psi_{K}}_{[\Psi]},\phi_{k}^{\#} \rangle^{2}.\] 
\subsection{Waldspurger formulas in higher cohomology} \label{subsec:waldspurger} In this section,  we relate the cap product $\langle C^{\psi_{K}}_{[\Psi]},\phi_{k}^{\#} \rangle$ with the Waldspurger period integral $\cP(\psi_{K},\phi_{k}^{\#})$ considered by Santiago Molina in \cite[Theorem 4.6]{Mol22}. Since the formulation of \cite{Mol22} is of an adelic nature, we introduce some notation.  Let $B \defeq M_{2}(F)$ be the split quaternion algebra over $F$. We may view the newform $\cF_{k}^{\#}$ as an automorphic newform of $B^{\times}$ and let $\pi_{k}$ be the automorphic representation it generates.  Note that $\pi_{k}$ has trivial central character. Let $G$ (resp. $T$) be the algebraic group associated to $B^{\times}/F^{\times}$ (resp. to $K^{\times}/F^{\times}$). Via the embedding $\Psi$, we have an inclusion of algebraic groups $T \subseteq G$. 
\subsubsection{Fundamental classes associated to tori} \label{subsubsec:fundamentalclasses} We briefly recall the construction of fundamental classes in homology following the exposition in \cite[Section 2.1]{Mol22}. Note that by Dirichlet's Unit theorem, the $\ZZ$-rank of $T(\cO_{F})$ is $1$ and we fix a generator 
\[ \xi \in \tupH_{1}(T(\cO_{F}),\ZZ) \cong \ZZ \]
Let us fix the compact subgroup $U \defeq T(\cO_{F} \otimes \hat{\ZZ}) \subset T(\AA_{F}^{\infty})$.  Note that $T(\cO_{F}) = T(F) \cap U$. Further, let us fix a fundamental domain $\widetilde{\cF} \subset T(\AA_{F}^{\infty})$ for the action of $T(F)/T(\cO_{F})$ on $T(\AA_{F}^{\infty})/U$. Then the set of continuous functions $C(\widetilde{\cF},\ZZ)$ has an action of $T(\cO_{F})$ (Note that $\widetilde{\cF}$ is $U$-invariant).  In particular, the characteristic function $\mathbbm{1}_{\widetilde{\cF}} \in \tupH^{0}(T(\cO_{F}),C(\widetilde{\cF},\ZZ))$ is $T(\cO_{F})$-invariant.
\begin{definition} \label{def:fundamentalclass}
The fundamental class associated to the tori $T \subset G$ is defined as the cap product
\[ \widetilde{\eta} \defeq \xi \cap \mathbbm{1}_{\widetilde{\cF}} \in \tupH_{1}(T(\cO_{F}),C(\widetilde{\cF},\ZZ)). \] 

\end{definition}
Further, by Shapiro's lemma
\begin{equation} \label{eqn:Uinvariant} 
\tupH_{1}(T(\cO_{F}),C(\widetilde{\cF},\ZZ)) \cong \tupH_{1}(T(F),C_{c}(T(\AA_{F}^{\infty})/U,\ZZ)) 
\end{equation}   
from which we conclude that the class $\widetilde{\eta}$ is $U$-invariant. Here $C_{c}(T(\AA_{F}^{\infty})/U,\ZZ))$ is the set of $\ZZ$-valued continuous functions on $T(\AA_{F}^{\infty})/U$ with compact support.
 
On the other hand, let $C^{\emptyset}(T(\AA_{F}),\ZZ)$ denote the set of locally constant functions on $T(\AA_{F})$ and $C_{c}^{\emptyset}(T(\AA_{F}),\ZZ)$ be the subset of functions in $C^{\emptyset}(T(\AA_{F}),\ZZ)$ that are compactly supported when restricted to $T(\AA_{F}^{\infty})$. By \cite[Lemma 2.4]{Mol22}, there is an isomorphism of $T(F)$-modules
\[ \mathrm{Ind}_{T(\cO_{F})}^{T(F)}(C(\widetilde{\cF},\ZZ)) \cong C_{c}^{\emptyset}(T(\AA_{F}),\ZZ)\]
and hence by Shapiro's Lemma, we have an isomorphism
\[ \tupH_{1}(T(F),C_{c}^{\emptyset}(T(\AA_{F}),\ZZ)) \cong \tupH_{1}(T(\cO_{F}),C(\widetilde{\cF},\ZZ))\]
This way we may regard the fundamental class $\widetilde{\eta}$ as a class in $\tupH_{1}(T(F),C_{c}^{\emptyset}(T(\AA_{F}),\ZZ))$.
\subsubsection{Cohomology of arithmetic groups and Eichler--Shimura morphism} \label{subsubsec:cohomologyofarithmeticgroups} Let $\cA^{\infty}(V_{k,k}^{\vee})$ be the set of functions $\phi : G(\AA_{F}^{\infty})/U_{0}(\cM) \longrightarrow V_{k,k}^{\vee}$ with a natural $G(F)$-action defined by 
\[\gamma.\phi(g) \defeq \phi(\gamma^{-1}g). \]
Since the class number of $F$ is one, the double quotient space $G(F)\backslash G(\AA_{F}^{\infty})/U_{0}(\cM)$ is trivial and $\Gamma_{0}(\cM) = G(F) \cap U_{0}(\cM)$. Shapiro's lemma then induces
\begin{equation} \label{eqn:cohomology}  \tupH^{1}(G(F), \cA^{\infty}(V_{k,k}^{\vee})) \cong \tupH^{1}(\Gamma_{0}(\cM), V_{k,k}^{\vee}).
\end{equation}  In particular, by (\ref{eqn:cohomology}), we may consider $\phi_{k}^{\#} \in \tupH^{1}(G(F), \cA^{\infty}(V_{k,k}^{\vee}))$.  

Let 
\[ C^{(k,k)}(T(\AA_{F}),\mathbb{C}) \defeq C^{0}(T(\AA_{F}),\mathbb{C}) \otimes V_{k,k} \]
be the space of locally polynomial functions in $T(\AA_{F})$ considered by Molina in \cite[Section 4.2]{Mol22} which has a natural $T(F)$-action.  Now, via the Artin reciprocity map we may view the unramified character $\psi_{K} : \Gal(H_{K}/K) \rightarrow \mathbb{C}^{\times}$ as a locally constant character $T(\AA_{F}) \rightarrow \mathbb{C}^{\times}$ that is $T(F)$-invariant. Similarly the element $P_{\Psi}(x,y)^{k/2} \in V_{k,k}$ is $T(F)$-invariant.  In particular, we have 
\[ \psi_{K} \otimes P_{\Psi}(x,y)^{k/2} \in \tupH^{0}(T(F), C^{(k,k)}(T(\AA_{F}),\mathbb{C})).\]

We recall the $T(F)$-equivariant pairing of \cite[Section 4.3]{Mol22} (See Remark 4.2 in particular) 
\begin{align*} \label{eqn:Santiagopairing}
\varphi : C^{(k,k)}(T(\AA_{F}),\mathbb{C}) \otimes \cA^{\infty}(V_{k,k}^{\vee}) & \longrightarrow C^{\emptyset}(T(\AA_{F}),\mathbb{C}) \\
\varphi((f \otimes P) \otimes \phi)(z,t) & \defeq  f(z,t)\cdot \phi(t)(P) 
\end{align*}
for all $z \in T(F_{\infty})$ and $t \in T(\AA_{F}^{\infty})$. The natural pairing $\langle \cdot, \cdot \rangle_{T} : C^{\emptyset}(T(\AA_{F}),\mathbb{C}) \times C_{c}^{\emptyset}(T(\AA_{F}),\mathbb{Z}) \rightarrow \mathbb{C}$ given by the Haar measure on $T(\AA_{F})$ induces a cap product 
\begin{equation} \label{eqn:capproductwaldspurger}
\langle \cdot,\cdot \rangle : \tupH^{1}(T(F), C^{\emptyset}(T(\AA_{F}),\mathbb{C})) \times \tupH_{1}(T(F),C_{c}^{\emptyset}(T(\AA_{F}),\mathbb{Z})) \longrightarrow{\mathbb{C}}
\end{equation}
We set \[\cP(\psi_{K},\phi_{k}^{\#}) \defeq \varphi((\psi_{K} \otimes P_{\psi}(x,y)) \otimes \phi_{k}^{\#}) \cap \widetilde{\eta}\]
under the cap product of (\ref{eqn:capproductwaldspurger}).
\begin{proposition} \label{prop:comparingfundamentalclasses}
With notation as above, we have an equality of cap products, i.e.
\[ \langle C_{[\Psi]}^{\psi_{K}},\phi_{k}^{\#} \rangle = \cP(\psi_{K},\phi_{k}^{\#}) \]
\end{proposition}
\begin{proof}
Note that it suffices to show that the two fundamental homology classes viz. $\eta$ and $\widetilde{\eta}$ coincide.  Recall that $\eta \in \tupH_{1}(\Gamma_{\Psi},\ZZ).$ Since $\Gamma_{\Psi} = T(F) \cap U$, where $U = T(\cO_{F} \otimes \hat{\ZZ})$, we have
\[ C_{c}^{\emptyset}(T(\AA_{F})/U,\ZZ) = \mathrm{Ind}_{\Gamma_{\Psi}}^{T(F)}\ZZ. \]
Again, by Shapiro's Lemma, we may hence consider the class $\eta \in \tupH_{1}(T(F), C_{c}^{\emptyset}(T(\AA_{F})/U,\ZZ))$.  Now, since $\widetilde{\eta} \in \tupH_{1}(T(F), C_{c}^{\emptyset}(T(\AA_{F}),\ZZ))$ is $U$-invariant (See (\ref{eqn:Uinvariant}) above), we conclude that the two fundamental homology classes $\eta$ and $\widetilde{\eta}$ coincide upto renormalizing.  In particular, it follows that the two pairings defined via cap product in (\ref{eqn:capproductclassical}) and (\ref{eqn:capproductwaldspurger}) respectively are the same.
\end{proof}
\subsubsection{A Waldspurger type formula} \label{subsubsec:Waldspurger}
We can now prove Theorem~\ref{thm:popageneralization} in the main text which relates the cap products considered above to the central $L$-value of the quadratic base change of Bianchi newforms.
\begin{theorem} \label{thm:popageneralizationappendix} Let $\psi_{K}: \mathrm{Gal}(H_{K}/K) \rightarrow \mathbb{C}^{\times}$ be an unramified character as above. Then
\[
\mathbb{L}(\cF_{k}^{\#}, \psi_{K}) =  L^{\mathrm{alg}}(\cF_{k}^{\#}/K, \psi_{K}, k/2 + 1) \in E_{k}(\psi_{K})
\]
where \begin{equation} \label{eqn:appendixbasechangeperiods}
L^{\mathrm{alg}}(\cF_{k}^{\#}/K, \psi_{K}, k/2 + 1) \defeq \frac{Tu_{K}^{2}}{(\Omega_{k}^{\#})^{2}} \times \frac{\prod_{\nu \in \Sigma_{\infty}^{F}} C'_{\nu}(K,\pi_{k},\psi_{K})}{\sqrt{\mathrm{N}_{F/\QQ}(\cD_{K/F})}}\cdot\frac{((k/2)!)^{4}}{(2\pi)^{2k+4}}L(\cF_{k}^{\#}/K, \psi_{K}, k/2 + 1) 
\end{equation}
where $u_{K} \defeq \left[ \mu(\cO_{K}):\mu(\cO_{F}) \right]$ and $T$ and $C'_{\nu}(K,\pi_{k},\psi_{K})$ are some explicit constants.
\end{theorem}
\begin{proof}
The main ingredient in the proof is Molina's Waldspurger formula in higher cohomology. More precisely, it is shown in \cite[Theorem 4.6]{Mol22} that
\[ \cP(\psi_{K},\phi_{k}^{\#})^{2} = \frac{u_{K}^{2}}{(\Omega_{k}^{\#})^{2}} \cdot L(1/2, \pi_{k} \times \pi_{\psi_{K}})\cdot\prod_{\nu \nmid \infty} \alpha_{\pi_{k,\nu},\psi_{K,\nu}}(\phi_{k}^{\#}) \]
where $\pi_{k}$ is the automorphic representation of $G(\AA_{F})$ that the newform $\cF_{k}^{\#} \in S_{\underline{k}+2}(U_{0}(\cM))$ generates, $\Omega_{k}^{\#} \in \mathbb{C}^{\times}$ is the complex period of (\ref{eqn:period}) and $\alpha_{\pi_{k,\nu},\psi_{K,\nu}}(\phi_{k}^{\#})$ are explicit local factors (almost all one). Since $\phi_{k}^{\#}$ is the modular symbol associated to the newform $\cF_{k}^{\#}$, we may use the explicit Waldspurger formula of \cite{MW09} given in terms of period integrals of Gross--Prasad test vectors $\phi \in \pi_{k}$. In particular, by \cite[Theorem 4.2]{MW09}, we get that
\begin{align*}
\cP(\psi_{K},\phi_{k}^{\#})^{2} & =  \frac{u_{K}^{2}}{(\Omega_{k}^{\#})^{2}} \times \lvert \int \phi(t)\psi_{K}^{-1}(t)dt \rvert^{2} \\
& = \frac{u_{K}^{2}}{(\Omega_{k}^{\#})^{2}} \times \frac{L^{S_{2}(\pi_{k},K)}(1/2,\pi_{k} \times \pi_{\psi_{K}})\left(\phi,\phi\right)}{\left(\phi_{\pi},\phi_{\pi}\right)} \times \frac{L_{S(\psi_{K})}(1,\epsilon_{K/F})^{2}}{\sqrt{\mathrm{N}_{F/\QQ}(\cD_{K/F})}} \\ 
& \times \frac{L_{S_{1}(\pi_{k},K)}(1,\epsilon_{K/F})}{L_{S_{1}(\pi_{k},K)}(1,1_{F})} \times \prod\limits_{\nu \in S_{2}(\pi_{k},K)} C'(\pi_{k,\nu}) \times \prod\limits_{\nu \in \Sigma_{\infty}^{F}} C'_{\nu}(K,\pi_{k},\psi_{K}).
\end{align*}
where
\begin{itemize}
\item $\phi_{\pi}$ is a new vector for $\pi_{k}$,
\item $S_{1}(\pi,K) = \lbrace\mbox{places of }F\mbox{ where }\pi_{k} \mbox{ ramifies but } K \mbox{ doesn't.}\rbrace$
\item $S_{2}(\pi,K) = \lbrace\mbox{places of }F\mbox{ where both }\pi_{k} \mbox{ and } K \mbox{ ramify.}\rbrace$  
\item $S(\psi_{K}) =  \lbrace\mbox{places of }F\mbox{ above which } \psi_{K} \mbox{ ramifies.}\rbrace$
\item $C'(\pi_{k,\nu})$ and $C'_{\nu}(K,\pi_{k},\psi_{K})$ are certain non-archimedean and archimedean constants defined in Section 4.2.1 and 4.2.2 of \cite{MW09} respectively.
\end{itemize} 
In our setting, note that $S(\psi_{K})$ and $S_{2}(\pi, K)$ are both empty while $S_{1}(\pi,K) = \lbrace \mathfrak{l} \mid \cM \rbrace$. In particular, all places in $S_{1}(\pi,K)$ split in $K$ under \textbf{(SH-Hyp)}. Further, the Gross--Prasad test vector $\phi$ can be chosen to be a translate of the new vector $\phi_{\pi}$ so that $\left( \phi, \phi \right) = T\left(\phi_{\pi},\phi_{\pi}\right)$ for some scalar $T$. We can simplify the above equation as
\begin{equation} \label{eqn:waldspurgerfinal1}
\mathbb{L}(\cF_{k}^{\#}, \psi_{K}) = \frac{Tu_{K}^{2}}{(\Omega_{k}^{\#})^{2}} \times \frac{\prod_{\nu \in \Sigma_{\infty}^{F}} C'_{\nu}(K,\pi_{k},\psi_{K})}{\sqrt{\mathrm{N}_{F/\QQ}(\cD_{K/F})}} \times L(1/2,\pi_{k} \times \pi_{\psi_{K}}). 
\end{equation}
Using the relation
\[ L(s,\pi_{k} \times \pi_{\psi_{K}}) = \left(\frac{\Gamma(s + (k+1)/2)}{(2\pi)^{s+(k+1)/2}}\right)^{4} L(\cF_{k}^{\#}/K, \psi_{K},k/2+1), \]
we get that 
\begin{align*}
\mathbb{L}(\cF_{k}^{\#}, \psi_{K}) & = \frac{Tu_{K}^{2}}{(\Omega_{k}^{\#})^{2}} \times \frac{\prod_{\nu \in \Sigma_{\infty}^{F}} C'_{\nu}(K,\pi_{k},\psi_{K})}{\sqrt{\mathrm{N}_{F/\QQ}(\cD_{K/F})}}\cdot\frac{((k/2)!)^{4}}{(2\pi)^{2k+4}} \times L(\cF_{k}^{\#}/K, \psi_{K},k/2+1)  \\
& = L^{\mathrm{alg}}(\cF_{k}^{\#}/K, \psi_{K},k/2+1) \in E_{k}(\psi_{K}).
\end{align*}
\end{proof}
\begin{remark} \label{rem:archimedeanconstants}
When $\psi_{K,\nu}$ is trivial for each archimedean place $\nu$ of $F$, then \cite[Section 4.2.2]{MW09} calculates $C'_{\nu}(K,\pi_{k},\psi_{K}) = 16\pi$. In particular, for the trivial character, we have 
\begin{equation} \label{eqn:trivialcharacter}
L^{\mathrm{alg}}(\cF_{k}^{\#}/K, k/2+1) = \frac{Tu_{K}^{2}}{(\Omega_{k}^{\#})^{2}} \times \frac{(16\pi)^{2}}{\sqrt{\mathrm{N}_{F/\QQ}(\cD_{K/F})}}\cdot\frac{((k/2)!)^{4}}{(2\pi)^{2k+4}} \times L(\cF_{k}^{\#}/K, k/2+1)
\end{equation}
\end{remark}
\begin{remark} \label{rem:Santiagonormalization}
Note that $P_{\Psi}(x,y)^{k/2} \in V_{k,k}$ that appears in the definition of the period integral $\cP(\psi_{K},\phi_{k}^{\#})$ corresponds to the dual element $\mu_{\underline{0}} \in V_{k,k}^{\vee}$, for $\underline{0} = (0,0)$, in \cite{Mol22} via the isomorphism (cf. \cite[Section 3.1]{Mol22})
\[ V_{k,k}^{\vee} \cong V_{k,k},\quad\quad\mu \mapsto \mu((Xy - Yx)^{k}). \]
\end{remark}

\bibliographystyle{alpha}
\bibliography{references}

\begin{thebibliography}{LMH20}

\bibitem[AS00]{AS00}
A.~Ash and G.~Stevens.
\newblock {$p$}-adic deformations of cohomology classes of subgroups of {${\rm
  GL}(n,{\mathbb{Z}})$}: The non-ordinary case.
\newblock 2000.
\newblock Preprint.

\bibitem[AS08]{AS08}
A.~Ash and G.~Stevens.
\newblock $p$-adic deformations of arithmetic cohomology.
\newblock 2008.
\newblock Preprint.

\bibitem[BC91]{BC91}
J.-F. Boutot and H.~Carayol.
\newblock Uniformisation {$p$}-adique des courbes de {S}himura: les
  th\'{e}or\`emes de {C}erednik et de {D}rinfeld.
\newblock Number 196-197, pages 7, 45--158 (1992). 1991.
\newblock Courbes modulaires et courbes de Shimura (Orsay, 1987/1988).

\bibitem[BD07]{BD07}
Massimo Bertolini and Henri Darmon.
\newblock {H}ida families and rational points on elliptic curves.
\newblock {\em Invent. Math.}, 168(2):371--431, 2007.

\bibitem[BD09]{BD09}
Massimo Bertolini and Henri Darmon.
\newblock The rationality of {S}tark--{H}eegner points over genus fields of
  real quadratic fields.
\newblock {\em Ann. Math.}, 170(1):343--370, 2009.

\bibitem[BDI10]{BDI10}
Massimo Bertolini, Henri Darmon, and Adrian Iovita.
\newblock Families of automorphic forms on definite quaternion algebras and
  {T}eitelbaum's conjecture.
\newblock {\em Ast\'{e}risque}, 331:29--64, 2010.

\bibitem[Bel12]{Bel12}
Jo\"{e}l Bella\"{\i}che.
\newblock Critical {$p$}-adic {$L$}-functions.
\newblock {\em Invent. Math.}, 189(1):1--60, 2012.

\bibitem[BK90]{BK07}
Spencer Bloch and Kazuya Kato.
\newblock {$L$}-functions and {T}amagawa numbers of motives.
\newblock In {\em The Grothendieck Festschrift}, pages 333--400. Progress in
  Mathematics 89, 1990.

\bibitem[BL11]{BL11}
Baskar Balasubramanyam and Matteo Longo.
\newblock {$\Lambda$}-adic modular symbols over totally real fields.
\newblock {\em Comment. Math. Helv.}, 86(4):841--865, 2011.

\bibitem[BSW19]{BW19}
Daniel Barrera~Salazar and Chris Williams.
\newblock Exceptional zeros and $\mathcal{L}$-invariants of {B}ianchi modular
  forms.
\newblock {\em Trans. Amer. Math. Soc.}, 372(1):1--34, 2019.

\bibitem[BSW21]{BW20a}
Daniel Barrera~Salazar and Chris Williams.
\newblock Families of {B}ianchi modular symbols: critical base-change
  {$p$}-adic {$L$}-functions and {$p$}-adic {A}rtin formalism.
\newblock {\em Selecta Math. (N.S.)}, 27(5):Paper No. 82, 45, 2021.

\bibitem[CM09]{CM09}
Frank Calegari and Barry Mazur.
\newblock Nearly ordinary {G}alois deformations over arbitrary number fields.
\newblock {\em J. Inst. Math. Jussieu}, 8(1):99--177, 2009.

\bibitem[Cre13]{Cre81}
John.~E. Cremona.
\newblock {\em Modular {S}ymbols}.
\newblock PhD thesis, University of Oxford, 2013.

\bibitem[CST14]{CST14}
Li~Cai, Jie Shu, and Ye~Tian.
\newblock Explicit {G}ross-{Z}agier and {W}aldspurger formulae.
\newblock {\em Algebra Number Theory}, 8(10):2523--2572, 2014.

\bibitem[Dar01]{Dar01}
Henri Darmon.
\newblock Integration on $\mathcal{H}_p\times\mathcal{H}$ and arithmetic
  applications.
\newblock {\em Ann. of Math.}, 154:589--639, 2001.

\bibitem[FG21]{FG21}
Michele Fornea and Lennart Gehrmann.
\newblock Plectic {S}tark--{H}eegner points, 2021.
\newblock (Preprint).

\bibitem[FH95]{FH95}
Solomon Friedberg and Jeffrey Hoffstein.
\newblock Nonvanishing theorems for automorphic {$L$}-functions on {${\rm
  GL}(2)$}.
\newblock {\em Ann. of Math. (2)}, 142(2):385--423, 1995.

\bibitem[FMP17]{FMP17}
Daniel File, Kimball Martin, and Ameya Pitale.
\newblock Test vectors and central {$L$}-values for {${\rm GL}(2)$}.
\newblock {\em Algebra Number Theory}, 11(2):253--318, 2017.

\bibitem[GS93]{GS93}
Ralph Greenberg and Glenn Stevens.
\newblock $p$-adic {$L$}-functions and $p$-adic periods of modular forms.
\newblock {\em Invent. Math.}, 111:407 -- 447, 1993.

\bibitem[GSS16]{GSS16}
Matthew Greenberg, Marco~Adamo Seveso, and Shahab Shahabi.
\newblock Modular {$p$}-adic {$L$}-functions attached to real quadratic fields
  and arithmetic applications.
\newblock {\em J. Reine Angew. Math.}, 721:167--231, 2016.

\bibitem[GZ86]{GZ86}
Benedict Gross and Don Zagier.
\newblock Heegner points and derivatives of {$L$}-series.
\newblock {\em Invent. Math.}, 84:225--320, 1986.

\bibitem[Han17]{Han17}
David Hansen.
\newblock Universal eigenvarieties, trianguline {G}alois representations and
  $p$-adic {L}anglands functoriality.
\newblock {\em J. Reine. Angew. Math.}, 730:1--64, 2017.

\bibitem[Hid94]{Hid94}
Haruzo Hida.
\newblock On the critical values of {$L$}-functions of {GL}(2) and
  {GL}(2)$\times${GL}(2).
\newblock {\em Duke Math.}, 74:432--528, 1994.

\bibitem[IS03]{IS03}
Adrian Iovita and Michael Spie\ss.
\newblock Derivatives of {$p$}-adic {$L$}-functions, {H}eegner cycles and
  monodromy modules attached to modular forms.
\newblock {\em Invent. Math.}, 154(2):333--384, 2003.

\bibitem[Kit94]{Kit94}
Koji Kitagawa.
\newblock On standard {$p$}-adic {$L$}-functions of families of elliptic cusp
  forms.
\newblock In {\em {$p$}-adic monodromy and the {B}irch and {S}winnerton-{D}yer
  conjecture ({B}oston, {MA}, 1991)}, volume 165 of {\em Contemp. Math.}, pages
  81--110. Amer. Math. Soc., Providence, RI, 1994.

\bibitem[Kol88]{Kol88}
Victor Kolyvagin.
\newblock Finiteness of {$E(\Q)$} and {$\Sha(E/\Q)$} for a subclass of {W}eil
  curves.
\newblock {\em Izv. Akad. Nauk. USSR ser. Matem.}, 52, 1988.
\newblock In Russian.

\bibitem[Lin05]{Ling05}
Mark.~P. Lingham.
\newblock {\em Modular forms and elliptic curves over imaginary quadratic
  fields.}
\newblock PhD thesis, University of Nottingham, 2005.

\bibitem[LMH20]{LMH20}
Matteo Longo, Kimball Martin, and Yan Hu.
\newblock Rationality of {D}armon points over genus fields of non-maximal
  orders.
\newblock {\em Ann. Math. Qu\'{e}.}, 44(1):173--195, 2020.

\bibitem[LV14]{LV14}
Matteo {Longo} and Stefano {Vigni}.
\newblock {The rationality of quaternionic Darmon points over genus fields of
  real quadratic fields}.
\newblock {\em {Int. Math. Res. Not.}}, 2014(13):3632--3691, 2014.

\bibitem[Mok11]{Mok11}
Chung~Pang Mok.
\newblock Heegner points and {$p$}-adic {$L$}-functions for elliptic curves
  over certain totally real fields.
\newblock {\em Comment. Math. Helv.}, 86(4):867--945, 2011.

\bibitem[Mol22]{Mol22}
Santiago Molina.
\newblock Waldspurger formulas in higher cohomology, 2022.

\bibitem[MW09]{MW09}
Kimball Martin and David Whitehouse.
\newblock Central {$L$}-values and toric periods for {${\rm GL}(2)$}.
\newblock {\em Int. Math. Res. Not. IMRN}, 1:Art. ID rnn127, 141--191, 2009.

\bibitem[Pal23]{Pal21}
Luis~Santiago Palacios.
\newblock Functional equation of the p-adic l-function of bianchi modular
  forms.
\newblock {\em Journal of Number Theory}, 242:725--753, 2023.

\bibitem[Pop06]{Pop06}
Alexandru~A. Popa.
\newblock Central values of {R}ankin {$L$}-series over real quadratic fields.
\newblock {\em Compos. Math.}, 142(4):811--866, 2006.

\bibitem[PS11]{PS11}
Robert Pollack and Glenn Stevens.
\newblock Overconvergent modular symbols and $p$-adic {$L$}-functions.
\newblock {\em Ann. Sci. \'{E}cole Norm. Sup.}, 44(1):1--42, 2011.

\bibitem[PS13]{PS12}
Robert Pollack and Glenn Stevens.
\newblock Critical slope $p$-adic {$L$}-functions.
\newblock {\em J. Lond. Math. Soc.}, 87(2):428--452, 2013.

\bibitem[RS12]{RS12}
Victor Rotger and Marco Seveso.
\newblock $\mathscr{L}$--invariants and {D}armon cycles attached to modular
  forms.
\newblock {\em J. Euro. Math. Soc.}, 14 (6):1955--1999, 2012.

\bibitem[Ser80]{Ser80}
Jean-Pierre. Serre.
\newblock {\em Trees}.
\newblock Springer-Verlag, 1980.

\bibitem[Sev12]{Sev12}
Marco Seveso.
\newblock {$p$}-adic {$L$}-functions and the rationality of {D}armon cycles.
\newblock {\em Canad. J. Math.}, 64(5):1122--1181, 2012.

\bibitem[Sev14]{Sev14}
Marco~Adamo Seveso.
\newblock Heegner cycles and derivatives of {$p$}-adic {$L$}-functions.
\newblock {\em J. Reine Angew. Math.}, 686:111--148, 2014.

\bibitem[Tri06]{Tri06}
Mak Trifkovi\'{c}.
\newblock Stark-{H}eegner points on elliptic curves defined over imaginary
  quadratic fields.
\newblock {\em Duke Math.}, 135, no. 3:415--453, 2006.

\bibitem[Urb95]{Urb95}
Eric Urban.
\newblock Formes automorphes cuspidales pour $gl_2$ sur un corps quadratique
  imaginaire. {Valeurs} sp\'eciales de fonctions $l$ et congruences.
\newblock {\em Compositio Mathematica}, 99(3):283--324, 1995.

\bibitem[VW21]{VW19}
Guhan Venkat and Chris Williams.
\newblock Stark–{H}eegner cycles attached to {B}ianchi modular forms.
\newblock {\em Journal of the London Mathematical Society}, 104(1):394--422,
  2021.

\bibitem[Wil17]{Wil17}
Chris Williams.
\newblock {$P$}-adic {$L$}-functions of {B}ianchi modular forms.
\newblock {\em Proc. Lond. Math. Soc.}, 114 (4):614 -- 656, 2017.

\end{thebibliography}
\end{document}